\documentclass[reqno,10pt]{amsart}

\usepackage{pdfsync}
\usepackage{ifthen}
\usepackage{amsmath,paralist}
\usepackage{amsthm}
\usepackage{graphicx,eucal}
\usepackage{graphpap,latexsym,epsf,bm}
\usepackage{color}
\usepackage{hyperref}
\usepackage{amssymb,mathrsfs,enumerate}
\usepackage[normalem]{ulem}
\usepackage{ifthen}

\numberwithin{equation}{section}
\numberwithin{figure}{section}

\newcommand{\R}{\mathbb{R}}
\newcommand{\N}{\mathbb{N}}

\mathchardef\emptyset="001F

\newtheorem{theorem}{Theorem}[section]

\newtheorem{lemma}[theorem]{Lemma}
\newtheorem{remark}[theorem]{Remark}

\newtheorem{definition}[theorem]{Definition}
\newtheorem{proposition}[theorem]{Proposition}

\newtheorem{notation}[theorem]{Notation}
\newtheorem{corollary}[theorem]{Corollary}

\numberwithin{equation}{section}

\newcommand{\up}{\uparrow}

\newcommand{\down}{\downarrow}
\newcommand{\weaksto}{\rightharpoonup^*}

\newcommand{\QED}{\mbox{}\hfill\rule{5pt}{5pt}\medskip\par}

 \def\calE{{\mathcal E}} \def\calF{{\mathcal F}}
  \def\calI{{\mathcal I}}
\def\calJ{{\mathcal J}}  
  
\def\calP{{\mathcal P}}  \def\calR{{\mathcal R}}
  
 \def\calW{{\mathcal W}} 
 
  \def\rmC{{\mathrm C}}
\def\rmD{{\mathrm D}}  
  
\def\rmJ{{\mathrm J}}

\def\dd{\;\!\mathrm{d}} 

\newcommand{\pairing}[4]{ \sideset{_{ #1 }}{_{ #2 }}  {\mathop{\langle #3 , #4
\rangle}}}

\newcommand{\eps}{\varepsilon}
\newcommand{\teta}{\vartheta}
\newcommand{\foraa}{\text{for a.a.\ }}

\newcommand{\Xs}{X}
 \newcommand{\AC}{\mathrm{AC}}
   \newcommand{\BV}{\mathrm{BV}}
   \newcommand{\BD}{\mathrm{BD}}
   \newcommand{\M}{\mathrm{M}}
   \newcommand{\spbv}[4]{\BV_{#1}([#2,#3];#4)}
\newcommand{\mdn}{\mathsf{d}}

\newcommand{\mds}[3]{\mathsf{d}_{#1}(#2,#3)}
\newcommand{\mdsn}[1]{\mathsf{d}_{#1}}
\newcommand{\mdss}[4]{\tilde{\mathsf{d}}^{#1}_{#2}(#3,#4)}
\newcommand{\ene}[2]{\mathcal{E}(#1,#2)}
\newcommand{\rene}[2]{\mathcal{I}(#1,#2)}
\newcommand{\renen}{\calI}
\newcommand{\enet}[3]{\mathcal{E}(#1,#2,#3)}
\newcommand{\enetk}[4]{\mathcal{E}_{#1}(#2,#3,#4)}

\newcommand{\pert}[2]{\mathcal{F}(#1,#2)}
\newcommand{\pertt}[3]{\mathcal{F}(#1,#2, #3)}

\newcommand{\pertot}[2]{\mathcal{F}_0(#1,#2)}
\newcommand{\pertok}[3]{\mathcal{F}_{0,#1}(#2,#3)}
\newcommand{\pertokname}[1]{\mathcal{F}_{0,#1}}

\newcommand{\pwn}{\partial_t \calE}

\newcommand{\pwt}[3]{\partial_t \calE(#1,#2,#3)}
\newcommand{\cmdn}{\mathsf{D}}

\newcommand{\cmdsn}[1]{\mathsf{D}_{#1}}
\newcommand{\cmds}[3]{\mathsf{D}_{#1}(#2,#3)}

\newcommand{\corrn}{\delta}
\newcommand{\corrsn}[1]{\delta_{#1}}
\newcommand{\corrs}[3]{\delta_{#1}(#2,#3)}
\newcommand{\Vars}[3]{\mathrm{Var}_{#1}(#2,#3)}
\newcommand{\Vari}[4]{\mathrm{Var}_{#1}(#2,[#3,#4])}

\newcommand{\Varname}[1]{\mathrm{Var}_{#1}}

\newcommand{\lli}[2]{{#1}({#2}{-})}
\newcommand{\lri}[2]{{#1}({#2}{\pm})}
\newcommand{\rli}[2]{{#1}({#2}{+})}
\newcommand{\jump}[1]{\mathrm{J}_{#1}}

\newcommand{\Jvar}[4]{\mathrm{Jmp}_{#1}(#2;[#3,#4])}

\newcommand{\vecostname}{\mathsf{c}}
\newcommand{\vecost}[3]{\mathsf{c}(#1,#2,#3)}

\newcommand{\bvcostname}{\mathsf{v}}
\newcommand{\bvcost}[3]{\mathsf{v}(#1,#2,#3)}
\newcommand{\idelta}[4]{\Delta_{#1}(#2,#3,#4)}
\newcommand{\aidelta}[5]{\Delta_{#1}(#2,#3,#4,#5)}
\newcommand{\RIS}{(\Xs,\calE,\mdsn{\spz})}
\newcommand{\RISK}[1]{(\Xs,\calE_{#1},\mdsn{\spz})}
\newcommand{\VE}{\mathrm{VE}}

\newcommand{\stab}[1]{\mathscr{S}_{#1}}

\newcommand{\rstab}[2]{\mathscr{R}(#1,#2)}
\newcommand{\rstabname}{\mathscr{R}}
\newcommand{\rstabnamek}[1]{\mathscr{R}_{#1}}
\newcommand{\rstabt}[2]{\mathscr{R}(#1,#2)}
\newcommand{\rstabtk}[3]{\mathscr{R}_{#1}(#2, #3)}
\newcommand{\slope}[4]{|\rmD_z {#1}|(#2,#3,#4)}
\newcommand{\hole}[1]{\mathfrak{h}(#1)}
\newcommand{\tcost}[4]{\mathrm{Trc}_{#1}(#2,#3,#4)}

\newcommand{\Gap}[3]{\mathrm{GapVar}_{#1}(#2,#3)}

\newcommand{\nresc}[1]{\mathfrak{s}_{#1}}

\newcommand{\ninvresc}[1]{\mathfrak{t}_{#1}}

\newcommand{\invcur}[1]{\zeta_{#1}}
\newcommand{\invcu}{\zeta}
\newcommand{\newmresc}[1]{\mathfrak{m}_{#1}}

\newcommand{\rresc}[2]{\mathsf{\sigma}_{#1}^{#2}}
\newcommand{\siresc}[2]{\mathsf{\tau}_{#1}^{#2}}
\newcommand{\piecewiseConstant}[2]{\overline{#1}_{\kern-1pt#2}}
\newcommand{\pwc}{\piecewiseConstant}

\newcommand{\ENE}{\mathrm{E}}

\definecolor{dmagenta}{rgb}{0.8,0,0.8}

\definecolor{vgreen}{rgb}{0.1,0.5,0.2}

\newcommand{\vu}{u}
\newcommand{\vz}{z}
\newcommand{\spu}{U}
\newcommand{\spz}{Z}
\newcommand{\Spu}{U}
\newcommand{\Spz}{Z}
\newcommand{\weakto}{\rightharpoonup}
\newcommand{\wsigma}{\stackrel{\sigma}{\to}}
\newcommand{\wsigmaR}{\stackrel{\sigma_{\R}}{\to}}
\newcommand{\wsigmaz}{\stackrel{\sigma_{\spz}}{\to}}
\newcommand{\wsigmau}{\stackrel{\sigma_{\spu}}{\to}}
\newcommand{\redn}{\mathcal{I}}
\newcommand{\rednk}[1]{\redn_{#1}}
\newcommand{\red}[2]{\mathcal{I}(#1,#2)}
\newcommand{\redk}[3]{\mathcal{I}_{#1}(#2,#3)}

\newcommand{\Dir}{{\!\scriptscriptstyle\mathrm{D}}}
\newcommand{\Gdir}{\Gamma_\Dir}
\newcommand\JUMP[1]{\mathchoice
                  {\big[\hspace*{-.3em}\big[#1\big]\hspace*{-.3em}\big]}
                   {[\hspace*{-.15em}[#1]\hspace*{-.15em}]}
                   {[\![#1]\!]}
                   {[\![#1]\!]}}
\newcommand{\GC}{\Gamma_{\!\scriptscriptstyle{\rm C}}}
\newcommand{\Neu}{{\!\scriptscriptstyle\mathrm{N}}}
\newcommand{\Gneu}{\Gamma_\Neu}
\newcommand{\gdir}{\phi_\Dir}
\newcommand{\sfd}{\mdn}

\newcommand{\transp}[1]{{#1}^{\mathsf{T}}}
\newcommand{\btransp}[1]{(#1)^{\mathsf{T}}}

\newcommand{\transpi}[1]{{#1}^{-\mathsf{T}}}
\newcommand{\GLD}{\mathrm{GL}^+(d)}

\newcommand{\dom}{\mathrm{dom}}
\newcommand{\SLD}{\mathrm{SL}(d)}
\newcommand{\qf}{q_F}

\newcommand{\Kirchx}[2]{\mathbb K(#1,#2)}

\newcommand{\argmin}{\mathrm{Argmin}}
\newcommand{\domene}[1]{\mathrm{D}_{#1}}
\newcommand{\inpow}[4]{\Kirchx{#1}{\nabla{#2} (t,#3)\nabla {#3}{#4}}}
\newcommand{\ds}[3]{{#1}_{#2}^{#3}}
\newcommand{\bbC}{\mathbb{C}}
\newcommand{\bbM}{\mathbb{M}}
\newcommand{\sym}{\mathrm{sym}}
\newcommand{\dev}{\mathrm{dev}}
\newcommand{\ojt}[3]{\teta_{#1,#2}^{#3}}

\begin{document}

\title[Visco-Energetic solutions and applications]{Visco-Energetic solutions to some rate-independent systems in  damage, delamination,  and plasticity}

\date{March 09, 2018}

\author{Riccarda Rossi}
\address{DIMI, Universit\`a di
  Brescia, via Valotti 9, I--25133 Brescia, Italy.}
\email{riccarda.rossi\,@\,unibs.it}

\thanks{ R.R.\   acknowledges
  support from the institute  IMATI (CNR), Pavia.}

\begin{abstract}
This paper revolves around a  newly introduced weak solvability concept for rate-independent systems,
alternative to the notions of \emph{Energetic} ($\ENE$) and \emph{Balanced Viscosity} ($\BV$) solutions.
\emph{Visco-Energetic}  ($\VE$) solutions  have been recently obtained by passing to the time-continuous limit in a time-incremental
scheme, akin to that for Energetic solutions, but perturbed by a `viscous' correction term, as
in the case of Balanced Viscosity solutions. However, for Visco-Energetic solutions this viscous correction is
tuned by a fixed parameter. The resulting solution notion turns out to describe a kind of evolution 
 in between   Energetic and  Balanced Viscosity evolution.
\par
In this paper we aim to investigate the application of  $\VE$ solutions to the paradigmatic example of perfect plasticity, and to nonsmooth rate-independent processes in solid mechanics
such as damage and plasticity at finite strains. With the limit passage from adhesive contact to brittle delamination, we also  provide a first result
of Evolutionary Gamma-convergence for $\VE$ solutions. 
The analysis of these applications reveals the wide applicability  of this solution concept and  confirms its intermediate character.
\end{abstract}
\maketitle

\noindent 
\textbf{Keywords:} Rate-independent systems, Visco-Energetic solutions, damage, delamination,  perfect plasticity, finite-strain plasticity. 


\section{Introduction}
In this paper we  explore the application of the newly introduced  concept of \emph{Visco-Energetic} solution to a rate-independent process. 
We address
rate-independent systems in solid mechanics   that can be described in terms of two variables
$(\vu,\vz) \in \spu\times \spz$.  
Typically, $\vu$ is the displacement, or the deformation of the body, whereas $\vz$ is an internal variable specific of  the phenomenon under investigation, in accordance with the theory of \emph{generalized standard materials} by \textsc{Halphen \& Nguyen} \cite{HalNgu75MSG},
 cf.\ also the modeling approach by \textsc{M.\ Fr\'emond} \cite{Fre02}. 
In the class of systems we consider here,  $\vu$ is governed by a \emph{static} balance law (usually the Euler-Lagrange equation for the minimization of the elastic energy), whereas $\vz$ evolves rate-independently. Indeed, when the ambient spaces $\spu$ and $\spz$ have a Banach structure, the equations of interest 
\begin{subequations}
\label{grad-structure}
\begin{align} & 
\label{grad-structure-u}
\rmD_u \enet t{u(t)}{z(t)} =0 \quad \text{ in } \spu^*, \ t \in (0,T), 
\\
\label{grad-structure-z}
\partial \calR(\dot{z}(t)) + & 
\rmD_z \enet t{u(t)}{z(t)} \ni 0 \quad \text{ in } \spz^*, \ t \in (0,T), 
\end{align}
\end{subequations}
feature the derivatives w.r.t.\ $u$ and $z$ of the driving energy functional $\calE: [0,T]\times \spu \times \spz \to (-\infty,\infty]$, and the (convex analysis) subdifferential $\partial \calR: \spz \rightrightarrows \spz^*$ of a convex, $1$-positively homogeneous dissipation potential $\calR: \spz \to [0,\infty]$. 
System \eqref{grad-structure} reflects the ansatz that energy is dissipated through 
changes of the internal variable $z$ only: in particular,  the doubly nonlinear evolution inclusion  \eqref{grad-structure-z} balances the dissipative frictional forces from $\partial \calR(\dot z)$ with the restoring force $\rmD_z \enet t{u}{z}$. 
\par
System \eqref{grad-structure}  is most often only formally written:  the very first issue attached to its analysis is the quest of a proper
weak solvability notion. In fact, the energy $\calE(t,\cdot,\cdot)$  can
be nonsmooth, e.g.\ incorporating indicator terms to ensure suitable physical constraints on the variables $u$ and $z$. 
However, it is
rate-independence  that poses the most significant challenges. Since the dissipation potential $\calR$ has linear growth at infinity, one can in general expect only $\BV$-time regularity for 
 $z$. Thus $z$ may have jumps as a function of time and the pointwise derivative $\dot z$ in the subdifferential inclusion \eqref{grad-structure-z}
need not  be 
defined. This has motivated the development of various weak solution concepts for system \eqref{grad-structure}, suited to the 
poor time regularity of $z$ and, at the same time,
also able to  properly capture the behavior of the system at jumps.
 The latest of these notions, Visco-Energetic solutions, is the focus of this paper. 
Before illustrating it, let us briefly review the two other notions of \emph{Energetic} and \emph{Balanced Viscosity} solutions, with which we shall often compare Visco-Energetic solutions. We refer to \cite{Mielke-Cetraro, MieRouBOOK} for a thorough survey of all the other
weak solvability concepts advanced for rate-independent systems.
\par
From now on,  we will leave the Banach setting and simply assume that 
\begin{compactitem}
\item[-] The state spaces $\spu$ and $\spz$ are
endowed with two topologies $\sigma_\spu$ and $\sigma_\spz$; 
\item[-]
Dissipative mechanisms are mathematically modeled in terms of a 
\emph{dissipation distance} $\mdsn {\spz} $ on $\spz$ (in fact, throughout the paper \emph{extended, asymmetric  quasi-distances}
will be  considered, cf.\ the general setup
introduced in Sec.\ \ref{s:2}); 
 \item[-] The driving energy $\calE(t,\cdot)$  is a $(\sigma_\spu{\times}\sigma_\spz)$-lower semicontinuous functional. 
\end{compactitem}
Henceforth, we will  
often write $X$ in place of $\spu \times \spz $ and 
refer to the triple $\RIS$ 
as a rate-independent system.  On the one hand, this generalized setup comprises the Banach one  of \eqref{grad-structure}, where
 $\mds{\spz}{z}{z'}= \calR(z'{-}z)$.  On the other hand, working in a metric-topological setting
  is  natural in view of the  application to, e.g., fracture, where the state  space  for the crack variable only has  a topological  structure,  or finite-strain plasticity, where dissipation  is  described in terms of a Finsler-type distance reflecting the geometric nonlinearities of the model.
\subsection{Energetic, Balanced Viscosity, and Visco-Energetic solutions at a glance}
 \emph{Energetic} (often abbreviated as $\ENE$) solutions were advanced in  \cite{MieThe99MMRI,MieThe04RIHM}, cf.\ also   the parallel notion of `quasistatic
evolution' in the realm of crack propagation, dating back to \cite{DM-Toa2002}. In the context of the rate-independent system $\RIS$, they can be constructed by 
recursively solving the time-incremental minimization scheme
\begin{equation}
\label{tim-E}
\tag{$\mathrm{IM}_{\ENE}$}
(\ds u\tau n,\ds z\tau n) \in \mathrm{Argmin}_{(u,z)\in \Xs} \left(  \mds{\spz}{\ds z\tau{n-1}}{z} + \enet {\ds t\tau n}{u}z\right), \quad n =1,\ldots, N_\tau\,,
\end{equation}
where $\{ \ds t\tau n\}_{n=0}^{N_\tau}  $ is a partition of $[0,T]$ 
with fineness  $\tau = \max_{n=1,\ldots, N_\tau} (\ds t\tau n{-} \ds t\tau {n-1})$. 
Under suitable conditions on $\calE$, the piecewise constant interpolants $(\pwc Z\tau)_{\tau}$ 
of the discrete solutions $(\ds z\tau n)_{n=1}^{N_\tau}$
converge
as $\tau\down0$
 to an $\ENE$ solution of the 
rate-independent system $\RIS$, namely a 
 curve
 $z\in \mathrm{BV}_{\mdsn \spz}([0,T];\spz)$, together with 
 \begin{equation}
	\label{u-once-for-all}
	u:[0,T]\to \spu,  \text{  an (everywhere defined) measurable selection } u(t) \in \mathrm{Argmin}_{\tilde u \in \spu} \enet t{\tilde u}{z(t)},
	\end{equation} 
  fulfilling 
 \begin{compactitem}
 	\item[-] the \emph{global} stability condition 
	\begin{equation}
\label{globStab}
\tag{$\mathrm{S}$}
\enet t{u(t)}{z(t)} \leq \enet t {u'}{z'} + \mds{\spz}{z(t)}{z'} \quad \text{for all } (u',z') \in \spu \times \spz \text{ and all } t \in [0,T];
\end{equation}
 	\item[-] 
the  `$\ENE$ energy-dissipation' balance for all $t\in [0,T]$
\begin{equation}
\label{enbal-E}
\tag{$\mathrm{E}$}
\enet t{u(t)}{z(t)} + \Vari {\mdsn{\spz}}{z}0{t} = \enet 0{u(0)}{z(0)} +\int_0^t \pwt s{u(s)}{z(s)} \dd s.
\end{equation}
 	\end{compactitem}
 Due to its flexibility, 	the Energetic concept 
	has been successfully  applied  to a wide scope of  problems, see e.g.\ \cite{MieRouBOOK} for a  survey.  
	However, 
it has been observed that, because of compliance with the  	
	 \emph{global} stability  condition 
	\eqref{globStab}, 
	 	$\ENE$ solutions driven by nonconvex 
	energy functionals may have to jump `too early' and `too long', c.f., e.g.,  their characterization for $1$-dimensional rate-independent 
	systems obtained in \cite{RosSav12}. 
	This fact has motivated the introduction of an alternative 
	weak solvability concept,
	pioneered in \cite{EfeMie06RILS}.  The global
	character of \eqref{globStab} in fact stems from the global minimization problem 
	\eqref{tim-E}, whereas a scheme based on \emph{local} minimization would be preferable.
	This localization can be achieved by perturbing 	\eqref{tim-E} by a term
	that penalizes the squared distance from the previous step 
	$\ds z\tau{n-1}$, namely
	\begin{equation}
\label{tim-BV}
\tag{$\mathrm{IM}_{\BV}$}
(\ds u\tau n,\ds z\tau n) \in \mathrm{Argmin}_{(u,z)\in \Xs} \left(  \mds{\spz}{\ds z\tau{n-1}}{z} + 
\frac{ \eps}{2  \tau}  \mdss2{\spz}{\ds z\tau{n-1}}{z}
+ \enet {\ds t\tau n}{u}z\right) \quad \text{for } n=1,\ldots, N_\tau\,.
\end{equation}	
Here, the \emph{viscous correction} $ \mdss2{\spz}{\ds z\tau{n-1}}{z}
$, with $\tilde{\mathsf{d}}_{\spz}$ a  second, possibly different 
distance on $\spz$, is modulated by a parameter $\eps$, vanishing to zero with $\tau$ in  such a way  that 
$\tfrac{\eps}\tau \uparrow \infty$. 
 Under appropriate conditions on $\calE$ (cf.\ \cite{MRS12, MRS13}), the approximate solutions
$(\pwc Z \tau)_{\tau}$ originating from \eqref{tim-BV}  converge as $\tau\down0$  to a 
\emph{Balanced Viscosity}
{$\BV$} solution of  the rate-independent system $\RIS$, namely 
 a curve  $z\in \mathrm{BV}_{\mdsn{\spz}}([0,T];\spz)$, with  $u: [0,T]\to \spu$ as in \eqref{u-once-for-all},
 %
 fulfilling 
	 \begin{compactitem}
 	\item[-] the \emph{local}
	stability condition
 	\begin{equation}
 	\label{loc-stab}
 	\tag{$\mathrm{S}_{\BV}$}
 	\slope \calE t{u(t)}{z(t)} \leq 1 \quad \text{for every } t \in [0,T]\setminus \jump z,
 	\end{equation}
	where $|\rmD_z \calE|$ is the metric slope of $\calE$ w.r.t.\ $z$, i.e.\
	$
	\slope \calE tuz: = \limsup_{w\to z} \frac{(\enet tuz {-} \enet tuw)^+}{\mds{\spz} zw},
	$ and $ \jump z$ the set of jump points of $z$;
 	\item[-] 
	the `$\BV$ energy-dissipation'  balance for all $t\in [0,T]$ 
 	\begin{equation}
 	\label{bv-enbal}
 	\tag{$\mathrm{E}_{\mathrm{BV}}$}
\enet t{u(t)}{z(t)} + \Vari {\mdsn \spz,\bvcostname}{z}0{t} = \enet 0{u(0)}{z(0)}  +\int_0^t \partial_t \enet s{u(s)}{z(s)} \dd s \,.  \quad 
	\end{equation}
 	\end{compactitem} 
	In \eqref{bv-enbal} $ \Varname {\mdsn \spz,\bvcostname}$
is an \emph{augmented notion} of total variation, fulfilling
 $ \Varname {\mdsn \spz,\bvcostname} \geq   \Varname {\mdsn \spz}$
and
	measuring the energy dissipated at a  jump point  $t\in \mathrm{J}_u$ in terms of a Finsler-type cost 
	$\bvcost t{\cdot}{\cdot}$. Without entering into   details, we mention here that  $\bvcost t{\cdot}{\cdot}$
	  records the possible onset of  \emph{viscosity}, hence of  rate-dependence,
   into
  the description of the system behavior at  the jump point $t$, cf.\ also 
  \cite{MRS13} for more details. 
  Because of the \emph{local} character of the stability condition 
  \eqref{loc-stab}, $\BV$ solutions driven by nonconvex energies have mechanically feasible jumps, as shown by their characterization in  \cite{RosSav12}.
  Nonetheless, a crucial requirement underlying the
  Balanced Viscosity concept is that the energy $\calE$ complies with a chain-rule type condition. This  is ultimately related to 
 convexity/regularity properties of $\calE$  and unavoidably restricts the range of   applicability of  $\BV$ solutions.
  \par
   That is why, 
 \emph{Visco-Energetic} ($\VE$) solutions   have recently  been advanced     in  \cite{SavMin16} as a yet  alternative solvability concept for the rate-independent system $\RIS$. 
 The key idea at the core of this novel notion 
 is to broaden the 
 class of admissible
 viscous corrections of the original time-incremental scheme \eqref{tim-E}.
 The quadratic perturbation
 $\tfrac{ \eps}{2  \tau}  \mdss2{\spz}{\ds z\tau{n-1}}{z}$ in scheme \eqref{tim-BV} is in fact replaced by the term
 $\delta_{\spz}(\ds z\tau{n-1},z)$, with $\delta_\spz: \spz \times \spz \to [0,\infty]$ a  general lower semicontinuous functional.  This turns
 \eqref{tim-E} into
 \begin{equation}
 \label{tim:VE-true}
 \tag{$\mathrm{IM}_{\VE}$}
   (\ds u\tau n, \ds z \tau n  ) \in \mathrm{Argmin}_{(u,z) \in \Xs}
   \left(  \enet {\ds t\tau n}{u}{z} + \mds{\spz}{\ds z \tau{n-1}}{\vz} +\delta_{\spz}(\ds z\tau{n-1},z)
   \right), \ n=1,\ldots, N_\tau\,.
 \end{equation}
 For simplicity,
 we shall confine the exposition  in this Introduction to the simpler, but still significant, case in which
$
 \delta_\spz(z,z') = \tfrac{\mu}2 \mdss 2{\spz}z{z'} $ with $\mu>0$ a \emph{fixed} parameter and $\tilde{\mathsf{d}}_\spz$ a (possibly different)
 distance on $\spz$, 
 postponing
 the discussion of the general case to   Sec.\ \ref{s:2}.
 This choice 
  gives rise to the time-incremental minimization scheme 
 \begin{equation}
 \label{tim-VE}
  (\ds u\tau n, \ds z \tau n  ) \in \mathrm{Argmin}_{(u,z) \in \Xs}
   \left(  \enet {\ds t\tau n}{u}{z} + \mds{\spz}{\ds z \tau{n-1}}{\vz} +\frac{\mu}2 \mdss 2{\spz}{\ds z\tau{n-1}}{z}
   \right), \ n=1,\ldots, N_\tau, \ \mu>0 \text{ fixed.}
\end{equation} 
 In \cite[Thm.\ 3.9]{SavMin16} it has been shown that, under suitable conditions (cf.\ 
 Sec.\ \ref{s:2} ahead),
the discrete solutions  $(\pwc Z\tau)_{\tau}$ of
	\eqref{tim-VE}  converge, as $\tau\down 0$, to a $\VE$ solution of  $\RIS$, i.e.\ a curve
 $z\in \mathrm{BV}_{\mdsn \spz}([0,T];\spz)$, together with $u: [0,T]\to \spu$ as in \eqref{u-once-for-all},   fulfilling
 \begin{compactitem}
 	\item[-] the viscously perturbed 
	 stability condition
 	\begin{equation}
 	\label{ve-stab}
 	\tag{$\mathrm{S}_{\mathrm{VE}}$}
 	\enet t{u(t)}{z(t)} \leq \enet t{u'}{z'} + \mds {\spz}{z}{z'} + \frac\mu2 \mdss 2{\spz}{z(t)}{z'} \quad \text{ for all  } (u',z') \in \spu \times \spz \text{ and   all }  t \in [0,T]\setminus \jump z;
 	\end{equation}
 	\item[-] 
	the  `$\VE$-energy-dissipation' balance for all $t\in [0,T]$ 
 	\begin{equation}
 	\label{ve-enbal}
 	\tag{$\mathrm{E}_{\mathrm{VE}}$}
\ene t{u(t)} + \Vari {\mdsn{\spz},\vecostname}{u}0{t} = \ene 0{u(0)} +\int_0^t \partial_t \ene s{u(s)} \dd s \,. 
	\end{equation}
 	\end{compactitem}
	In \eqref{ve-enbal},  dissipation  of energy is described by  the  total variation functional $\mathrm{Var}_{\mdsn{\spz},\vecostname}$,  which differs from the `$\BV$ total variation'  $\mathrm{Var}_{\mdsn{\spz},\bvcostname}$
	in the contributions
	at  jump points. In the $\VE$-concept, the energy dissipated at jumps is in fact `measured' in terms of 
	a new cost function $\vecostname$, obtained by minimizing a suitable transition cost along curves connecting the two end-point $\lli zt$
	and $\rli zt$ of the curve $z$ at $t\in \mathrm{J}_z$,  namely
	\begin{equation}
	\label{ve-cost-intro}
	\begin{aligned}
\vecost {t}{\lli z t }{\rli z t}: = \inf \Big \{   \mathrm{Trc}_{\VE}(t;\teta, E)\, :   \   E \Subset \R, 
  \teta\in \mathrm{C}_{\sigma_\spz,\mathsf{d}_\spz}(E;\spz),  \   \teta(\inf E) = \lli z t ,  \  \teta(\sup E) = \rli z t \Big\}\,.
  \end{aligned}
  \end{equation}
  The transition cost 
    \[
  \mathrm{Trc}_{\VE}(t;\teta, E) :  = 
   \Vars {\mdsn{\spz}}\teta E + \Gap{\delta_{\spz}}\teta E  + \sum_{s\in E{\setminus}\{\sup E\}} \rstabt t{\teta(s)}
   \]
    features  \emph{(i)} the $\mdsn{\spz}$-total variation of the curve $\teta$;  \emph{(ii)}   a quantity related to the `gaps', or `holes', of the set $E$ (which is just  an  arbitrary compact subset of $\R$ and may have a more complicated structure than an interval);  \emph{(iii)}
the residual function  $\mathscr{R} : [0,T]\times \spz \to [0,\infty) $ (defined in \eqref{residual-stability-function} ahead), which
records the violation  of the $\VE$-stability condition, 
as it fulfills
\[
\rstabt tz>0 \text{ if and only if \eqref{ve-stab} does not hold.} 
\]
\par 
Visco-Energetic solutions  are \emph{in between}
Energetic and Balanced Viscosity solutions in several  respects: 
\begin{compactenum}
\item \underline{The   structure of  the solution concept}: On the one hand, the stability condition \eqref{ve-stab}, though perturbed by a viscous correction,
still retains a \emph{global} character, like for $\ENE$ solutions. This globality plays a key role in the
 argument
 for proving convergence of the discrete solutions of \eqref{tim-VE} to a $\VE$-solution. Indeed, 
 as shown in \cite{SavMin16}, 
 once 
 \eqref{ve-stab} is established for the time-continuous limit,
 it is sufficient to check the upper  estimate  $\leq$  to conclude  \eqref{ve-enbal} with an equality sign. In particular, no chain rule for $\calE$ is needed for the energy balance.  On the other hand,
 $\VE$ solutions provide a description of the system behavior at jumps comparable to that of $\BV$ solutions. Indeed,
 optimal jump transitions (i.e., transitions between the two end-points of a jump attaining the  
 $\inf$ in \eqref{ve-cost-intro}),  exist at every jump point. Moreover, they turn out to solve a  minimum problem akin  to the time-incremental minimization scheme
 \eqref{tim:VE-true}, cf.\ \eqref{minimum-jump} ahead. Similarly, optimal jump transitions for $\BV$ solutions solve a (possibly rate-dependent) evolutionary problem
  related to  the scheme  \eqref{tim-BV} they originate from. 
\item \underline{Their characterization for   $1$-dimensional rate-independent systems}: In the $1$-dimensional  setting
it was shown in \cite{Minotti17}
that $\VE$ solutions originating from scheme \eqref{tim-VE} where, in addition,  $\tilde{\mathsf{d}}_\spz = \mdsn{\spz}$,  have a behavior  
strongly dependent on the parameter $\mu>0$. If $\mu$ is above a certain threshold,
$\VE $ solutions exhibit a behavior at jumps akin to that of $\BV$ solutions, cf.\ \cite{Minotti17}. With a `small' $\mu$, the behavior is intermediate between $\ENE$ and $\BV$ solutions.
\item \underline{The singular limits $\mu\down0$ and $\mu \up \infty$}: 
in \cite{RS17},
in  a general metric-topological setting
but, again, with the special viscous correction $\delta_\spz = \tfrac{\mu}2 \mathsf{d}^2_\spz$, $\VE$ solutions have been shown  to converge to
$\ENE$ and $\BV$ solutions as $\mu\down 0$ and $\mu\up\infty$, respectively.
\item \underline{The assumptions for the existence theory}:  Loosely speaking, they turn out to be 
 weaker than for $\BV$ solutions, and 
 stronger
than for $\ENE$ solutions. Therefore, the range of applicability of $\VE$ solutions to rate-independent processes in solid mechanics
is \emph{intermediate} between  the $\ENE$ and the $\BV$ concepts.
\end{compactenum}
\subsection{Our results} 
In this paper
 we are going to demonstrate the in-between character 
 $\VE$ solutions by addressing their application to a rate-independent system for damage, and to a model for finite-strain plasticity; the highly nonlinear and nonsmooth character  of these examples
also shows the flexibility of the $\VE$ concept. 
\par
 In the  case of the damage system, 
the  existence theory for $\ENE$ solutions \cite{MiRou06, ThoMie09DNEM,Thom11QEBV} and for $\BV$ solutions \cite{KnRoZa13VVAR,KnRoZa17,Negri16} 
seems to be well established. With \underline{\bf Theorem \ref{thm:VEdam}} ahead we will 
prove the existence of $\VE$ solutions
by applying the existence result \cite[Thm.\ 3.9]{SavMin16}
to a quite general damage system.
 Our assumptions
 on the constitutive functions of the model and on the problem data
  will  \emph{(i)}
 coincide with the conditions for $\ENE$  solutions in the  case of the 
viscous correction $\delta_\spz = \tfrac{\mu}2 \mathsf{d}^2_\spz$;   \emph{(ii)} turn out to be slightly stronger than those for $\ENE$ solutions
(in particular forcing a stronger gradient regularization for the damage variable), in the case of a `nontrivial' viscous correction $\delta_\spz$ involving a 
distance \emph{different} from the dissipation distance $\mdsn{\spz}$; 
   \emph{(iii)} be definitely weaker than those for $\BV$ solutions, cf.\ also Remark \ref{rmk:inbetw-dam} ahead.
   \par
   The system for rate-independent elastoplasticity at finite strains we are going to address
   has been analyzed from the viewpoint of Energetic solutions in \cite{MaiMi08}, whereas   no result on the existence of $\BV$ solutions seems to be available up to now. In fact, the corresponding, viscously regularized system  has been only recently tackled in  \cite{MRS2018}, where an 
     existence result 
     has been obtained after  considerable regularization of the driving energy functional to ensure the validity of the chain rule. 
   In contrast, as we will see the existence of $\VE$ solutions to the rate-independent finite-strain plasticity system
     can be checked again under the same conditions as for $\ENE $ solutions in the case of a `trivial' viscous correction. In turn,  the `nontrivial' case requires stronger assumptions, cf.\ \underline{\bf Theorem \ref{thm:VEplast}} and Remark \ref{rmk:intermediate-plast} ahead.
   \par
  We are going to examine $\VE$ solutions  from yet another viewpoint, by testing them on the benchmark example of 
  the Prandtl-Reuss system for associative elastoplasticity.
  In  \underline{\bf Theorem \ref{thm:pp}} we are going to show that Visco-Energetic solutions  for that system are indeed Energetic. 
  The key point for our  argument, cf.\ Prop.\ \ref{prop:VE=E}, will be to deduce  that $\VE$ solutions 
  comply with the `Energetic' global stability condition   \eqref{globStab}. 
  Exploiting  the `global' character  of the $\VE$-stability condition,  in fact,   we will be able to prove that $\VE$ solutions
  fulfill  a characterization  of 
  \eqref{globStab}  obtained in \cite{DMDSMo06QEPL},  and ultimately relying
  on the \emph{convex} character of the perfectly plastic system.  
   \par
  Finally, we will tackle the application of $\VE$ solutions to a rate-independent system
  for \emph{brittle} delamination, which can be thought of as a model for fracture on a prescribed 
   surface. 
  Due to the highly nonconvex  and nonsmooth character of the underlying energy functional, 
  the existence results from \cite{SavMin16} do not directly apply. 
  In fact, in 
  \underline{\bf Theorem \ref{th:VE-brittle}}
   the existence of $\VE$ solutions will  be proved by passing to the limit in an approximating system  that models adhesive contact. 
  In this way we will thus provide a first result on the convergence of $\VE$ solutions for systems driven by $\Gamma$-converging energies; 
   our proof will rely on a careful  asymptotic analysis of optimal jump transitions in the adhesive-to-brittle limit passage. 
  In a future paper we plan to address  the issue of \emph{Evolutionary $\Gamma$-convergence} 
  (in the sense of \cite{Mielke-evol}) for $\VE$ solutions in a more systematic and comprehensive way.
\paragraph{\bf Plan of the paper.}  In \underline{Section \ref{s:2}} we shall revisit the theory of $\VE$ solutions from \cite{SavMin16} and  slightly adapt it to processes described in terms of  two variables $(u,z)$ (while  \cite{SavMin16} mostly focused on rate-independent systems in the single variable $z$). 
\underline{Sections \ref{s:perfect-plas}, \ref{s:3}, \ref{s:4}} will be centered  on the applications to perfect plasticity,  damage, and finite-strain plasticity, respectively. Finally, the limit passage in the $\VE$ formulation from adhesive contact to brittle delamination will be addressed in \underline{Section \ref{s:adh2bri}}. 
\begin{notation}
\label{gen-not}
\upshape
Throughout the paper,  we shall use the symbols
$c,\,c',\, C,\,C'$, etc., whose meaning may vary even within the same   line,   to denote various positive constants depending only on
known quantities.
\par
Given a topological space 
 $\mathbf{X}$, we will 
\emph{(i)} denote by $\mathrm{B}([0,T];\mathbf{X}) $ the space of \emph{everywhere} defined and measurable functions $v: [0,T] \to \mathbf{X}$; 
\emph{(ii)}  if 
$(\mathbf{X},\mathbf{d})$ is  a metric space, denote by 
$\mathrm{BV}_{\mathbf{d}}([0,T];\mathbf{X}) $ the space of \emph{everywhere} defined functions $v: [0,T] \to \mathbf{X}$ with bounded variation. 
\par
Finally,  if  $\mathbf{X}$ is  also a  normed  space,  the symbol $\overline{B}_r^{\mathbf{X}}$ will denote  the closed  ball of $\mathbf{X}$ of radius $r>0$, centered at $0$. We will frequently omit the symbol  $\mathbf{X}$ to avoid overburdening notation. 
For the same reason, we will often write  $\| \cdot \|_{\mathbf{X}}$ in place of  
 $\| \cdot \|_{\mathbf{X}^d}$, and,  in place of $\pairing{\mathbf{X}^*}{\mathbf{X}}{\cdot}{\cdot}$, we shall write $\pairing{}{\mathbf{X}}{\cdot}{\cdot}$ (or even  $\pairing{}{}{\cdot}{\cdot}$ when the duality pairing is clear
 from the context or has to be specified later) . 
 \end{notation}
\paragraph{\bf Acknowledgements.}
I am grateful to Giuseppe Savar\'e for sharing  his insight  on Visco-Energetic solutions with me  and for several fruitful discussions, and to Alexander Mielke for various  suggestions on dissipation distances in finite-strain plasticity.
\section{Setup, definition, and existence result for  Visco-Energetic solutions}
\label{s:2}
In this section we recapitulate  
the  basic assumptions and definitions underlying the notion of Visco-Energetic solutions. We draw all concepts from 
\cite{SavMin16}. There, however, the focus was on energies depending on the sole dissipative variable $z$ (which was in fact denoted 
as $u$ in  \cite{SavMin16}), and the case of functionals also depending on the variable at equilibrium $u$ was recovered through a marginal procedure, cf.\  \cite[Sec.\ 4]{SavMin16}.
Here we 
will partially revisit the presentation in \cite{SavMin16} by directly working with energy functionals depending on the \emph{two} variables $(u,z)$.
\subsection{The abstract setup for Visco-Energetic solutions}
\label{ss:2.1}
In what follows we collect the assumptions on the metric-topological setup, on the energy functional, on the dissipation (quasi-)distance, and on the viscous correction, at the core
of the existence theory for $\VE$ solutions. 
\subsubsection{\bf The metric-topological setting}
Throughout the paper we will
denote by $\sigma$ the product topology on $X = \spu \times \spz$  induced by the two topologies $\sigma_\spu$ and $\sigma_\spz$, and 
 by $\sigma_\R$ the topology induced by $\sigma$ on $[0,T]\times \Xs$.
We will  often write $(u_n,z_n) \wsigma (u,z)$ as $n\to\infty$ to signify convergence w.r.t.\ $\sigma$-topology, and  we will use an analogous notation  for  $\sigma_\R$-,
 $\sigma_\spz$-, and $\sigma_\spu$-convergence.
\par The  mechanism of energy dissipation will be described in terms of 
 an \emph{extended}, possibly \emph{asymmetric}  quasi-distance   
\begin{equation}
\label{asymm-dist-z}
\begin{gathered}
 \mdsn{\spz}: \spz \times \spz \to [0,\infty], \quad \text{l.s.c.\ on $\spz\times \spz$,   s.t. }
\begin{cases}
\mds{\spz}{\vz}{\vz} =0, 
\\
 \mds{\spz}{\vz_o}{\vz}<\infty \text{ for some reference point } \vz_o \in \spz,
\\
  \mds{\spz}{\vz}{w} \leq \mds{\spz}{\vz}{\zeta} + \mds{\spz}\zeta{w} \quad \text{for all } \vz, \, \zeta, \,  w \in \spz.
\end{cases}
\end{gathered}
\end{equation}
We say that $ W \subset Z$ is $\mdsn{\spz}$-bounded if $\sup_{w\in X}\mds{\spz}{\vz_o}{w}<\infty$, and that $\mdsn{\spz}$ separates the points of $W$ if 
\[
w,\, w' \in W, \quad \mds{\spz}{w}{w'} =0 \quad \Rightarrow \quad w=w'.
\]
\par
Our 
\underline{first condition} concerns this metric-topological setting:
\begin{description}
\item[$<T>$] We require that  
\begin{subequations}
\begin{align}
\label{top-space}
&
\text{the topological spaces }
(U,\sigma_U) \text{ and } (Z,\sigma_Z)  \text{ are   Hausdorff  and satisfy the first axiom of countability},
\\
&
\label{Souslin}
(U,\sigma_U) \text{ is a Souslin space},
\end{align}
namely the image of a Polish (i.e.\ a separable completely metrizable) space under a continuous mapping.
Furthermore, we impose that 
\begin{equation}
\label{d-z-sep}
\mdsn {\spz} \text{ separates the points of } \spz.
\end{equation}
\end{subequations}
\end{description}
\par
Let us now recall from \cite{SavMin16} the definition of \emph{$(\sigma_\spz,\mdsn{\spz})$-regulated} function, encompassing a crucial property that the Visco-Energetic solution component $z$ shall enjoy at jumps. 
\begin{definition}{\cite[Def.\ 2.3]{SavMin16}}
\label{def:sigma-d-reg}
We call a curve $z:[0,T]\to \spz$ \emph{$(\sigma_\spz,\mdsn{\spz})$}-regulated if for every $t\in [0,T]$ there exist the left- and right-limits of $z$ 
w.r.t.\ $\sigma_\spz$-topology,
i.e.\
\begin{subequations}
\label{sigma-d-reg}
\begin{equation}
\label{sigma-d-reg-a}
\lli zt = \lim_{s\up t} z(s) \quad \text{ in } (\spz,\sigma_\spz), \qquad   \rli zt = \lim_{s\down t} z(s) \quad \text{ in } (\spz,\sigma_\spz)
\end{equation}
(with the convention $\lli z0: = z(0)$ and $\rli zT: = z(T)$), 
also satisfying
\begin{equation}
\label{sigma-d-reg-b}
\begin{aligned}
& \lim_{s\up t} \mds{\spz}{z(s)}{\lli zt} =0,  && \lim_{s\down t} \mds{\spz}{\rli zt}{z(s)} =0,
\\
& \mds{\spz}{\lli z t }{z(t)} =0 \ \Rightarrow \ \lli z t =z(t), &&  \mds{\spz}{z(t)}{\rli z t } =0 \ \Rightarrow \ z(t) = \rli z t.  
\end{aligned}
\end{equation}
\end{subequations}
We denote by $\spbv {\sigma_\spz,\mdsn{\spz}}0T{\spz}$ the space of $(\sigma_\spz,\mdsn{\spz})$-regulated functions   $z$   with finite $\mdsn{\spz}$-total variation
$\Vari{\mdsn \spz}z0T$, where we define, for  a subset  $E\subset [0,T]$,
\begin{equation}
\label{def-tot-var}
\begin{aligned}
\Vars{\mdn}z{E}:  = \sup \left\{ \sum_{j=1}^M \mds{\spz}{\teta_z(t_{j-1})}{\teta_z(t_j)}\, : \ t_0<t_1<\ldots<t_M, \ \{ t_j\}_{j=0}^M \in \mathfrak{P}_f(E)\right\}
 \end{aligned}
 \end{equation}
 with $ \mathfrak{P}_f(E)$ the collection of all finite subsets of $E$. 
 \end{definition}
 \noindent
If   $(\spz,\mdsn{\spz})$ is a complete metric space, every  function $ z \in \spbv {\mdsn{\spz}}0T{\spz}$ is ($\mdsn{\spz}$-)regulated, namely at every $t\in [0,T]$ there exist the left- and right-limits of $u$ w.r.t.\ the metric $\mdsn{\spz}$.  
However, since in the present context  we are not assuming  
completeness of $(\spz,\mdsn{\spz})$,
the concept of $(\sigma_\spz,\mdsn{\spz})$-regulated function  turns out to be significant. Observe that, for every $z \in \spbv {\sigma_\spz,\mdsn{\spz}}0T{\spz}$  the jump set 
\begin{equation}
\label{jump-z}
\jump z: = \jump z^- \cup \jump z^+, \quad \text{with } \jump z^- : = \{ t \in [0,T]\, : \ \lli zt \neq z(t)\}, \qquad \jump z^+ : = \{ t \in [0,T]\, : \ z(t) \neq  \rli zt \}, 
\end{equation}
coincides with the jump set of the real monotone function $V_z: [0,T] \to \R$, $t\mapsto V_z(t):= \Vari {\mdsn{\spz}}z0t$. Therefore, $\jump z$ is at most countable. 
\par
Finally, as we will discuss at the beginning of Section \ref{ss:2.2}, the $u$-component of a Visco-Energetic solution is in principle only an element in $\mathrm{B}([0,T];\spu)$
(cf.\ Notation \ref{gen-not}). 
 However, in qualified situations (cf.\ Lemma \ref{l:u-regulated} 
 ahead) $u$ will additionally   be a 
\begin{equation}
\label{sigma-u-regulated}
\text{$\sigma_\spu$-regulated function, i.e. } \ \forall\, t \in [0,T] \quad   \exists\, \lli ut = \lim_{s\up t} u(s)  \text{ in } (\spu,\sigma_\spu), \quad   \rli ut = \lim_{s\down t} u(s)  \text{ in } (\spu,\sigma_\spu)\,.
\end{equation}

\subsubsection{\bf The energy functional}
\noindent 
We now recall the basic assumptions on the energy functional $\calE$ enucleated in \cite{SavMin16}.
In view of Proposition \ref{prop:VE=E} ahead, 
differently from \cite{SavMin16} we
choose not to encompass 
 lower semicontinuity and  compactness requirements into a unique condition. 
\paragraph{\bf Assumption $<A>$} The RIS $\RIS$ fulfills
\begin{description}
\item[$<A.1>$] \textbf{Lower semicontinuity:} 
The proper domain $\mathrm{D}(\calE(t,\cdot))$ does not depend on $t$, namely there exists
$\domene {}  \subset \Xs $  
such that  
  $ \mathrm{D}(\calE(t,\cdot))  \equiv  \domene{}$ for all $t\in [0,T]$.
In what follows, we will use the notation 
\begin{equation}
\label{notation-4-domains}
\domene u: = \pi_1(\domene{}), \qquad \domene z: = \pi_2(\domene{})
\end{equation}
with $\pi_1: \Xs \to \spu$ and  $\pi_2: \Xs \to \spz$ the  projection operators. 
We require that  
\begin{equation}
\label{pert-func}
\begin{gathered}
\text{there exists $F_0 \geq 0$ such that the perturbed functional}
\\
\mathcal{F}: [0,T]\times \Xs \to (-\infty,\infty] \qquad \pert t{(u,z)}: = \ene t{(\vu,\vz)} + \mds{\spz}{z_o}{z}+F_0
\\
\text{ fulfills }  \pert t{(u,z)} \geq 0 \quad \text{for all }  (t,(u,z))\in [0,T]\times \Xs\,,
\end{gathered} 
\end{equation}
with $z_o$   the reference point  satisfying \eqref{asymm-dist-z}. 
In what follows, with slight abuse of notation we will  write
\[
\enet t{\vu}{\vz} \text{ in place of } \ene t{(\vu,\vz)}, \text{ and analogously for $\mathcal{F}$.}
\]
We impose that 
$\calE$ is $\sigma$-l.s.c.\ on the sublevels of $\calF$.
\item[$<A.2>$] \textbf{Compactness:}  The sublevels of $\calF$
 are $\sigma_\R$-sequentially compact in $[0,T]\times \Xs$.
\item[$<A.3>$] \textbf{Power control:} The functional 
$t\mapsto \enet tuz$ is differentiable for all $(u,z)$, $\pwn: (0,T) \times  \domene{}  \to \R$
is 
 sequentially upper semicontinuous on the sublevels of $\calF$, and 
 \begin{equation}
 \label{power-control}
\begin{aligned}
\exists\,  \Lambda_P, \,  C_P>0   \ \ \forall\, (t,u,z) \in   (0,T) \times  \domene{}  \, : \quad  |\pwt tuz|\leq   \Lambda_P \left(  \pertt tuz{+} C_P \right) \,.
\end{aligned}
\end{equation}
\end{description}
\par
\begin{remark}
\label{rmk:wlogzo}
\upshape
A natural choice for the reference point $z_o$  in \eqref{asymm-dist-z} and \eqref{pert-func}
is the initial datum $z_0 \in \domene z $ for the rate-independent process. In fact, along the evolution there holds $\Vari{\mdsn {\spz}}{z}0T<\infty$,
cf.\ Remark \ref{rmk:sth-about-VE} ahead, 
 and therefore $\sup_{t\in [0,T]} \mds{\Spz}{z_0}{z(t)} \leq C <\infty$.  That is why, 
we may suppose without loss of generality that, for every $z\in \domene z$ there holds
$\mds{\Spz}{z_0}{z}<\infty$.
\par
In \cite{SavMin16} a more general version of the power-control condition was assumed, involving a  generalized   `power functional' 
$\calP :[0,T]\times \domene{} \to \R$  satisfying 
\[
\limsup_{s\down t} \frac{\enet suz -\enet tuz}{s-t} \leq \calP(t,u,z) \leq \liminf_{s\up t} \frac{\enet tuz -\enet suz}{t-s} \quad \text{for all } (t,u,z) \in [0,T]\times \domene{}, 
\]
 and in fact surrogating the   partial time derivative $\partial_t \calE$ whenever $\calE$ is not differentiable w.r.t.\ $t$. 
 This generalization was mainly motivated by the need to encompass  in the theory \emph{marginal energies}, i.e.\ functionals only depending on the dissipative variable $z$
 and obtained from energies depending on both variables $(u,z)$ via minimization w.r.t.\ $u$. 
 For simplicity, in this paper we shall not work with this power functional. 
 \par Finally, we point out that \eqref{power-control} could be weakened by allowing for a (positive) function $\Lambda_P \in L^1(0,T)$, in place of a (positive) constant $\Lambda_P$. 
\end{remark}
\noindent
A straightforward consequence of $<A.1>$ \&    $<A.2>$  is that 
\begin{equation}
\label{reduced-notempty}
\inf_{u \in \Spu} \enet tuz \neq \emptyset \qquad \text{for all } (t,z) \in [0,T]\times \domene z.
\end{equation}
In what follows, we will often work with the \emph{reduced energy} functional
\begin{equation}
\label{reduced}
\renen: [0,T] \times \Spz \to (-\infty,\infty] \qquad  \rene tz: = \begin{cases}
\inf_{u \in \Spu} \enet tuz  = \min_{u\in \Spu} \enet tuz & \text{if } (t,z) \in  [0,T]\times \domene z,
\\
\infty & \text{otherwise}.
\end{cases}
\end{equation}
\par
Combining the power-control estimate in 
\eqref{power-control} with the Gronwall Lemma, we conclude that 
\[
\pertt tuz\leq \pertt s{u}{z} \exp\left(C_P |t-s|\right) \quad \text{for all } s,\, t \in [0,T] \text{ and all } (u,z)\in \Xs. 
\]
In particular,
\begin{equation}
\label{prop-pert}
\sup_{t\in [0,T]} \pertt tuz \leq  \exp (C_P T ) \pertt 0{u}{z}  
 \quad \text{for all } (u,z) \in \Xs\,.
\end{equation}
That is why, in what follows we will direcly work  with the functional
\[
\pertot uz: = \mathcal{F} (0,u,z) \quad \text{for every } (u,z) \in \Xs.
\]
\par Finally, we highlight that 
the upper semicontinuity  of $\partial_t \calE$ required in
$<A.3>$ can be relaxed if $\mdsn{\spz} $ enjoys an additional continuity property, stated in $<A.3'>$ below. Indeed,   $<A.3'>$  can  replace  assumption
$<A.3.>$.
\begin{description}
\item[$<A.3'>$] 
$\mdsn{\spz}$ is left-continuous on the sublevels of $\calF_0$, i.e.\ for all sequences $(u_n,z_n)_n\subset \spu \times \spz$ s.t.
\begin{subequations}
\label{replac-pw}
\begin{equation}
\label{dist-continuity}
\pertot{u_n}{z_n} \leq C, \quad z_n \wsigmaz z \quad \text{there holds} \quad \mds{\spz}{z_n}{\zeta}\to \mds{\spz}{z}{\zeta} \qquad \text{for all } \zeta \in \spz,
 \end{equation}
and the map $\pwn:[0,T]\times \Xs \to \R$ satisfies \eqref{power-control} and the \emph{conditional} upper semicontinuity 
\begin{equation}
\label{usc-power}
(t_n,u_n,z_n) \wsigmaR (t,u,z), \quad \enet{t_n}{u_n}{z_n}\to \enet tuz \ \Rightarrow \ \limsup_{n\to\infty} \pwt{t_n}{u_n}{z_n} \leq \pwt tuz.
\end{equation}
\end{subequations}
\end{description}
The condition that convergence of the energies implies convergence of the powers is often required for the analysis of rate-independent systems, cf.\ \cite{MieRouBOOK}. 
For later use,  we recall here a result from  where this implication was proved in the case 
in which
 $\partial_t \calE$ is uniformly continuous on sublevels of $\calE$, namely
\begin{equation}
\label{unif-cont}
\begin{gathered}
\forall\, C>0 \text{ there exists a modulus of continuity } \omega_C: [0,T]\to [0,\infty) \text{ such that }
\\
\forall\, (u,z) \in \spu \times \spz \, : \  \pertot uz \leq C \ \Rightarrow \ |\partial_t \enet{t_1}{u}{z}{-} \partial_t \enet{t_2}{u}{z}  |
\leq \omega_C(|t_1-t_2|)  \text{ for all } t_1,\,t_2 \in [0,T]\,.
\end{gathered}
\end{equation} 
\begin{proposition}{\cite[Prop.\ 3.3]{FraMie06ERCR}}
\label{prop:GF}
Assume \eqref{unif-cont}. Then, for every $t\in [0,T]$ the following implication holds
\begin{equation}
\label{FRANCFORT}
\left(  (u_n,z_n) \wsigma (u,z) \text{ in } \Xs,  \quad 
\enet{t_n}{u_n}{z_n}\to \enet tuz \right) \ 
\Longrightarrow \ \partial_t \enet{t_n}{u_n}{z_n}\to \partial_t \enet tuz\,.
\end{equation}
\end{proposition}
\subsubsection{\bf The viscous correction of the time-incremental scheme}
We consider 
\[
\text{a lower semicontinuous map }
\corrsn{\spz}: \spz \times \spz \to [0,\infty]  \quad \text{with } \corrs{\spz}{z}{z}=0 \quad \text{for all } z \in \spz\,.
\]
We introduce the `corrected' dissipation
\[
\cmds{\spz}{z}{z'} : = \mds\spz{z}{z'} + \corrs{\spz}{z}{z'}\,.
\]
\begin{definition}
	\label{def:stability}
	Let $Q\geq 0$. We say that $(t,\vu,\vz)\in [0,T]\times 
	\Xs$
	is \emph{$(\cmdsn{\spz},Q)$-stable} if it satisfies
	\begin{equation}
	\label{Qstable}
	\enet t{\vu}{\vz} \leq \enet t{\vu'}{\vz'} + \cmds{\spz}{\vz}{\vz'}  +Q \qquad \text{for all } (\vu',\vz')\in  \Xs. 
	\end{equation}
	If $Q=0$, we will simply say that $(t,\vu,\vz)$ is $\cmdsn{\spz}$-stable. We denote by $\mathscr{S}_{\cmdsn{\spz}}$ the collection of all the $\cmdsn{\spz}$-stable points, and by
	 $\mathscr{S}_{\cmdsn{\spz}}(t)$ its section at the  process  time $t\in [0,T]$. 
\end{definition}
\noindent
In view of $<A.1>$ \&  $<A.2>$  (which guarantee \eqref{reduced-notempty}), 
the quasi-stability condition  \eqref{Qstable} is equivalent to 
\begin{equation}
\label{reduced-stability}
\red t{\vz} \leq \red t{\vz'} + \cmds{\spz}{\vz}{\vz'} +Q \quad \text{for all } \vz' \in \spz \ 
\end{equation}
involving the reduced energy $\calI$ from \eqref{reduced}. 
That is why, 
\begin{itemize}
\item[-]
in what follows we will often allow for the abuse of notation 
$
(t,z) \in \mathscr{S}_{\cmdsn{\spz}}$ (and $z \in \mathscr{S}_{\cmdsn{\spz}}(t)$), in place of 
$(t,u,z) \in \mathscr{S}_{\cmdsn{\spz}}$. 
\item[-]
 we  now  introduce the \emph{residual stability function}
  $\rstabname: [0,T]\times \spz \to \R$ directly in terms of the reduced  energy $\calI$, namely we define
 \begin{equation}
 \label{residual-stability-function}
 \begin{gathered}
 \rstabt tz:  = \sup_{z'\in \spz} \left\{ \red tz - \red t{z'} -\cmds{\spz} {z}{z'} \right\} = \red tz -\mathscr{Y}(t,z)  \quad \text{with}
\\
 \mathscr{Y}(t,z)  =  \inf_{z'\in \spz} \left( \rene t{z'}+\cmds{\spz} {z}{z'}\right). 
 \end{gathered}
\end{equation} 
\end{itemize}
Note that, 
as soon as the energy functional $\calE$ complies with  $<A.1>$ and  $<A.2>$    (and we will suppose this hereafter),
the $\inf$ in the definition of $\mathscr{Y}$ is attained, i.e.
\begin{equation}
\label{minimal-set}
M(t,z): =  \mathrm{Argmin}_{z'\in \spz}  \left( \rene t{z'}+\cmds{\spz} {z}{z'}\right) \neq \emptyset.
\end{equation}
 Observe that $\mathscr{R}$  in fact records the failure of the stability condition at a given point  $(t,z) \in [0,T]\times \spz$, since
 \begin{equation}
 \label{propR}
 \begin{aligned}
 &
 \rstabt tz \geq 0 \quad \text{for all } (t,z) \in [0,T]\times \spz, \quad \text{with}
 \\ 
 &  \rstabt tz=0 \text{ if and only if } (t,z) \in \stab {\cmdsn{\spz}}\,.
 \end{aligned}
 \end{equation}
\par
Let us now specify the compatibility properties that
 \emph{admissible} viscous corrections have to enjoy with respect to the 
 driving distance $\mdsn{\spz}$.  
\begin{description}
	\item[$<B.1>$] \textbf{$\mdsn{\spz}$-compatibility:} 
	For every $z,\, z',\, z'' \in \spz$ 
	\begin{equation}
	\label{B1}
	\mds{\spz}{z}{z'} =0 \ \Rightarrow \ \corrs{\spz}{z''}{z'} \leq \corrs{\spz}{z''}z \quad \text{and} \quad \corrs{\spz}z{z''}  \leq \corrs{\spz}{z'}{z''}\,.
	\end{equation}
		\item[$<B.2>$] \textbf{Left $\mdsn{\spz}$-continuity:} 
		For every sequence $(u_n,z_n)_n$ and every  $(u,z) \in \Xs$ we have
		\begin{equation}
		\label{B2}
		\sup_n \pertot{u_n}{z_n} <\infty, \quad 
		z_n \wsigmaz z, \quad
		\mds{\spz}{z_n}{z} \to 0 \ \Rightarrow \ \lim_{n\to\infty} \corrs{\spz}{z_n}{z}=0.
		\end{equation}
			\item[$<B.3>$] \textbf{$\cmdsn{\spz}$-stability yields local $\mdsn{\spz}$-stability:} 
			for all $(t,u,z) \in \mathscr{S}_{\cmdsn{\spz}}$ and all $M>1$ there exist $\eta>0$ and a neighborhood $I_U \times I_Z$
			of $(u,z)$ such that 
			\begin{equation}
			\label{B3}
			\begin{aligned}
			\enet s{u'}{z'} \leq \enet  suz + M \mds{\spz}{z'}{z} \quad &  \text{for all } (s,u',z') \in \mathscr{S}_{\cmdsn{\spz}} \\ & \text{ with } s \in [t-\eta,t], \ 
			\text{for all }
			(u,z)\in I_U \times I_Z \text{ with } \mds{\spz}{z'}{z} \leq \eta.
			\end{aligned}
			\end{equation}			
\end{description}
\begin{remark}
\label{rmk:MS}
\upshape
As already observed in \cite{SavMin16},
 \eqref{B3}  is in fact equivalent to  the condition 
\begin{equation}
\label{cond-w-I}
\limsup_{(s,z') \widetilde{\rightharpoonup} (t,z)}
\frac{\red s{z'} - \red tz}{\mds{\spz}{z'}{z}} \leq 1\,,
\end{equation}
involving the reduced energy $\calI$ from \eqref{reduced-stability},
where we have written
$
(s,z') \widetilde{\rightharpoonup} (t,z)  $ as a  place-holder for  $(  s \to t, \ 
z' \wsigmaz z,  \  \mds{\spz}{z}{z'}\to 0,  \ 
	(s,z') \in \mathscr{S}_{\mathsf{D}_\spz}, \ s\leq t ). $ 
In turn,
a sufficient condition for \eqref{cond-w-I} is 
\begin{equation}
\label{will-be-checked}
\limsup_{(s,z') \widetilde{\rightharpoonup} (t,z)}
\frac{\corrs{\spz}{z'}{z}}{\mds{\spz}{z'}{z}} =0 \qquad \text{for every } z \in \mathscr{S}_{\cmdsn{\spz}}(t) \text{ and all } t \in [0,T]. 
\end{equation}
In particular,
any viscous correction of the form
\begin{equation}
\label{a-funct-of-d}
\corrs{\spz}{z}{z'} = h(\mds{\spz}{z}{z'}) \qquad \text{with } h \in \mathrm{C}([0,\infty)) \text{ nondecreasing and fulfilling } \lim_{r\down 0} \frac{h(r)}{r} =0
\end{equation}
satisfies \eqref{will-be-checked} and, in fact, the whole  Assumption $<B>$. 
\end{remark}
\paragraph{\bf Closedness of the (quasi-)stable set.}
Finally, we require
\begin{description}
	\item[$<C>$]
	For every $Q\geq 0$ the $(\cmdsn{\spz},Q)$-quasistable sets have $\sigma$-closed intersections with the sublevels of $\calF_0$.
	\end{description}
It was proved in \cite[Lemma 3.11]{SavMin16} that $<C>$  holds if and only if 
a property  akin to the \emph{mutual recovery sequence} condition from \cite{MRS06} holds, namely
\begin{equation}
\label{MRS}
\begin{aligned}
&
\text{for every sequence } (t_n,z_n)_n \subset [0,T]\times Z \text{ with } 
t_n\to t, \ z_n \wsigmaz z, \ \sup_n \mds{\spz}{z_o}{z_n}<\infty
\\
& \text{ and } \lim_{n\to\infty} \red {t_n}{z_n} = \red tz +\eta, \ \eta\geq 0,\\
&
\text{there exists } z'\in M(t,z) \text{ and a sequence } (z_n')_n \text{ such that}
\\
& 
\liminf_{n\to\infty} \left( \red {t_n}{z_n'} + \cmds{\spz}{z_n}{z_n'} \right) \leq \red t{z'} + \cmds{\spz}{z}{z'} +\eta\,,
\end{aligned}
\end{equation}
(recall that $M(t,z)$ denotes  the set of minimizers associated with the functional $\mathscr{Y}$ in \eqref{residual-stability-function}). 
\subsection{Definition of  Visco-Energetic solution}
\label{ss:2.2}
 As already mentioned in the Introduction,
the concept of Visco-Energetic solution of the rate-independent system $\RIS$ 
(cf.\   Definition \ref{def:VE} ahead)
consists of the $\cmdsn{\spz}$-stability condition 
\eqref{ve-stab}
combined with the   energy-dissipation balance \eqref{ve-enbal}.
 In  \eqref{ve-enbal} the energy dissipated at jumps is measured in terms of a jump dissipation cost $\vecostname$ that keeps track of the viscous correction $\delta_\spz$. This jump dissipation is obtained by minimizing a suitable transition cost over a class of continuous curves connecting the two end-points of a jump.
In what follows, 
\begin{enumerate}
\item
Firstly, we will  specify what we mean by `end-points of a jump' of a curve $(u,z)$ enjoying the properties of a Visco-Energetic solution, viz.
\begin{equation}
\label{regularity-of-u-z}
z \in \spbv{\sigma_\spz, \mdsn{\spz}}0T{\spz} \quad \text{ and } t \mapsto u(t) \text{ is a measurable selection in } \argmin_{u\in \spu}\enet t{u}{z(t)}.
\end{equation}
Namely, for  a curve $(u,z)$ as in \eqref{regularity-of-u-z},   we 
will introduce  \emph{surrogate} left- and right-limits for $u$ at a jump point $t\in \jump z$.
\item Secondly, we will rigorously introduce the cost $\vecostname$.
\end{enumerate}
\paragraph{\bf 1. Surrogate left- and right limits of $u$:} given a curve $(u,z)$ as in \eqref{regularity-of-u-z},  we extend  $u$ in this way:
\begin{equation}
\label{def-end-points-u}
\text{at every } t \in \jump z \text{ we denote by  }  \ 
\begin{cases}  \lli u t & \text{  an element }  \argmin_{u\in \spu}\enet t{u}{\lli zt}\,,
\\
\rli u t  & \text{  an element  in }  \argmin_{u\in \spu}\enet t{u}{\rli zt}\,,
\end{cases}
\end{equation}
with the convention that $\lli u t  = \rli u t = u(t)$ if $t \notin \jump z$, such that the extended mapping, still denoted by $u$, is still measurable. 
\par
Observe that this definition is meaningful in view of
\eqref{reduced-notempty}.  The notation $\lli ut$ and $\rli u t$ is used here in an extended sense, as  the true left- and right-limits of $u$ at $t$
w.r.t.\ $\sigma_\spu$-topology
 need not exist. Nonetheless,   in Lemma \ref{l:u-regulated} ahead,   we will provide some sufficient conditions, which can be verified for  a reasonable class  of examples, ensuring that, 
 if $(u,z) $ is a Visco-Energetic solution, then
  $u$ is $\sigma_\spu$-regulated and, in that case, $\lli u t$ and $\rli ut$ defined by \eqref{def-end-points-u} are its left- and right-limits. 
\paragraph{\bf 2. The Visco-Energetic cost $\vecostname$} It involves minimization of a suitable cost functional over a class 
 of continuous curves, connecting the left- and right-limits $(\lli u t, \lli z t)$ and $(\rli u t, \rli z t)$ at a jump point $t\in \jump z$ (with $\lli ut$ and $\rli ut$
as in \eqref{def-end-points-u}). 
 Such curves  
 are in general defined on a compact subset $E\subset \R$ with  a possibly more complicated structure than  that of an interval.  To describe it, we fix some notation:
\begin{subequations}
\begin{equation}
\label{notation-setE}
E^-: = \inf E, \quad E^+: = \sup E\,.
\end{equation}
We also 
introduce
\begin{equation}
\label{holes}
\text{the collection $\hole E$ of the connected components of the  set $[E^-,E^+]\setminus E$.}
\end{equation}
\end{subequations}
 Since $[E^-,E^+]\setminus E$ is an open set, $\hole E$ consists of at most countably many open intervals, which we will often refer to as the `holes' of $E$. 
We are now in a position to introduce 
 the transition cost at the basis of the concept of Visco-Energetic solution, evaluated along curves 
 $\teta = (\teta_u,\teta_z) \in \mathrm{B}(E;\Xs)$ such that, in addition 
 \begin{equation}
 \label{Csigma-d}
 \teta_z \in \mathrm{C}_{\sigma_\spz,\mdsn{\spz}}(E;\spz): = \mathrm{C}_{\sigma_\spz}(E;\spz) \cap \mathrm{C}_{\mdsn \spz}(E;\spz).
 \end{equation}
 Here, $\mathrm{C}_{\sigma_\spz}(E;\spz) $ is  the space of functions from $E$ to $\spz$ that are continuous with respect to the $\sigma_\spz$-topology,  while $ \mathrm{C}_{\mdsn \spz}(E;\spz)$ is the space of functions  $\teta_z: E \to \spz$  satisfying the following continuity condition w.r.t.\ $\mdsn \spz$:
 \[
 \forall\, \eps>0 \ \exists\, \eta>0 \ \forall\, s_0,\, s_1 \in E \text{ with } s_0\leq s_1 \leq s_0+\eta\, : \qquad \mds{\spz}{\teta_z(s_0)}{\teta_z(s_1)} \leq \eps\, .
 \]
\begin{definition}
\label{def-VE-trans-cost}
Let $E$ be a compact subset of $\R$ and $\teta = (\teta_u,\teta_z) \in   \mathrm{B}(E;\spu) \times  \mathrm{C}_{\sigma_\spz,\mdsn{\spz}}(E;\spz) $.
For every $t\in [0,T]$ we define the \emph{transition cost function} 
\begin{equation}
\label{ve-tcost}
\tcost{\VE}{t}{\teta}{E} : = \Vars {\mdsn{\spz}}{\teta_z} E + \Gap{\delta_\spz}{\teta_z} E  + \sum_{s\in E{\setminus}\{E^+\}} \rstab t{\teta_z(s)}\,,
\end{equation}
with
\begin{enumerate}
\item $ \Vars {\mdsn{\spz}}\teta E $  the  $\mdsn{\spz}$-total variation of the curve $\teta$, cf.\ \eqref{def-tot-var};
\item $\Gap{\delta_\spz}\teta E
\colon = \sum_{I\in \hole E} \delta_\spz(\teta_z(I^-),\teta_z(I^+)) $;
\item the (possibly infinite) sum
\[
 \sum_{s\in E{\setminus}\{E^+\}} \rstab t{\teta_z(s)}: =  \begin{cases}
 \sup \{  \sum_{s\in P} \rstab t{\teta_z(s)}\, :  \ P \in \mathfrak{P}_f(E{\setminus} \{E^+\}) \} &\text{ if }  E{\setminus}\{E^+\} \neq \emptyset,
 \\
 0 &\text{ otherwise}\,.
 \end{cases}
\]
\end{enumerate}
\end{definition}
\noindent
Along with \cite{SavMin16}, we observe that, for every fixed $t\in [0,T]$ and admissible $\teta $, the transition cost fulfills the additivity property
\[
\tcost{\VE}{t}{\teta}{E \cap [a,c]} = \tcost{\VE}{t}{\teta}{E \cap [a,b]} + \tcost{\VE}{t}{\teta}{E \cap [b,c]}  \quad \text{for all } a<b<c\,.
\]
\par
We are now in a position to define the  \emph{Visco-Energetic jump dissipation cost} $\vecostname: [0,T]\times X \times X \to [0,\infty]$ 
between the two end-points of a jump  of a curve $(u,z)$ as in \eqref{regularity-of-u-z}. Namely, we set
 \begin{equation}
 \label{vecost}
 \begin{aligned}
  \vecost t{(u_-,z_-)}{(u_+,z_+)}: = \inf\{   \tcost{\VE}{t}{\teta}{E}\, :     &  \ E \Subset \R, \ \teta  = (\teta_u,\teta_z) \in 
  \mathrm{B}(E;\spu) \times   \mathrm{C}_{\sigma_\spz,\mdn_\spz}(E;\spz), 
 \\  &
   \ \teta(E^-) =(u_-,z_-), \  \teta(E^+) =(u_+,z_+) \}.
   \end{aligned}
\end{equation}
\begin{remark}
\upshape
\label{rmk:dep-only-z}
\upshape
In fact, 
for every admissible transition curve $\teta = (\teta_u,\teta_z)$
between two pairs $(u_-,z_-)$ and $(u_+,z_+)$,
all of the three contributions to the transition cost from \eqref{ve-tcost}
  only depend on the $\teta_z$-component. 
   That is why, from now on with slight abuse of notation we will simply write 
\begin{equation}
\label{caveat-on-c}
  \vecost t{z_-}{z_+} \text{ in place of }   \vecost t{(u_-,z_-)}{(u_+,z_+)}\,.
\end{equation}
Accordingly, we will introduce the concept of \emph{Optimal Jump Transition}, cf.\ \eqref{def:OJT} ahead, only in terms of the $\teta_z$-component of an admissible transition curve $\teta = (\teta_u,\teta_z)$.  
\end{remark}
\par
With the jump dissipation cost $\vecostname$ we associate the  \emph{incremental cost} $\Delta_{\vecostname}: [0,T]\times \Xs \times \Xs \to [0,\infty]$ defined
at all $ t\in[0,T] $ and $ (u_-,z_-),\,(u_+,z_+) \in \Xs$
 by
\begin{equation}
\label{Delta-e}
 \Delta_{\vecostname}(t,(u_-,z_-),(u_+,z_+)) =  \Delta_{\vecostname}(t,z_-, z_+) : = \vecost  t{z_-}{z_+} - \mds{\spz}{z_-}{z_+}
\end{equation}
(in fact, observe that $ \vecost  t{z_-}{z_+} \geq \mds{\spz}{z_-}{z_+}$, so that  $ \Delta_{\vecostname}(t,z_-,z_+)\geq 0$,  for all $t\in [0,T]$
 and 
 $z_\pm \in \spz$). 
We will also use the notation
\[
\Delta_{\vecostname}(t,z_-,z,z_+) : = \Delta_{\vecostname}(t,z_-,z) + \Delta_{\vecostname}(t,z,z_+)\,.
\]
The \emph{augmented total variation} functional induced by $\vecostname$ is  defined, along a  curve  $(u,z)\in \BV([0,T];\Xs)$, by 
 \begin{equation}
 \label{augm-tot-var}
 \Vari {\mdsn{\spz},\vecostname}{(u,z)}{t_0}{t_1} : = \Vari{\mdsn{\spz}}{z}{t_0}{t_1} +  \Jvar {\Delta_{\vecostname}}{(u,z)}{t_0}{t_1} \quad\text{for any sub-interval } [t_0,t_1]\subset [0,T],
\end{equation}
where 
 the \emph{incremental jump variation} of $(u,z) $ on   $[t_0,t_1] $  is given by 
\begin{equation}
\label{jump-Delta-e}
\begin{aligned}
 \Jvar {\Delta_{\vecostname}}{(u,z)}{t_0}{t_1} : = 
  & \idelta{\vecostname}{t_0}{z(t_0)}{\rli z{t_0}} +  \idelta{\vecostname}{t_1}{ \lli z{t_1}}{ z(t_1)} \\  & + \sum_{t\in \jump u \cap (t_0,t_1)}
 \aidelta{\vecostname}{t}{\lli z t}{z(t)}{ \rli zt}\,.
 \end{aligned}
 \end{equation}
 Ultimately, also  this jump contribution  only depends on 
  the $z$-component, namely
\[
 \Jvar {\Delta_{\vecostname}}{(u,z)}{t_0}{t_1}  =   \Jvar {\Delta_{\vecostname}}{z}{t_0}{t_1}\,.
\]
Therefore, hereafter we shall write
\[
 \Vari {\mdsn{\spz},\vecostname}{z}{t_0}{t_1} \quad \text{in place of} \quad
  \Vari {\mdsn{\spz},\vecostname}{(u,z)}{t_0}{t_1}\,.
\]
\par
 As observed in \cite{SavMin16},
 although it is not canonically induced by a distance,
  the total variation functional 
 $\mathrm{Var}_{\mdsn{\spz},\vecostname}$ still enjoys the additivity property
  \[
 \Vari {\mdsn{\spz},\vecostname}z{a}{c} =  \Vari {\mdsn{\spz},\vecostname}z{a}{b} +  \Vari {\mdsn{\spz},\vecostname}z{b}{c} \qquad \text{for all } 0\leq a \leq  b \leq c \leq T.
\]
\par
We are now in a position to define the concept of Visco-Energetic solution $(u,z)$ of the rate-independent system $(X,\calE,\mdsn{\spz})$,  featuring the 
$\cmdsn{\spz}$-stability condition, and the
 energy-dissipation balance with the total variation functional $ \Varname {\mdsn{\spz},\vecostname}$. Let us stress in advance that, since 
$ \Varname {\mdsn{\spz},\vecostname} \geq  \Varname {\mdsn{\spz}} $  only controls the $z$-component of the curve $(u,z)$, it will  be for $z$  only that we shall claim $z\in \BV_{\mdsn{\spz}} ([0,T];\spz)$ (in fact, $z\in \BV_{\sigma_\spz,\mdsn{\spz}} ([0,T];\spz)$), while for the $u$ component only measurability will be a priori  asked for.  
\begin{definition}[Visco-Energetic solution]
\label{def:VE}
A curve  $(u,z) : [0,T]\to \Xs$, 
with $u \in \mathrm{B}([0,T];\spu) $
  and $z \in \BV_{\sigma_\spz,\mdsn{\spz}}([0,T];\spz)$, is   a Visco-Energetic ($\VE$) solution of the rate-independent system $\RIS$ with the viscous correction $\delta_\spz$,
if 
it satisfies
\begin{itemize}
\item[-] the minimality condition 
\begin{equation}
\label{MINIMALITY}
u(t) \in \argmin_{u\in \spu} \enet tu{z(t)} \qquad \text{for all $t\in [0,T]$;}
\end{equation}
\item[-] the $\cmdsn{\spz}$-stability condition
\begin{equation}
\label{stab-VE}
	\tag{$\mathrm{S}_{\mathrm{VE}}$}
	\begin{aligned}
\enet t{u(t)}{z(t)} &  \leq \enet t{u'}{z'}  +
\cmds{\spz}{z(t)}{z'} \\ & =   \enet t{u'}{z'} 
+ \mds{\spz}{z(t)}{z'} + 
\corrs{\spz}{z(t)}{z'}  \quad \text{for all } (u',z') \in \Xs \text{ and all } t \in [0,T]\setminus \jump z,
\end{aligned}
\end{equation}
\item[-]  the $(\mdsn{\spz},\vecostname)$-energy-dissipation balance
\begin{equation}
\label{enbal-VE}
	\tag{$\mathrm{E}_{\mathrm{VE}}$}
\enet t{u(t)}{z(t)} + \Vari {\mdsn{\spz},\vecostname}{z}0{t} = \enet 0{u(0)}{z(0)} +\int_0^t \pwt s{u(s)}{z(s)} \dd s \quad \text{for all } t \in [0,T]\,.
\end{equation}
\end{itemize}
\end{definition}
\begin{remark}
\label{rmk:sth-about-VE}
\upshape
From the energy-dissipation balance, exploiting the power-control condition \eqref{power-control} to estimate the power term on the 
right-hand side of  \eqref{enbal-VE}, 
we easily deduce that 
\begin{equation}
\label{sup-perto}
\begin{cases}
\sup_{t\in [0,T]} |\enet t{u(t)}{z(t)}| \leq \sup_{t\in [0,T]}  \calF(t,u(t),z(t)) \leq C_0,
\\
 \Vari {\mdsn{\spz}}{z}0{T}\leq  \Vari {\mdsn{\spz},\vecostname}{z}0{T} \leq C_0
\end{cases}
\end{equation}
for a constant $C_0>0$ only depending on $(u(0),z(0))$. 
\par
Observe that the $\cmdsn{\spz}$-stability condition, 
tested with $(u',z') =(u',z(t))$ and $u'$ arbitrary in $\spu$, in particular ensures that 
$u(t) \in \argmin_{u\in \spu} \enet tu{z(t)}$ for all $t\in [0,T] \setminus \jump z$. We want to claim this property at \emph{all} $t \in [0,T],$  though. 
That is why, \eqref{MINIMALITY} is required, as a separate property, at all $t\in [0,T]$. 
\end{remark}

\subsection{Characterization, properties, and main existence result for Visco-Energetic solutions}
\label{ss:2.3}
 In all of the following statements  we will implicitly assume that the rate-independent system $\RIS$ satisfies
conditions  $<T>, $ $  <A>,$ $  <B>,  $ and $<C>$  enucleated in Sec.\  \ref{ss:2.1}; we will impose them explicitly only  in the statement of  Theorem \ref{thm:exists-VE}.
\begin{lemma}
\label{l:u-regulated} 
Suppose that 
\begin{equation}
\label{unique-element}
\argmin_{u\in \spu} \enet tuz \ \  \text{ is a singleton for every } (t,z) \in [0,T]\times\domene z.
\end{equation}
Let $(u,z)$ be a Visco-Energetic solution to $\RIS$. Then, $u$ is $\sigma_\spu$-regulated, with left- and right-limits given by \eqref{def-end-points-u}.
\end{lemma}
\begin{proof}
Let us fix $t\in [0,T)$. 
In order to show that  the only element  $\rli ut $ in  $ \argmin_{u\in \spu} \enet tu{\rli z t}$ is the right-limit of $u$ w.r.t.\ the $\sigma_\spu$-topology,
it is sufficient to show that, for all $(s_n)_n \subset (0,T)$ with $s_n\downarrow t$, there holds
$u(s_n)\to \rli ut $ in $(\spu,\sigma_\spu)$. 
Since $\jump z $ is at most countable, we may suppose that $(s_n)_n\subset (0,T)\setminus \jump z$.  
It follows from \eqref{sup-perto} and $<A.2>$ that there exists some $u^*\in \spu$ such that, up to a (not relabeled) subsequence,
 $u(s_n)  \to u^*$ in $(\spu,\sigma_\spu)$ as $n\to\infty$. Clearly, $z(s_n)\to \rli zt$ in $(\spz,\sigma_\spz)$.  By the closure of the stable set $\mathscr{S}_{\mathsf{D}_\spz}$, we conclude that $(t,u^*,\rli zt) \in \mathscr{S}_{\mathsf{D}_\spz}$. Then, $u^* \in  \argmin_{u\in \spu} \enet tu{\rli z t}$, which yields $u^* = \rli ut$. 
 \par
 The argument for  the existence of the left-limit $\lli ut$ at all $t\in (0,T]$ is completely analogous.
\end{proof}
\par
We recall the following \underline{characterization of Visco-Energetic solutions}.
\begin{proposition}{\cite[Prop.\ 3.8]{SavMin16}}
\label{prop:charact}
A curve $(u,z)\in  \mathrm{B}([0,T];\spu) \times  \BV_{\sigma_\spz,\mdsn{\spz}}([0,T];\spz)$ satisfying the $\cmdsn{\spz}$-stability condition \eqref{stab-VE} is a $\VE$ solution of the 
rate-independent system $\RIS$ with the viscous correction $\delta_\spz$ if and only if $z$ satisfies, in addition,
\begin{enumerate}
\item  the $(\mdsn{\spz},\vecostname)$-energy-dissipation upper estimate
\begin{equation}
\label{enue-VE}
\enet T{u(T)}{z(T)} + \Vari {\mdsn{\spz},\vecostname}{z}0{T} \leq \enet 0{u(0)}{z(0)} +\int_0^T \pwt s{u(s)}{z(s)} \dd s;
\end{equation}
\item the $\mdsn{\spz}$-energy-dissipation upper estimate 
\begin{equation}
\label{enue-EN}
\enet T{u(T)}{z(T)} + \Vari {\mdsn{\spz}}{z}0{T} \leq \enet 0{u(0)}{z(0)} +\int_0^T \pwt s{u(s)}{z(s)} \dd s,
\end{equation}
joint with the following jump conditions at every jump point $t\in \jump z$:
\begin{equation}
\label{jump-conditions}
\begin{aligned}
&
\enet t{\lli ut}{\lli zt} - \enet t{u(t)}{z(t)}  && = &&  \vecost t{\lli zt}{z(t)}
\\
& 
\enet t{u(t)}{z(t)} - \enet t{\rli ut}{\rli zt}  && =  &&  \vecost t{z(t)}{\rli zt}
\\
&
\enet t{\lli ut}{\lli zt} - \enet t{\rli ut}{\rli zt} &&  = &&  \vecost t{\lli zt}{\rli zt}\,.
\end{aligned}
\end{equation}
\end{enumerate}
\end{proposition}
\par
Let us now gain further insight into the description of the system behavior at jumps provided by the $\VE$ concept,
via the  properties of 
\underline{Optimal Jump Transitions}. 
We recall that (cf.\ \cite[Def.\ 3.13]{SavMin16}), given $t\in [0,T]$ and $z_-, \, z+\in \spz$, 
an admissible transition curve  $\teta_z \in \mathrm{C}_{\sigma_\spz,\mdsn{\spz}}(E;\spz)$, with $E\Subset \R$, 
is an optimal transition between $z_-$ and $z_+$ at time $t\in [0,T]$ if it is  a minimizer for 
$\vecost t{z_-}{z_+}$, namely
\begin{equation}
\label{def:OJT}
\teta_z(E^-) = z_-, \quad \teta_z(E^+) = z_+,  \quad   \tcost{\VE}{t}{\teta_z}{E} = \vecost t{z_-}{z_+}\,.
\end{equation}
Furthermore,  we say that $\teta_z$ is  a
\begin{compactenum}
\item \emph{sliding transition}, if $\rstab t{\teta_z(s)} =0$ for  all $s\in E$;
\item \emph{viscous transition}, if  $\rstab t{\teta_z(s)} >0$ for  all $s\in E \setminus \{ E^-,E^+\}$.
\end{compactenum}
It has been shown in \cite[Rmk.\ 3.15, Cor.\ 3.17]{SavMin16} that,
 for a viscous transition $\teta_z$  between $z_-$ and $z_+$
 the compact set $E \setminus \{ E^-,E^+\}$ is discrete, i.e.\ all of its points are isolated: namely, $\teta_z$ is a \emph{pure jump} transition. In fact, $\teta_z$ may be represented as a finite, or countable, sequence $(\teta_n^z)_{n\in O}$, with $O $ a compact interval of $ \mathbb{Z}$, satisfying
 (recall the definition \eqref{minimal-set} of the set $M(t,z)$)
 \begin{equation}
 \label{minimum-jump}
 \teta_n^z \in M(t,\teta_{n-1}^z) = \mathrm{Argmin}_{z'\in \spz}  \left( \calI(t,z') {+}  \cmds{\spz}{\teta_{n-1}^z}{z'} \right) 
  \quad \text{for all } n \in O\setminus\{O^-\}.
 \end{equation}
 Furthermore, it has been  proved  in \cite[Prop.\ 3.18]{SavMin16} that any optimal jump transition can be canonically decomposed into (at most) countable collections of sliding and viscous, pure jump transitions. 
 Finally, it has been shown in \cite[Thm.\ 3.14]{SavMin16} that,  at every jump point $t$ of a $\VE$ solution $z$ there exists 
 an optimal jump transition $\teta_z$ between $\lli zt$ and $\rli zt$ such that $\teta_z(s) = z(t)$ for some $s\in E$. 
\par
We conclude this section by giving an existence result  for $\VE$ solutions, 
 proved in \cite[Thm.\ 4.7]{SavMin16}. For completeness,  in the statement   below we  also encompass   the convergence result (cf.\ 
  \cite[Thm.\ 7.2]{SavMin16})
  for  the  (left-continuous)  piecewise constant interpolants
 \begin{equation}
 \label{pwc-interp}
 \pwc Z\tau: [0,T]\to \spu, \quad \pwc Z\tau(0): = z_0, \quad \pwc Z\tau(t): = z^n \quad \text{for } t \in (t_\tau^{n-1},t_\tau^n],
\quad n=1,\ldots, N_\tau
 \end{equation}
 associated with the discrete solutions $(z_\tau^n)_{n=1}^{N_\tau}$ of the time-incremental minimization problem \eqref{tim:VE-true}. 
 We shall discuss the convergence of the interpolants $( \pwc U{\tau})_\tau$  of  the 
 elements $(u_{\tau}^n)_{n=1}^{N_{\tau}}$, with $u_{\tau}^n $  minimizers for  time-incremental minimization problem  \eqref{tim:VE-true},
 right after the statement of Thm.\ \ref{thm:exists-VE}. 
\begin{theorem}{\cite[Thm.\ 4.7]{SavMin16}}
\label{thm:exists-VE} 
Under Assumptions $<T>, $ $  <A>,$ $  <B>,  $ and $<C>,$  let $z_0   \in \domene z$. Then, for every  sequence $(\tau_k)_k$ of time steps
 with $\tau_k\down 0$ as $k\to\infty$ there exist a (not relabeled) subsequence $ (\pwc Z{\tau_k})_k$
  and $z \in \BV_{\sigma_\spz,\mdsn{\spz}}([0,T];\spz)$
  such  that 
  \begin{enumerate}
  \item 
  $z(0)=z_0$,  and 
\begin{equation}
\label{convergence-of-z}
\pwc Z{\tau_k}(t) \wsigmaz  z(t) \qquad \text{in $\spz$ for all $t\in [0,T]$};
\end{equation}
\item 
there exists $u \in \mathrm{B}([0,T];\spu) $ such that $(u,z)$ is a 
 $\VE$ solution to the rate-independent system $\RIS$, with the viscous correction $\delta$.
 \end{enumerate}
 \end{theorem}
 \noindent
 In fact, the curve $u$ 
 in the above statement
 is obtained as a measurable selection in $\mathrm{Argmin}_{u  \in \spu} \enet t{u}{z(t)}$. 
 It is not, in general, related to the limit of the  piecewise constant interpolants $(\pwc U{\tau_k})_k$. 
However, if, in addition, property \eqref{unique-element} holds,
and the functional $\calE$ fulfills the following $\Gamma$-$\limsup$ estimate, i.e.\ 
\begin{equation}
\label{gamma-limsup}
\begin{gathered}
\text{for all $(t_k)_k \subset [0,T]$ and $(z_k)_k\subset \spz$ with } t_k\to t,  \ z_k \wsigmaz z \text{ in } \spz \text{ then } 
\\
\text{for all } v \in \spu \text{ there exists } (v_k)_k\subset \spu \text{ such that } 
\limsup_{k\to\infty} \enet {t_k}{v_k}{z_k} \leq \enet tvz,
\end{gathered}
\end{equation}
then it is possible to prove convergence to the curve $u$. Namely, that 
\begin{equation}
\label{u-convergence}
\pwc U{\tau_k}(t) \wsigmau  u(t) \text{ in $\spu$ for all $t\in [0,T]$}.
\end{equation}
To check this, we may observe that from   \eqref{tim:VE-true} it follows that 
\begin{equation}
\label{in-argmin}
\pwc U{\tau_k}(t) \in \mathrm{Argmin}_{u  \in \spu} \enet {\pwc t{\tau_k}(t)}{u}{\pwc Z{\tau_k}(t)} \qquad \text{for all $t\in (0,T]$}
\end{equation}
 (with 
$\pwc t{\tau_k}$ the left-continuous  piecewise constant interpolant associated with the partition of $[0,T]$). 
From the energy bound $\calF_0(\pwc U{\tau_k}(t),\pwc Z{\tau_k}(t) ) \leq C$ for a constant independent of $k\in \N $ and $t\in [0,T]$,
cf.\ \cite[Thm.\ 7.1]{SavMin16},  combined with Assumption $<A.2>$,
we infer that there exists a compact subset $\mathbf{U} \Subset U$ such that 
$\pwc U{\tau_k}(t) \in \mathbf{U}$ for all $t\in [0,T]$ and $k\in \N$. Then, for all $t\in [0,T]$ there exists $u_*(t) \in U $ such that, along a (not relabeled) subsequence possibly depending on $t$, there holds 
\begin{equation}
\label{depending-on-t-cvg}
\pwc U{\tau_k}(t)\wsigmau u_*(t).
\end{equation}  
Combining \eqref{in-argmin} and \eqref{depending-on-t-cvg}
 with \eqref{convergence-of-z} and taking into account the lower semicontinuity $<A.1>$
we find that
$
\enet {t}{u_*(t)}{z(t)} \leq \liminf_{k\to\infty}  \enet {\pwc t{\tau_k}(t)}{\pwc U{\tau_k}(t)}{\pwc Z{\tau_k}(t)}  \leq
 \liminf_{k\to\infty}  \enet {\pwc t{\tau_k}(t)}{v}{\pwc Z{\tau_k}(t)}  $ for all $v\in \spu$ and all $t \in [0,T]$. Exploiting 
 \eqref{gamma-limsup}, we conclude that $u_*(t) \in 
 \argmin_{u\in \spu} \enet t{u(t)}{z(t)}$. Since the latter set is a singleton
 by \eqref{unique-element}, 
  convergence \eqref{depending-on-t-cvg} holds for the whole sequence $(\tau_k)_k$, and we conclude \eqref{u-convergence}. 

\section{When Visco-Energetic solutions are Energetic: the case of perfect plasticity}
\label{s:perfect-plas}
The following result characterizes the situation in which  $\VE$ solutions turn out to be $\ENE$ solutions as well. Note that it holds under the sole  conditions 
$<A.1>$  and  $<A.3>$.
\begin{proposition}
\label{prop:VE=E}
 Assume $<T>$, $<A.1>$,  and  $<A.3>$.   Then, a Visco-Energetic solution
$(u,z)$ of the rate-independent system $\RIS$ is an Energetic solution if and only if it satisfies the global stability condition \eqref{globStab} at every $t\in [0,T]$. 
In that case, at   every jump point $t\in \mathrm{Jump}_z$   the curves $(u,z)$ fulfill the jump conditions 
\begin{equation}
\label{E-jump-conds}
\begin{aligned}
&
\enet t{\lli ut}{\lli zt} - \enet t{u(t)}{z(t)}  =   \mds{\spz}{\lli zt}{z(t)}, \\
&
\enet t{u(t)}{z(t)} - \enet t{\rli ut}{\rli zt}  =   \mds {\spz}{z(t)}{\rli zt}. 
\end{aligned}
\end{equation}
\end{proposition}
\begin{proof}
Clearly, 
if $(u,z)$ is an $\ENE $ solution, then  \eqref{globStab} holds.
\par
 Conversely,  
let $(u,z)$ be a  $\VE $ solution complying with  \eqref{globStab}.
Since
$\Vari {\mdsn{\spz},\vecostname}{z}0{t} \geq \Vari {\mdsn{\spz}}{z}0{t} $ for all $t\in [0,T]$, 
from the  energy-dissipation balance \eqref{enbal-VE}
we deduce that  $(u,z)$ fulfills the Energetic energy-dissipation upper estimate
\eqref{enue-EN} on $ [0,t]$. Then, taking into account  \eqref{globStab},  we may apply either 
\cite[Prop.\ 2.1.23]{MieRouBOOK} or \cite[Lemma 6.2]{SavMin16}, mimicking the argument of the proof of Thm.\ 6.5 therein.  In this way we conclude that 
$(u,z)$ is an Energetic solution.
Hence, comparing \eqref{enbal-VE} and \eqref{enbal-E} we ultimately find 
\[
  \Vari {\mdsn{\spz}}{z}0{t} =  \Vari {\mdsn{\spz},\vecostname}{z}0{t} 
  \stackrel{\eqref{augm-tot-var}}{=}   \Vari{\mdsn \spz}{z}{0}{t} +  \Jvar {\Delta_{\vecostname}}{z}{0}{t} \quad \text{for all $t\in[0,T]$}\,.
  \]
Therefore, at every $t\in \mathrm{Jump}_z$ there holds
$ \Delta_{\vecostname}(t,\lli zt, z(t)) =  \Delta_{\vecostname}(t,z(t),\rli z t) =0$, i.e.\
\begin{equation}
\label{costs=dists}
\vecost t{\lli z t}{z(t)} = \mds{\spz}{\lli z t}{z(t)} \quad \text{ and  } \quad 
\vecost t{z(t)}{\rli zt} = \mds{\spz}{z(t)}{\rli zt}.
\end{equation} 
Combining \eqref{costs=dists} with the Visco-Energetic jump conditions \eqref{jump-conditions} we immediately deduce \eqref{E-jump-conds}. 
\end{proof}
\noindent Small-strain associative elastoplasticity, with the Prandtl-Reuss flow rule 
(without hardening) for the plastic strain, provides an example of a rate-independent system to which 
Proposition \ref{prop:VE=E} applies, cf.\ Thm.\ \ref{thm:pp} ahead.
\par
Before entering into details, we fix the following
\begin{notation}
\label{not:3.1}
\upshape
 We will use the symbol
 $\bbM^{d\times d}$  
for  the space of $d{\times} d$ 
  matrices, endowed with the  Frobenius inner product 
$\eta : \xi : = \sum_{i j} \eta_{ij} \xi_{ij}$ for two matrices $\eta= (\eta_{ij})$ and $\xi = (\xi_{ij})$.
We will denote by $|\cdot|$ the 
 induced the matrix norm 
 and,
   in accordance with Notation \ref{gen-not}, by  $\overline{B}_r$  the closed ball  with radius $r$ centered  at $0$ in $\bbM_\sym^{d\times d }$. The latter symbol denotes
 the subspace of symmetric matrices, while $\bbM_\dev^{d\times d}$ stands for the subspace of symmetric matrices with null trace. In fact, 
 every $\eta \in \bbM_\sym^{d\times d}$ can be written as 
$
\eta = \eta_\dev+ \frac{\mathrm{tr}(\eta)}d I
$
with $\eta_\dev$ the orthogonal projection of $\eta$ into $\bbM_\dev^{d\times d} $. We will refer to $\eta_\dev$ as the deviatoric part of $\eta$.  
With the symbol $\odot$ we will  denote the symmetrized tensor product  of two vectors $a,\, b \in \R^d$, defined as the symmetric matrix with entries $ \frac{a_ib_j + a_j b_i}2$.
Finally,
\[
\BD(\Omega;\R^d) : = \{ \tilde{u}\in L^1(\Omega;\R^d)\, : \ \eps(\tilde{u}) \in \M(\Omega;\bbM_\sym^{d\times d})\}
\]
 is the space of functions with bounded deformation,  such that the (distributional) strain tensor $\eps(\tilde{u})$  is a Radon measure  on $\Omega$, valued in $\bbM_\sym^{d\times d}$, and 
\[
 \M(\Omega{\cup}\Gdir;\bbM_\dev^{d\times d}) \text{ is the space of
($\bbM_\dev^{d\times d}$-valued) Radon measures on $\Omega{\cup}\Gdir$.}
\]
\end{notation}
 The PDE system governing perfect plasticity,
formulated in a (bounded, Lipschitz)  domain $\Omega \subset \R^d$  (the reference configuration)
 consists of 
 \begin{compactitem}
 \item[-]
the equilibrium equation 
\begin{subequations}
\label{perf-plast-syst}
\begin{equation}
\label{perf-plast-eq}
-\mathrm{div}(\mathbb{C} e) = f \quad \text{ in } \Omega \times (0,T),
\end{equation}
where $f$ is a time-dependent body force,  $\bbC$ is the (symmetric, positive definite)  elasticity tensor, $e$ the elastic strain, which enters into the \emph{additive} decomposition 
of the  (symmetric) linearized strain 
tensor $\eps(\tilde{u}) = \tfrac12 (\nabla \tilde{u} +\transp{\nabla \tilde{u}} )$ (with $\tilde{u}:\Omega \to \R^d$ the displacement and $\transp{A}$ the transpose of a matrix $A$),
into an elastic  and a plastic part, i.e.
\begin{equation}
\label{el-pl-decomp}
\eps(\tilde{u}) = e+p  \quad \text{ in } \Omega \times (0,T);
\end{equation}
\item[-] the flow rule for the plastic tensor $p$
\begin{equation}
\label{perf-plast-flow-rule}
\partial \mathrm{R}(\dot p) \ni  \sigma_\dev \quad \text{ in } \Omega \times (0,T),
\end{equation}
where
$\sigma_\dev $ is the deviatoric part of the stress 
 $\sigma : = \mathbb{C} e$, the $1$-homogeneous dissipation potential
 $ \mathrm{R}$ is the support function 
 of the   closed convex  subset $K\subset \bbM_{\dev}^{d\times d}$ to which the (deviatoric part of the) stress is constrained to belong, and $\partial \mathrm{R}:  \bbM_{\dev}^{d\times d} \rightrightarrows  \bbM_{\dev}^{d\times d} $ is the convex analysis subdifferential of $\mathrm{R}$. 
  \end{subequations}
 \end{compactitem}
 \begin{subequations}
 \label{ass:perf-plast}
 Along the footsteps of \cite{DMDSMo06QEPL}, 
  we will suppose hereafter
that
\begin{equation}
 \label{assK-pp}
 \overline{B}_{r_k} \subset K \subset \overline{B}_{R_K} \quad \text{for some } 0<r_K\leq R_K.
 \end{equation}
Furthermore, we will have 
 $\partial\Omega = \Gdir \cup \Gneu \cup \partial\Gamma$, with $\Gdir$ and $\Gneu$ disjoint  open
 sets and $\partial\Gamma$ their common boundary,   and we will denote by $\nu$ the external unit normal to $\partial\Omega$.  We will  assume that 
 \begin{equation}
 \label{assOmega-pp}
\mathscr{H}^{d-1}( \Gdir)>0   \ \text{ and } \  \partial\Omega, \, \partial\Gamma \text{ are of class } \mathrm{C}^2
 \end{equation}
 (with  $\mathscr{H}^{d-1}$ the $(d{-}1)$-dimensional Hausdorff measure). 
On the Dirichlet part of the boundary $\Gdir$ we will prescribe a Dirichlet condition through an assigned function 
\begin{equation}
\label{assw-pp}
w_\Dir\in \mathrm{C}^1([0,T];H^1(\Omega;\R^d)),
\end{equation}
with trace on $\Gdir$ still denoted by $w_\Dir$. On the Neumann part $\Gneu$ we will apply a non-zero traction $g$. A standard 
condition in perfect plasticity is that the body and surface forces 
\begin{equation}
\label{assf&g-pp}
\begin{gathered}
f\in \rmC^1([0,T];L^d(\Omega;\R^d)), \ g \in \rmC^1 ([0,T];L^\infty (\Gneu;\R^d)) \text{ satisfy the \emph{safe-load} condition:}
\\
\exists\, \varrho \in \AC ([0,T];L^2(\Omega;\bbM_\sym^{d\times d})) \text{ with } \varrho_\dev \in \AC ([0,T]; L^\infty(\Omega;\bbM_\sym^{\dev}) \text{ s.t. } \begin{cases}
-\mathrm{div}(\varrho) = f \text{ in } \Omega
\\
\varrho \nu  = g \text{ in } \Gneu
\end{cases}  \text{ on }  (0,T)
\\
\text{ and  fulfilling  } \quad \exists\, \mathfrak{r}>0  \ \forall\, t\in [0,T] \ \foraa\, x \in \Omega \,: \ \varrho_\dev(t,x)+ 
 \overline{B}_\mathfrak{r}   \subset K\,.
\end{gathered}
\end{equation}
  \end{subequations}
With $f$ and $g$ we associate the total load function 
\begin{equation}
\label{total-load}
\ell: [0,T]\to \BD(\Omega;\R^d)^*, \qquad  \pairing{}{\BD(\Omega;\R^d)}{\ell(t)}{v}: = \int_{\Omega} f(t) v \dd x  + \int_{\Gneu}g(t)v \dd \mathscr{H}^{d{-}1}(x).
\end{equation}
Indeed, the above integrals are well defined for any $v\in  \BD(\Omega;\R^d)$ due to the embedding and trace properties of 
$ \BD(\Omega;\R^d)$. Clearly, $\ell(t)$ is also an element of $H^1(\Omega;\R^d)^*$ for every $t\in [0,T]$; in what follows,
to avoid overburdening notation, we will often omit to specify the spaces when writing   the duality pairing $\pairing{}{}{\ell(t)}{v}$.
\par
With the boundary datum $w_\Dir$ we associate the set $\mathcal{A}(w_\Dir) $ of the kinematically admissible states $(\tilde{u},p)$,  viz.
\begin{equation}
\label{def-admissible-states}
\begin{aligned}
&
\begin{aligned}
(\tilde{u},p) \in \mathcal{A}(w_\Dir) \text{ if and only if } 
 \ & \text{(i)} && \tilde{u} \in \BD(\Omega;\R^d), \ p \in \M(\Omega{\cup}\Gdir;\bbM_\dev^{d\times d}),
 \\
 & \text{(ii)} && e = \eps(\tilde{u}) - p \in L^2(\Omega;\bbM_\sym^{d\times d}),
 \\
 & \text{(iii)} && p= (w_\Dir{-}\tilde{u})\odot \nu \mathscr{H}^{d-1} \text{ on } \Gdir.
\end{aligned}
\\
& 
\text{We set } \mathcal{A}: =  \mathcal{A}(0).
\end{aligned}
\end{equation}
 Indeed,  an admissible $\tilde{u}$ may have jumps   (i.e., the measure $\eps(\tilde{u})$  can concentrate on)
$\partial\Omega$. Hence, the boundary condition $\tilde{u}=w_\Dir$  on $\Gdir$ has to be relaxed in terms of \eqref{def-admissible-states}(iii) (to be understood as an equality between measures on $\Gdir$),
which expresses the fact that any jump of $\tilde{u}$ violating  the Dirichlet condition $\tilde{u}=w_\Dir$ is due to a localized plastic deformation. 
From now on, we will use the splitting
\begin{equation}
\label{split4u}
\tilde u = u+w_\Dir
\end{equation}
and work with the state variables $(u,p)$. 
\par
The Energetic formulation (cf. \cite{DMDSMo06QEPL})  of the perfectly plastic system \eqref{perf-plast-syst}  is given in this 
 setup:
 \paragraph{\bf Ambient space:}
\begin{subequations}
\label{ppsyst-setup}
\begin{equation}
\label{ppsyst-X}
\Xs = \spu \times \spz \quad \text{with } \spu =  \BD(\Omega;\R^d), \  \spz=  \M(\Omega{\cup}\Gdir;\bbM_\dev^{d\times d})
\end{equation}
and
 (1)
 $\sigma_\spz$ is  the weak$^*$-topology on $\M(\Omega{\cup}\Gdir;\bbM_\dev^{d\times d})
$, identified with the dual of the space of ($\bbM_\dev^{d\times d}$-valued) continuous functions with compact support  on $\Omega{\cup}\Gdir$; (2)  
  $\sigma_\spu$  is the weak$^*$ topology on $ \BD(\Omega;\R^d)$ (which  has in fact a predual, cf.\ e.g.\ \cite{Temam-Strang80}), inducing the following notion of weak$^*$-convergence: $u_k \weaksto u$ in $ \BD(\Omega;\R^d)$ if and only if 
$u_k\weakto u$ in  $L^1(\Omega;\R^d)$ and $\eps(u_k)\weaksto \eps( u) $ in $ \M(\Omega;\bbM_\sym^{d\times d})$.
\paragraph{\bf Energy functional:}
\begin{equation}
\label{ppsyst-E}
\enet t{  u}{p} : = \frac12\int_{\Omega} \bbC(\eps( u{+}w_\Dir(t)){-}p) :  (\eps( u{+}w_\Dir(t)){-}p)   \dd x + I_{\mathcal{A}}( u,p) 
- \pairing{}{\BD(\Omega;\R^d)}{\ell(t)}{ u{+}w_\Dir(t)}\,.
\end{equation}
Here, 
 the indicator function $I_{\mathcal{A}}$
forces the constraint $( u,p) \in \mathcal{A} $, so that $\tilde{u} = u+w_\Dir\in  \mathcal{A}(w_\Dir)$;
\paragraph{\bf Dissipation distance:} it
 is defined in terms of the support function
\[
 \mathrm{R}: \bbM_{\dev}^{d\times d} \to [0,\infty), \qquad 
 \mathrm{R}(\pi) : = \sup_{\omega \in K} \omega : \pi
\]
of the set $K$ from 
\eqref{assK-pp}, via 
\begin{equation}
\label{ppsyst-D}
\mds{\spz}{p}{\tilde p}:  = \mathcal{R}(\tilde{p} - p) \quad \text{with } \mathcal{R}(\mathcal{\pi}):= 
\int_{\Omega{\cup}\Gdir} \mathrm{R}\left(\frac{\pi}{|\pi|} \right) |\pi|(\dd x) \text{ for all } \pi \in \M(\Omega{\cup}\Gdir;\bbM_{\dev}^{d\times d})\,,
\end{equation}
where $|\pi|$ is the variation of $\pi$ and $\frac{\pi}{|\pi|} $ its Radon-Nykod\'im derivative w.r.t.\  $|\pi|$. 
\end{subequations}
\par
It is straightforward to check that in  the above metric-topological setting  $<T>$ is fulfilled.
For the reader's convenience, we recapitulate here the arguments from  \cite{DMDSMo06QEPL} to show that
\begin{lemma}
\label{l:en-pp}
Under conditions \eqref{ass:perf-plast}, the energy functional $\calE$ from \eqref{ppsyst-E} fulfills $<A.1>, <A.2>$,  $<A.3>$.
\end{lemma}
\begin{proof}
In view of  the safe-load condition \eqref{assf&g-pp} and   \cite[Lemma 3.1]{DMDSMo06QEPL}, 
the loading term rewrites as 
\[
 \pairing{}{\BD(\Omega;\R^d)}{\ell(t)}{ u{+}w_\Dir(t)} = \pairing{}{}{\varrho}{ \eps( u{+}w_\Dir (t)) {-}p} + \pairing{}{}{\varrho_\dev}{p}  + \pairing{}{}{\ell(t)}{w_\Dir (t)}  -  \pairing{}{}{\varrho}{ \eps(w_\Dir (t))},
 \]
 where the duality pairing $ \pairing{}{}{\varrho_\dev}{p} $ involving the 
 \emph{measure} $p$  has been  carefully defined in   \cite[Sec.\ 2]{DMDSMo06QEPL}, and the other duality pairings are not specified for notational simplicity. 
 Let us now fix  any     reference point  $p_o \in \spz$ satisfying the kinematical admissibility condition \eqref{def-admissible-states} (i.e.,
 such that there exists $u_o\in \spu$ such that $(u_o,p_o) \in \mathcal{A}(w_\Dir (0))$).
 Therefore,    suitably choosing $F_0$ (cf.\ \eqref{F0pp} below), we  find  for all $( u,p) \in \mathcal{A}$
that 
 \begin{equation}
 \label{ad-coerc-pp}
 \begin{aligned}
 \calF(t,{u},p) & = \enet t{{u}}p + \mds{\spz}{p_o}{p}+F_0
 \\ & 
 =  \frac12\int_{\Omega} \bbC(\eps( u{+}w_\Dir (t)){-}p) :  (\eps( u{+}w_\Dir (t)){-}p)   \dd x   -  \pairing{}{}{\varrho(t)}{ \eps( u{+}w_\Dir ) {-}p} + \mathcal{R}(p {-} p_o) -  \pairing{}{}{\varrho_\dev(t)}{p{-}p_o} 
 \\
 & \quad 
 + \pairing{}{}{\varrho_\dev(t)}{p_o}   -   \pairing{}{}{\ell(t)}{w_\Dir (t)}  +  \pairing{}{}{\varrho(t)}{ \eps(w_\Dir(t))} +F_0
 \\
 & 
 \stackrel{(1)}{\geq} 
 \frac{\gamma_{\bbC}}4 \|\eps( u{+}w_\Dir (t)){-}p\|_{L^2(\Omega)}^2 - \frac1{\gamma_{\bbC}} \|\varrho(t) \|_{L^2(\Omega)}^2  + 
 \mathfrak{r} \| p-p_o\|_{\mathrm{M}(\Omega{\cup}\Gdir)}
 \end{aligned}
 \end{equation}
(observe that the duality pairing $ \pairing{}{}{\varrho_\dev(t)}{p_o} $ is well defined since $p_o$ is a kinematically admissible strain, cf.\
  \cite[Sec.\ 2]{DMDSMo06QEPL}). 
Here, {\footnotesize (1)} follows from \emph{(i)} the estimate $ \tfrac12 \bbC e :  e - \varrho: e \geq\tfrac12 \gamma_{\bbC} |e|^2 -  \tfrac14 \gamma_{\bbC} |e|^2 - \tfrac1{\gamma_\bbC} |\varrho|^2$ by the positive-definiteness of  $\bbC$ and Young's inequality; \emph{(ii)}     \cite[Lemma 3.2]{DMDSMo06QEPL}, which ensures
the estimate $ \mathcal{R}(p {-} p_o) -  \pairing{}{}{\varrho_\dev(t)}{p{-}p_o} \geq   \mathfrak{r} \| p-p_o\|_{\mathrm{M}(\Omega{\cup}\Gdir)}$, 
  with $\mathfrak{r}>0$ from the safe-load condition   \eqref{assf&g-pp};  \emph{(iii)} choosing 
 \begin{equation}
 \label{F0pp}
 F_0 \geq \sup_{t\in [0,T]} \left(  | \pairing{}{}{\varrho_\dev(t)}{p_o} | {+}  |\pairing{}{}{\ell(t)}{w_\Dir(t)}|  {+}  \pairing{}{}{\varrho(t)}{ \eps(w_\Dir(t))} \right)\,
 \end{equation}
 thanks to \eqref{assw-pp} and  \eqref{assf&g-pp}. 
 From  \eqref{ad-coerc-pp} and again  \eqref{assw-pp}--\eqref{assf&g-pp} we thus deduce that 
 \begin{equation}
 \label{coerc-pp} 
 \begin{aligned}
 &
 \exists\, c,\, C>0 \ \forall\,  (t,{u},p) \in [0,T]\times \spu \times \spz \text{ with }  ( u,p) \in \mathcal{A} \, :
 \\
 &  \quad 
 \calF(t,{u},p) \geq c\left(     \|\eps( {u} )  \|_{\M(\Omega)} + \|p\|_{\mathrm{M}(\Omega{\cup}\Gdir)} + \|\eps( {u} ) {-} p  \|_{L^2(\Omega)}   \right) -C\,.
 \end{aligned}
 \end{equation}
 From the bound for $p$ and the information that ${u}\odot \mathscr{H}^{d-1} =-p$ we conclude a bound for ${u} $ in $L^1(\Gdir;\R^d)$. Therefore, a Poincar\'e-type estimate for $\BD$-functions
  (cf.\ e.g., \cite[Prop.\ 2.4, Rmk.\ 2.5]{Temam83}) yields a bound for $u$ in $\BD(\Omega;\R^d)$. We thus conclude that the sublevels of $\calF$ are bounded in $[0,T]\times \spu\times \spz $,
   and thus sequentially relatively compact w.r.t.\ the $\sigma_\R$ topology,  whence $<A.2>$. 
 \par
Relying on estimate \eqref{coerc-pp}  and on the closedness properties of the set $\mathcal{A}(w_\Dir)$ (cf.\ \cite[Lemma 2.1]{DMDSMo06QEPL}), it is standard to show that $\calE$ is sequentially l.s.c.\
 on  sublevels of $\calF$ w.r.t.\  the $(\sigma_U {\times} \sigma_Z)$-topology, i.e.\ $<A.1>$.
\par
It follows from  \eqref{assw-pp} and  \eqref{assf&g-pp} that 
$\partial_t \enet t{ u}{p}$ exists and 
\[
\begin{aligned}
\partial_t \enet t{ u}{p} = \int_{\Omega} \bbC \left( \eps( u) {-} p {+} \eps(w_\Dir(t))\right): \eps(\dot{w}_\Dir(t)) \dd x  &  - \pairing{}{H^1(\Omega;\R^d)}{\dot{\ell}(t)}{w_\Dir(t)}  - \pairing{}{H^1(\Omega;\R^d)}{{\ell}(t)}{\dot{w}_\Dir(t)} \\ &    \foraa\, t \in (0,T) \text{  and for all } ( u,p) \in\domene{}.
\end{aligned}
\]
Therefore,   in view of \eqref{coerc-pp} we find that 
\[
\begin{aligned}
|\partial_t \enet  t{ u}{p}  |  &  \leq C  \| \eps( u) {-} p\|_{L^2(\Omega)} \| \eps(\dot{w}_\Dir(t))\|_{L^2(\Omega)}  + C   \| \eps(w_\Dir(t))\|_{L^2(\Omega)}   
 \| \eps(\dot{w}_\Dir(t))\|_{L^2(\Omega)}  
 \\ & \quad 
+ \| \dot{\ell}(t)\|_{H^1(\Omega)^*} \| w_\Dir(t)\|_{H^1(\Omega)}  + \| \ell(t)\|_{H^1(\Omega)^*} \| \dot{w}_\Dir(t)\|_{H^1(\Omega)} 
\\ & 
\leq   \Lambda_P  \left( \calF(t,{u},p) +C_P \right)
\end{aligned}
\]
with $\Lambda_P(t) = C\sup_{t\in[0,T]}\left(  \| \dot{w}_\Dir(t)\|_{H^1(\Omega)} {+} \| \dot{\ell}(t)\|_{H^1(\Omega)^*}  \right)  $
and $C_P =  \| \ell\|_{L^\infty(0,T;H^1(\Omega)^*) }^2  + \| w_\Dir \|_{L^\infty (0,T;H^1(\Omega))}^2 $. It is straightforward to check that, again thanks to  \eqref{assw-pp}--\eqref{assf&g-pp}  and  \eqref{coerc-pp},  $\partial_t  \enet t{ u}p$ is 
indeed (sequentially)  continuous w.r.t.\ the $\sigma_\R$-topology
 on the sublevels of $\calF$. This concludes the proof of  $<A.3>$. 
\end{proof} 
\paragraph{\bf The viscous correction:}
Let us now consider the  family  of viscous corrections
\begin{equation}
\label{visc-corr-pp}
\delta_{\spz}(p,\tilde p): = h(\mds{\spz}{p}{\tilde p}) = h(\mathcal{R}(\tilde{p}{-}p)) \quad \text{for all } p,\, \tilde p \in \spz \text{ and } h \text{ as in  \eqref{a-funct-of-d}} 
\end{equation}
(cf.\ Remark \ref{rmk:gener-visc-corr-pp} for a discussion on more general viscous corrections).
With our next result we show that, in the frame of the rate-independent system $\RIS$ 
given by \eqref{ppsyst-setup} and 
with this choice of $\delta_\spz$,  the Visco-Energetic stability condition   \eqref{stab-VE} in indeed \emph{equivalent} to the Energetic stability \eqref{globStab}.
\begin{proposition}
\label{prop:VE2E-stab}
Assume \eqref{ass:perf-plast} and let $( u, p) \in  \mathrm{B}([0,T];\BD(\Omega;\R^d)) \times  \BV([0,T];\M(\Omega{\cup}\Gdir;\bbM_{\dev}^{d\times d})) $
 fulfill $( u(t),p(t)) \in \mathcal{A}$ for all $t\in [0,T]$. 
Then, the following conditions are \emph{equivalent}  at a given $t\in [0,T]$:
\begin{enumerate}
\item 
$( u,p)$ fulfill the   stability condition   \eqref{stab-VE} for the rate-independent system $\RIS$ 
 \eqref{ppsyst-setup}, with the viscous correction   $\delta_\spz$ from 
\eqref{visc-corr-pp};
 \item   there holds 
 \begin{subequations}
 \label{mom-bal-pp}
 \begin{equation}
 \begin{gathered}
 \sigma(t) = \bbC(\eps({u}(t)+w_\Dir(t)) -p(t)) \in \Sigma(\Omega) \cap 
  \mathcal{K}(\Omega) 
  \text{ with }\\
   \begin{cases}
  \Sigma(\Omega): =  \{ \sigma \in L^2(\Omega;\bbM_{\sym}^{d\times d})\, : \ \mathrm{div}(\sigma) \in L^d(\Omega;\R^d),\ \sigma_\dev \in L^\infty(\Omega;\bbM_\dev^{d\times d}) \},
  \\
   \mathcal{K}(\Omega):= \{ \sigma \in L^2(\Omega;\bbM_{\sym}^{d\times d})\, : \ \sigma_\dev(x) \in K \ \foraa\, x \in \Omega\}\,,
  \end{cases}
  \end{gathered}
 \end{equation}
 and 
  \begin{equation}
-\mathrm{div}(\sigma(t)) = f(t) \text{ a.e.\ in } \Omega, \qquad 
\sigma(t) \nu  = g(t) \text{ on } \Gneu\,;
  \end{equation}
  \end{subequations}
\item  $( u,p)$ fulfill the   stability condition    \eqref{globStab} for the rate-independent system $\RIS$  from 
 \eqref{ppsyst-setup}.
\end{enumerate}
\end{proposition}
\begin{proof}
First of all, we show that \textbf{(1) $\Rightarrow$ (2)}. Indeed, in the  stability condition   \eqref{stab-VE}, i.e. 
\[
\begin{aligned}
&
 \frac12\int_{\Omega} \bbC(\eps( u{+}w_\Dir(t)){-}p) :  (\eps( u{+}w_\Dir(t)){-}p)   \dd x 
- \pairing{}{\BD(\Omega;\R^d)}{\ell(t)}{ u(t)}\\ &  \leq 
 \frac12\int_{\Omega} \bbC(\eps( u'{+}w_\Dir(t)){-}p') :  (\eps( u'{+}w_\Dir(t)){-}p')   \dd x 
- \pairing{}{\BD(\Omega;\R^d)}{\ell(t)}{ u'}  + \calR(p'{-}p(t)) + h( \calR(p'{-}p(t)) ) 
\end{aligned}
\]
for all $ ({u}',p') \in \mathcal{A}$, we choose
$({u}',p') :  = ({u}(t) + \eta {v}, p(t)+\eta q)$, with arbitrary  $\eta\in \R$ and  $({v},q) \in \mathcal{A}$. With straightforward calculations we find
\[
\begin{aligned}
0\leq \frac12   \int_{\Omega} ( \eta \eps({v}){-}\eta q):( \eta \eps({v}){-}\eta q) \dd x
 & + \int_{\Omega}  \bbC(\eps( u{+}w_\Dir(t)){-}p) :( \eta \eps({v}){-}\eta q) \dd x -\pairing{}{\BD(\Omega;\R^d)}{\ell(t)}{\eta v} \\ & +   \calR(\eta q) + h( \calR(\eta q) )\,.
 \end{aligned}
\]
Hence, by the positive homogeneity of $\calR$ we conclude 
\[
\begin{aligned}
0\leq \eta^2 \frac12 \int_\Omega \bbC (\pm\eps({v}) \mp q) :  (\pm \eps({v}) \mp q)  \dd x 
 &  +\eta \int_{\Omega}  \bbC(\eps( u{+}w_\Dir(t)){-}p) :( \pm \eps({v}) \mp \eta q) \dd x -\eta \pairing{}{\BD(\Omega;\R^d)}{\ell(t)}{\pm v} \\ & +   \eta\calR(\pm q) + h( \calR(\eta (\pm q)) ) \quad \text{for all } \eta>0.
 \end{aligned}
\]
Dividing by $\eta$ and letting $\eta\down 0$,
and using that 
\begin{equation}
\label{used-here-pp}
\lim_{\eta\down 0} \frac {h( \calR(\eta (\pm q)) )}\eta = \lim_{\eta\down 0} \frac {h( \calR(\eta (\pm q)) )}{\calR(\eta (\pm q)) } \frac {\calR(\eta (\pm q))}{\eta} =0  
\end{equation}
thanks to property \eqref{a-funct-of-d},
 we find that 
\begin{equation}
\label{2sep-ineqs}
\begin{cases}
 - \calR(q) \leq \int_{\Omega} \sigma(t) : (\eps( v) {-} q ) \dd x  - \pairing{}{\BD(\Omega;\R^d)}{\ell(t)}{ v},
\\
  \int_{\Omega} \sigma(t) : (\eps( v) {-} q ) \dd x  - \pairing{}{\BD(\Omega;\R^d)}{\ell(t)}{ v} \leq \calR(q) 
\end{cases} 
\quad \text{for all } ( v,q) \in \mathcal{A}.
\end{equation}
It has been shown in \cite[Prop.\ 3.5]{DMDSMo06QEPL}   that \eqref{2sep-ineqs} is equivalent to \eqref{mom-bal-pp}.  This shows \textbf{(2)}. 
\par
In turn,  \textbf{(2) $\Leftrightarrow$ (3)} by \cite[Thm.\ 3.6]{DMDSMo06QEPL}. Finally,  we clearly have that
 \textbf{(3) $\Rightarrow$ (1)}.  This concludes the proof. 
\end{proof}
We are now in a position to prove
\begin{theorem}
\label{thm:pp}
Assume \eqref{ass:perf-plast}  and let  $( u, p) \in   \mathrm{B}([0,T];\BD(\Omega;\R^d)) \times  \BV([0,T];\M(\Omega{\cup}\Gdir;\bbM_{\dev}^{d\times d})) $ be a $\VE$ solution of the rate-independent system $\RIS$  from \eqref{ppsyst-setup}, with the viscous correction $\delta_\spz$ from  \eqref{visc-corr-pp}. Suppose that $( u, p)$  fulfills at $t=0$ the stability condition
\begin{equation}
\label{stab-at0}
\enet t{{u}(0)}{p(0)} \leq \enet t{{u}'}{p'} + \calR(p'{-}p(0)) \quad \text{for all } ({u}',p') \in \mathcal{A}.
\end{equation}
Then, $( u, p)$ is an Energetic solution of the rate-independent system $\RIS$   \eqref{ppsyst-setup}.
\end{theorem}
\begin{proof}
 We  have that $( u, p)$ fulfills the stability condition   \eqref{globStab} at $t=0$ and,
in view of  Prop.\ \ref{prop:VE2E-stab}, whenever it fulfills    \eqref{stab-VE}, i.e.\ at every  $t\in [0,T]\setminus \mathrm{J}_z$. Passing to the limit
in   \eqref{globStab} we conclude that it holds also at the  end-points of every jump, i.e.\
\begin{equation}
\label{stab-at-jumps-0}
\enet t{\lri  ut}{\lri p t} \leq \enet t{{u}'}{p'} +\calR(p'{-}\lri pt) \quad \text{for all } ({u}',p') \in \spu\times \spz \text{ and all } t \in \mathrm{J}_z\,.
\end{equation}
With the very same argument as in the proof  of \cite[Thm.\ 1]{RS17}, using the upper energy-dissipation estimate \eqref{enue-EN} we deduce that
\[
\enet t{ u(t)}{p(t)} +\calR(p(t){-}\lli pt)  \leq \enet t{\lli ut}{\lli pt}\quad  \text{ for all } t \in \mathrm{J}_z\,,
\]
which, combined with \eqref{stab-at-jumps-0} and the triangle inequality for $\calR$, delivers the stability  \eqref{globStab} at  all $t\in \mathrm{J}_z$. In view of
Lemma \ref{l:en-pp}, we may  then apply Prop.\ \ref{prop:VE=E} and conclude that $( u, p) $ is an Energetic solution.
\end{proof}
\begin{remark}
\label{rmk:gener-visc-corr-pp}
\upshape
Indeed, Theorem \ref{thm:pp} carries over to $\VE$ solutions
of the perfectly plastic system
 arising from a  more general viscous correction $\delta_\spz: \spz \times \spz \to [0,\infty]$, provided that it fulfills  the compatibility condition
\begin{equation}
\label{compatib-pp}
\lim_{\tilde p \to p \text{ strongly in } \spz} \frac{\delta_\spz(p,\tilde p)}{\calR(\tilde p-p)} =0\,.
\end{equation}
Note that \eqref{compatib-pp} is a strengthened version of \eqref{will-be-checked}, in turn implying  $<B.3>$.
As a matter of fact,  \eqref{compatib-pp} guarantees the analogue of  \eqref{used-here-pp}, and then  the proof of  Proposition \ref{prop:VE2E-stab} still goes through. This is sufficient to extend the proof of Thm.\ \ref{thm:pp}.
\end{remark}

\section{Visco-Energetic solutions for a  damage system}
\label{s:3}
We consider a rate-independent damage  process in a nonlinearly elastic material, located in a bounded Lipschitz domain $\Omega \subset \R^d$. The body is subject to a time-dependent external force 
and it is clamped on a portion $\Gdir$ of its boundary $\partial\Omega$, fulfilling $\mathscr{H}^{d-1}(\Gdir)>0$.
Hence,  on $\Gdir$  the displacement field $\tilde u:(0,T) \times \Omega \to \R^d$ is prescribed by 
the time-dependent Dirichlet condition
\begin{equation}
\label{Dir}
\tilde u(t) = w_\Dir(t) \quad \text{on } \Gdir, \ t \in (0,T).
\end{equation}
From now on, 
as in Sec.\ \ref{s:perfect-plas} 
we will use the splitting $\tilde u = u +w_\Dir$
with $u=0 $ on $  \Gdir$ and, with slight abuse of notation,  
$w_\Dir$ the extension of the Dirichlet datum 
into the domain 
 $\Omega$. 
 The state variables of the damage  process will thus be
  $u $  and  a scalar damage variable $z: (0,T) \times \Omega \to \R$, 
  with values 
  in
  the interval $[0,1]$, such that
 $z(t, x) = 1 $ means no damage and $z(t, x) = 0$ means maximal damage in the neighborhood of the point $x \in \Omega$,   at the process  time
$t\in [0,T]$.
 \par
 We will confine the discussion to a \emph{gradient theory} for damage, thus accounting for an internal length scale. Namely, we  allow for the gradient regularizing contribution
$\int_\Omega |\nabla z|^r \dd x $
  to the driving energy, along the footsteps of \cite{MiRou06, ThoMie09DNEM,Thom11QEBV} analyzing Energetic solutions.
 More precisely, 
 the condition  $r>d$ imposed in \cite{MiRou06}
 was weakened to $r>1$ in \cite{ThoMie09DNEM} and, further, to $r=1$ (i.e.\ a $\mathrm{BV}$-gradient) in \cite{Thom11QEBV}. Here we will stay with the case $r>1$,  possibly 
 strengthening this condition to $r>d$ when considering a viscous correction that involves a norm different from that of the rate-independent dissipation potential,
 cf.\ Thm.\ \ref{thm:VEdam} ahead. 
 \par
 All in all, we consider the rate-independent  PDE system for damage
 \begin{equation}
 \label{RIS-damage}
 \begin{aligned}
 &
 -\mathrm{div}(\mathrm{D}_e W(x,\eps(u{+}w_\Dir),z)) = f && \text{ in } \Omega \times (0,T),
 \\
 & 
 \partial\mathrm{R}(x,\dot z) -\Delta_r z +\partial I_{[0,1]}(z) \ni - \mathrm{D}_z W(x,\eps(u{+}w_\Dir),z)  && \text{ in } \Omega \times (0,T),
 \end{aligned}
 \end{equation}
 supplemented with the homogeneous Dirichlet condition $u=0$  on $  \Gdir$, with the Neumann boundary conditions $\eps (u{+}w_\Dir) \nu =g$ on $\Gneu= \partial\Omega \setminus \Gdir $ (where
  $\nu$ is the exterior unit normal to $\partial\Omega$), 
  and $\partial_\nu z =0$ on $\partial\Omega$. 
 The conditions on the elastic energy density $W = W(x,e,z)$ (whose
G\^ateau derivatives  w.r.t.\ $e$ and $z$ are denoted by $\mathrm{D}_e $ and  $\mathrm{D}_z $, respectively),  and on the body and surface forces $f$, $g$ will be specified in \eqref{Wdam} and \eqref{data-dam} ahead;  
  $-\Delta_r  $ is the $r$-Laplacian operator and $\partial I_{[0,1]}: \R \rightrightarrows \R$ is the subdifferential of the indicator function $I_{[0,1]}$, enforcing the constraint $0\leq z \leq 1$ a.e.\ in $\Omega$. 
  The  dissipation potential $\mathrm{R}:\Omega \times \R\to [0,\infty]$ is given by
  \begin{equation}
  \label{diss-dam}
  \mathrm{R}(x,v): = \begin{cases}
  \kappa(x)|v| & \text{if } v \leq 0,
  \\
  \infty & \text{otherwise}
  \end{cases}
  \quad \text{with }  \kappa \in L^\infty(\Omega),  \ 0 <\kappa_0<\kappa(x) \ \foraa\, x \in \Omega\,.
  \end{equation}
\par 
  The Energetic formulation of the damage system \eqref{RIS-damage} is given in the following setup:
\paragraph{\bf Ambient space:} $\Xs=\spu \times \spz $ with 
\begin{subequations}
\label{damage-setup}
\begin{equation}
\label{sp-dam}
\spu = W_{\Gdir}^{1,p}(\Omega;\R^d): = \{ u\in W^{1,p}(\Omega;\R^d)\, : \ u=0 \text{ on } \Gdir \} \text{ and } \spz = W^{1,r}(\Omega)\,.
\end{equation}
Here, $p$ is as in \eqref{Wdam-3} below, and $ r >1$.
  The   topology $\sigma_\spu$ on the space of   admissible displacements is  the weak topology of $W^{1,p}(\Omega;\R^d)$; analogously, $ \sigma_\spz$ is  the weak 
$W^{1,r}(\Omega)$-topology. 
\paragraph{\bf  Energy functional:}
 $\calE: [0,T]\times \Xs \to (-\infty, \infty]$ is given by the sum of 
(1) the stored elastic energy $\calW$; (2)  a term $\calJ$ encompassing the gradient regularization and the indicator term $I_{[0,1]}(z)$;
 (3) the power of the external loadings, with the force term $\ell$ comprising volume and surface forces $f$ and $g$
 via 
 \[
 \pairing{}{}{\ell(t)}v: = \pairing{}{}{f(t)}v +  \pairing{}{}{g(t)}v,
 \]
where the duality pairings involving the forces $f$ and $g$ are nor specified for simplicity, and the   duality pairing  between $\ell$ and 
$u+w_\Dir(t)$ will be settled below, cf.\eqref{data-dam}.
  Namely, $\calE$ is defined by
\begin{align}
\label{en-dam}
\begin{aligned}
&
\enet tuz: = \mathcal{W}(t,u,z) + \mathcal{J}(z)   - \langle \ell(t), u+w_\Dir(t) \rangle \text{ with }
\begin{cases}
 \mathcal{W}(t,u,z) : = \int_\Omega W(x,\eps(u) + \eps(w_\Dir(t)), z) \dd x,
 \\
 \mathcal{J}(z): = \int_\Omega \left( \tfrac1{r}|\nabla z|^r +  I_{[0,1]}(z) \right) \dd x, \quad  r>1.
 \end{cases}
\end{aligned}
\end{align}
Then,
\[
 \domene {z} : = \{ z \in W^{1,r}(\Omega)\, : \ z(x) \in [0,1] \ \foraa\,x \in \Omega \}\,.
\] 
\paragraph{\bf Dissipation distance:}
We consider the asymmetric extended  quasi-distance $\mdsn{\spz}: \spz \times \spz \to [0,\infty]$ defined by 
\begin{equation}
\label{d-dam}
\mds{\spz}{z}{z'} : = \mathcal{R}(z'-z)  \quad  \text{ with }  \mathcal{R}:L^1(\Omega) \to [0,\infty], \  \calR(\zeta) : = \int_\Omega \mathrm{R}(x,\zeta(x)) \dd x \,.
\end{equation}
\end{subequations}
\par
Along the footsteps of \cite{ThoMie09DNEM}, for the elastic energy density  $W$ we assume 
\begin{subequations}
\label{Wdam}
\begin{align}
&
\label{Wdam-1}
W(x,\cdot,\cdot) \in \mathrm{C}^0(\bbM_{\mathrm{sym}}^{d\times d} \times \R) \quad \foraa\, x \in \Omega, \quad W(\cdot, e,z) \text{ measurable on } \Omega \quad \text{for all } (e,z) \in \bbM_{\mathrm{sym}}^{d\times d} \times \R;
\\
&
\label{Wdam-2}
W(x,\cdot, z) \text{ is convex for every } (x,z)\in \Omega \times \R;
\\
&
\label{Wdam-3}
\exists\, c_1,\, C_1>0 \ \exists\, p \in (1,\infty) \ \forall\, (x,e,z) \in \Omega \times \bbM_{\mathrm{sym}}^{d\times d} \times \R \, : \qquad W(x,e,z) \geq c_1|e|^p - C_1;
\\
&
\label{Wdam-4}
\begin{aligned}
&
\text{for all } (x,z)\in \Omega \times [0,1] \text{ we have } W(x,\cdot,z) \in \mathrm{C}^1(\bbM_{\mathrm{sym}}^{d\times d}) \text{ and }
\\
&\qquad \exists\, c_2, \, C_2>0 \ \forall\, 
 (x,e,z) \in \Omega \times \bbM_{\mathrm{sym}}^{d\times d} \times \R \, : \qquad |\mathrm{D}_e W(x,e,z)| \leq c_2 (W(x,e,z)+C_2);
 \end{aligned}
 \\
 &
\label{Wdam-5}
\begin{aligned}
&
\exists\, c_3,\, C_3>0 \ \forall\, (x,e,z),\, (x,e,\tilde z) \in \Omega \times \bbM_{\mathrm{sym}}^{d\times d} \times \R \text{ with } \tilde z \leq z \, \text{ there holds}
\\
&
\qquad W(x,e,z) \leq W(x,e,\tilde z) \leq c_3(W(x,e,z)+C_3)\,.
\end{aligned}
\end{align}
\end{subequations}
While referring to \cite[Sec.\ 3]{ThoMie09DNEM} for all details, here we may comment  that \eqref{Wdam-4}  enters in the proof of the power-control condition $<A.3>$ for the energy functional $\calE$ \eqref{en-dam}, whereas the `monotonicity' type requirement \eqref{Wdam-5} is helpful for the closedness condition $<C>$. 
As for the data $\ell $ and $w_\Dir$, we require
\begin{subequations}
\label{data-dam}
\begin{align}
&
w_\Dir \in \mathrm{C}^1 ([0,T];W^{1,\infty}(\Omega;\R^d));
\\
&
\ell \in \mathrm{C}^1 ([0,T];W^{-1,p'}(\Omega;\R^d)),
\end{align}
so that 
the power of the external loadings 
 features the duality pairing between $W_{\Gdir}^{-1,p'}(\Omega;\R^d)$ and $W_{\Gdir}^{1,p}(\Omega;\R^d)$. 
\end{subequations}
\paragraph{\bf The viscous correction:}
We will either take a viscous correction 
  of the form 
  \begin{subequations}
  \label{delta-damage}
  \begin{equation}
  \label{delta-dam-easy}
  \corrs \spz z{z'} = h(\calR(z'{-}z)),
  \end{equation} 
  with $h$ as in \eqref{a-funct-of-d}, or consider
 the viscous correction
\begin{equation}
\label{delta-dam}
\corrn_{\spz}: \spz \times \spz \to [0,\infty] \text{ defined by } \corrs \spz z{z'} : =
\begin{cases}
 \frac1q \| z-z'\|_{L^q(\Omega)}^q  & \text{ if } z,\, z' \in L^q(\Omega),
 \\
 \infty &\text{ otherwise},
 \end{cases}
 \quad \text{and } q>1
 \end{equation}
\end{subequations}
(cf.\ also Remark \ref{rmk:weaker-r-d} ahead). 
\par
The main result of this section guarantees the existence of $\VE$ solutions of the rate-independent damage system $\RIS$ given by
\eqref{damage-setup}. 
\begin{theorem}
\label{thm:VEdam}
Assume \eqref{diss-dam}, \eqref{Wdam}, and \eqref{data-dam}. 
If the viscous correction 
$\corrn_{\spz}$ is given by \eqref{delta-dam}, suppose in addition that
\begin{equation}
\label{rgreatedd}
r>d\,.
\end{equation}
%
Then, for every $z_0 \in \domene z $ there exists a $\VE$ solution $(u,z)$ of the rate-independent damage system $\RIS$ \eqref{damage-setup} with  the viscous correction 
$\corrn_{\spz}$ from \eqref{delta-damage}, such that   $z(0)=z_0$ and 
\begin{equation}
\label{summab-u-z}
u\in L^\infty(0,T;W^{1,p}(\Omega;\R^d)), \quad  z \in 
L^\infty(0,T;W^{1,r}(\Omega)) \cap \BV ([0,T];L^1(\Omega))\,.
\end{equation}
\end{theorem}
\noindent The \emph{proof} will be carried out in Sec.\ \ref{ss:3.1} below.
\begin{remark}
\label{rmk:weaker-r-d}
\upshape
The condition $r>d$ can be weakened to the   requirement
\begin{equation}
\label{compatibility-below}
r >\frac{qd}{q+d}\,,
\end{equation} 
on  $r$, $q$, and the space dimension $d$, provided that we replace the viscous correction \eqref{delta-dam} by 
 \begin{equation}
 \label{altern-corr}
 \tilde{\delta}_\spz : \spz \times \spz \to [0,\infty] \text{ given by } 
 \tilde{\delta}_{\spz}(z,z'): = 
 \begin{cases}
 \frac1q \| z-z'\|_{L^q(\Omega)}^\gamma  & \text{ if } z,\, z' \in L^q(\Omega),
 \\
 \infty &\text{ otherwise}
 \end{cases}
 \end{equation}
and $\gamma>1$ satisfying a further compatibility condition with  $r$ and $q$, cf.\ \eqref{compatible-exps} in Remark \ref{rmk:gamma-techn} ahead.
\end{remark}
\begin{remark}[$\VE$ solutions are in between $\ENE$ and $\BV$ solutions (I)]
\upshape
\label{rmk:inbetw-dam}
The application of the $\VE$ concept to damage reveals
that this weak solvability notion  has an
 \emph{intermediate} character  between Energetic  and Balanced Viscosity solutions.
Indeed,
\begin{compactitem}
\item[-] When the viscous correction is given by \eqref{delta-dam-easy}, then the existence theory  for $\VE$-solutions works under the same conditions as for $\ENE$ solutions, cf.\ \cite{ThoMie09DNEM}. In particular, it is possible to consider 
a gradient regularization with an \emph{arbitrary} exponent $r>1$; the restriction $r>d$ (or \eqref{compatibility-below}) comes into 
play only upon choosing the viscous correction \eqref{delta-dam} (or \eqref{altern-corr}).
\item[-] Balanced Viscosity solutions to the rate-independent system
\eqref{RIS-damage} have   been in turn   addressed in \cite{KnRoZa17}, with a \emph{quadratic} viscous regularization (modulated by a vanishing parameter). 
The vanishing-viscosity analysis developed in   \cite{KnRoZa17}  crucially relies on the requirement  $r>d$ and, additionally, on the \emph{quadratic} character of the elastic energy density $W$, as well as 
 on smoothness requirements on the reference domain $\Omega$ (the smoothness of $\Omega$ can be dropped if 
 the   nonlinear $r$-Laplacian is replaced by a less standard fractional Laplacian regularization, cf.\  \cite{KnRoZa13VVAR}). Here, instead, we can allow for  an energy density $W$ of arbitrary $p$-growth and we do not need to restrict to smooth domains.
\end{compactitem}
\end{remark}
\subsection{Proof of Theorem \ref{thm:VEdam}}
\label{ss:3.1}
In what follows, we are going to check that the  rate-independent damage system $\RIS$ given by 
\eqref{damage-setup} complies with Assumptions 
$<A>$,  $<B>$ ,  and  $<C>$ of Theorem \ref{thm:exists-VE} (it is immediate to see that $<T>$ is satisfied).
As it will be clear from the ensuing proof, $<A>$  and  $<C>$  can be checked under the sole condition that the exponent $r$ is strictly bigger than $1$.  It is in the proof of $<B>$, in the case  
the viscous correction  $\delta$ is   given by 
\eqref{delta-dam}, that  the restriction $r>d$ comes into play.
\paragraph{\bf $\vartriangleright $ Assumption $<A>$:}
It was shown 
in \cite[Lemma 3.3]{ThoMie09DNEM} that $\calE$  satisfies the coercivity estimate
\begin{equation}
\label{coerc-dam}
\exists\, c_4,\, C_4>0 \ \forall\, (t,u,z) \in [0,T]\times \spu \times \spz \, : \quad \enet tuz \geq c_4  (\|u\|_{W^{1,p}(\Omega;\R^d)}^p + \|z\|_{W^{1,r}(\Omega)}^r  ) -C_4\,.
\end{equation}
Hence, the  sublevels  of $\enet t{\cdot}{\cdot}$ are bounded in $W^{1,p}(\Omega;\R^d)\times W^{1,r}(\Omega)$, uniformly w.r.t.\ $t\in[0,T]$.
In \cite[Lemma 3.4]{ThoMie09DNEM}  it was proved that $\calE(t,\cdot,\cdot)$ is sequentially lower semicontinuous w.r.t.\ the weak topology on 
 $W^{1,p}(\Omega;\R^d)\times W^{1,r}(\Omega)$.
 In view of  \eqref{data-dam},
 a standard modification of that argument yields the lower semicontinuity of $\calE$, hence $<A.1>$. 
  Therefore, its sublevels are (sequentially)
compact in $[0,T]\times \spu\times \spz$
 w.r.t.\ to the $\sigma_\R$-topology.  This ensures the validity of  $<A.2>$ .
 \par
 It was shown in \cite[Thm.\ 3.7]{ThoMie09DNEM}  that  there exist constants $c_5,\, C_5>0$ such that for all $(u,z) \in [0,T]\times \spu 
 \times \domene z  $ the function  $t\mapsto \enet tuz $ belongs to $\mathrm{C}^1([0,T])$, with
 \[
 \begin{aligned}
 &
 \partial_t \enet tuz= \int_\Omega  \mathrm{D}_e W(x,\eps(u+w_\Dir(t)),z) \colon \eps(\dot{w}_\Dir(t)) \dd x  - \langle \dot{\ell}(t), u+w_\Dir(t) \rangle - \langle \ell(t), \dot{u}_{\mathrm{D}}(t) \rangle \quad \text{and}
 \\
 & 
|\partial_t \enet tuz| \leq c_5(\enet tuz +C_5)\quad \text{for all } (t,u,z) \in (0,T) \times \domene{},
\end{aligned}
\]
whence 
 \eqref{power-control}. 
We now check $<A.3'>$: observe that $\mdsn{\spz}$ is left-continuous on the sublevels of $\calF_0$ since the latter subsets are bounded in $W^{1,r}(\Omega)$ by \eqref{coerc-dam} and $W^{1,r}(\Omega)\Subset L^1(\Omega)$. 
It remains to prove the conditional upper semicontinuity
\eqref{usc-power}
 of $\partial_t \calE$. For this, we apply \cite[Lemma 3.11]{ThoMie09DNEM}, ensuring that $\partial_t \calE$ complies with 
 \eqref{unif-cont}. Then, we are in a position to apply Proposition  \ref{prop:GF} and conclude the validity of property 
\eqref{usc-power}. 
\paragraph{\bf $\vartriangleright $ Assumption   $<C>$:}
We will verify property \eqref{MRS}
in the case of the viscous correction \eqref{delta-dam} (the case \eqref{delta-dam-easy}  can be handled with similar calculations). 
 Let $(t_n,u_n,z_n)_n$, $(t,u,z)$ fulfill the conditions of \eqref{MRS}, and let  $(u',z')$ be \emph{any} element in $M(t,z)$.
Preliminarily,
from $\sup_{n\in \N} \enet {t_n}{u_n}{z_n} <\infty$
we deduce, via \eqref{coerc-dam}, that the sequence $(z_n)_n$ is bounded in 
$W^{1,r}(\Omega)$ and, thus, that 
$z_n\weakto z$ in $W^{1,r}(\Omega)$ as $n\to\infty$. Since $0\leq z_n \leq	1$ a.e.\ in $\Omega$, 
we then infer that 
$z_n\to z$ in $L^s(\Omega)$ for all $s\in [1,\infty)$. 
 For the sequence $(u_n',z_n')_n$ we borrow the construction  for the \emph{mutual recovery sequence} devised in the proof of \cite[Thm.\ 3.14]{ThoMie09DNEM}.
  Note that this construction is in fact  applicable to \emph{any}  $(u',z') \in \spu \times \domene z$  such that $\mathcal{R}(z'{-}z)<\infty$. In particular, we pick $u'\in \mathrm{Argmin}_{u\in \spu} \enet tu{z'}$.
Namely, we set for every $n\in \N$
\begin{equation}
\label{mrs-marita}
\begin{aligned}
&
u_n' := u'
\\
&
z_n':= \min\{(z'-\delta_n)^+, z_n\} = \begin{cases}
(z'-\delta_n)^+ & \text{if } (z'-\delta_n)^+\leq z_n,
\\
z_n & \text{if } (z'-\delta_n)^+> z_n,
\end{cases}
\quad \text{with } \delta_n : = \|z_n-z\|_{L^r(\Omega)}^{1/r}\to 0 \text{ as } n \to\infty. 
\end{aligned}
\end{equation}
Observe that this construction gives
$z_n' \in W^{1,r}(\Omega)$ as well as
 $0\leq z_n'\leq z_n \leq 1$ a.e.\ in $\Omega$, so that $\calR(z_n'-z_n)<\infty$.
In the proof of \cite[Thm.\ 3.14]{ThoMie09DNEM} it is shown that 
\begin{equation}
\label{strong-w1r}
z_n' \rightharpoonup z' \quad \text{in } W^{1,r}(\Omega) \quad \text{as } n\to\infty\,.
\end{equation}
Slightly adapting the argument from  \cite[Thm.\ 3.14]{ThoMie09DNEM} to allow for a sequence $(t_n)_n$ of times converging to $t$, we find that
\[
\limsup_{n\to\infty} \left( \enet {t_n}{u_n'}{z_n'} + \mds{\spz}{z_n}{z_n'}  \right) 
\leq
\enet {t}{u'}{z'} + \mds{\spz}{z}{z'}\,.  
\]
Therefore, 
for the reduced energy $\calI(t,z) = \min_{u \in \spu} \enet tuz$
we deduce 
\begin{equation}
\label{1limsup}
\limsup_{n\to\infty} \left( \red {t_n}{z_n'} + \mds{\spz}{z_n}{z_n'}  \right) 
\leq
\red {t}{z'} + \mds{\spz}{z}{z'},
\end{equation}
 where we have used that $\calI(t_n,z_n') \leq  \enet {t_n}{u_n'}{z_n'}   $ and that 
$\calI(t,z') =\enet t{u'}{z'} $  by our choice of $u'$.  
On other other hand,  again using that $0\leq	z_n'\leq 1 $ a.e.\ in $\Omega$, from 
\eqref{strong-w1r} we infer that $z_n' \to z'$ in $ L^s(\Omega)$ for every $s\in [1,\infty)$. All in all,
 we  gather that $z_n\to z$  and $z_n'\to z'$ in $ L^q(\Omega)$. 
 Therefore,
 \begin{equation}
\label{2limsup}
 \lim_{n\to\infty}\corrs{\spz}{z_n}{z_n'} = \corrs{\spz}{z}{z'}\,
 \end{equation}
 which, combined with 
\eqref{1limsup}, finishes the proof of property \eqref{MRS}.
\paragraph{\bf $\vartriangleright $ Assumption   $<B>$:}
The viscous correction $\delta $ from \eqref{delta-dam} clearly complies with $<B.1>$ and $<B.2>$. To check $<B.3>$, we verify property \eqref{will-be-checked}.  
Preliminarily, observe that 
with the Gagliardo-Nirenberg inequality we have
\begin{equation}
\label{GN-est}
\frac{\delta_\spz (z',z)}{\mds{\spz}{z'}z} = \frac1q  \frac{ \| z{-}z'\|_{L^{q}(\Omega)}^q}{ \| z{-}z'\|_{L^{1}(\Omega)}} \leq C \frac{ \| z{-}z'\|_{W^{1,r}(\Omega)}^{\theta q}  \| z{-}z'\|_{L^{1}(\Omega)}^{(1{-}\theta) q }}
{ \| z{-}z'\|_{L^{1}(\Omega)}} 
\end{equation}
with
\begin{equation}
\label{GN}
\frac1q =\theta\left(\frac1r-\frac1d \right) + 1-\theta\,.
\end{equation}
Since $r>d$, 
there exists $\theta \in (0,1)$ complying with \eqref{GN}. 
\par Let us now consider $(t,z) \in \mathscr{S}_{\cmdn}$ and a sequence $(t_n,z_n)\widetilde{\rightharpoonup}(t,z) $, namely
$(t_n,z_n)_n
\subset \mathscr{S}_{\cmdn}$, 
$t_n \uparrow t$, 
$
z_n \weakto z$ in $W^{1,r}(\Omega)$,  $\calR(z_n{-}z)\to 0$.
Then, 
\begin{equation}
\label{W1r-bounded}
\sup_{n\in \N} \|z_n \|_{W^{1,r}(\Omega)} \leq C.
\end{equation}
Therefore,
\[
\limsup_{(t_n,z_n)\widetilde{\rightharpoonup}(t,z) } \frac{\delta_\spz (z_n,z)}{\mds{\spz}{z_n}z} \stackrel{(1)}{\leq} C \limsup_{(t_n,z_n)\widetilde{\rightharpoonup}(t,z)}    \frac{ \| z{-}z_n\|_{W^{1,r}(\Omega)}^{\theta q}  \| z{-}z_n\|_{L^{1}(\Omega)}^{(1{-}\theta) q }}
{ \| z{-}z_n\|_{L^{1}(\Omega)}}   \stackrel{(2)}{\leq}  C' \limsup_{(t_n,z_n)\widetilde{\rightharpoonup}(t,z)}     \| z{-}z_n\|_{L^{1}(\Omega)}^{(1{-}\theta) q - 1 }
  \stackrel{(3)}{=} 0,
\]
where {\footnotesize (1)} follows from \eqref{GN-est}, {\footnotesize (2)} from \eqref{W1r-bounded}, and  {\footnotesize (3)} from the fact that, since $r>d$, the exponent $\theta$  in \eqref{GN} fulfills $(1{-}\theta) q >1$.
This finishes the proof of  \eqref{will-be-checked}.
\paragraph{\bf Conclusion of  the proof:} Theorem \ref{thm:exists-VE} applies, yielding the existence of a $\VE$ solution.  The summability properties \eqref{summab-u-z} for $u$ and $z$ follow from combining the coercivity property 
\eqref{coerc-dam}
 with the energy bound $\sup_{t \in [0,T]} |\enet t{u(t)}{z(t)}| \leq C$, cf.\ \eqref{sup-perto} in Remark \ref{rmk:sth-about-VE}.
\QED
\begin{remark}
\upshape
\label{rmk:gamma-techn}
Observe that  \eqref{compatibility-below} is the sharpest condition ensuring that $\theta $ given by \eqref{GN} is in $(0,1)$.
 The requirement $r>d$ can be weakened to    \eqref{compatibility-below}, provided that we replace the viscous correction $\delta$ from \eqref{delta-dam} by that in \eqref{altern-corr}, with $\gamma>1$ chosen in such a way that  $\theta $ from \eqref{GN} fulfills
$(1-\theta)\gamma>1$. 
This amounts to imposing the following 
 condition on $\gamma$
\begin{equation}
\label{compatible-exps}
\gamma \left( \frac1q {-} \left( 1{-}\frac1q \right) \frac{d{-}r}{dr{+}r{-}d}  \right)>1.
\end{equation}
For instance,  if $d=3$ and $r=2$ (i.e.\ we consider the standard Laplacian regularization), 
then $q=2$ complies with the compatibility condition  \eqref{compatibility-below}. An admissible viscous correction would then be 
\[
\delta_{\spz}(z,z'): = \frac12 \| z'{-}z\|_{L^2(\Omega)}^\gamma \quad \text{with } \gamma>\frac52.
\]
\end{remark}
\section{Visco-Energetic solutions for plasticity at finite strains}
\label{s:4}
\noindent
We consider a  model for elastoplasticity  at finite strains in a bounded  body $\Omega \subset \R^d$  with Lipschitz boundary.
Finite 
plasticity is based on the multiplicative decomposition of the
gradient of the 
 elastic deformation $\varphi: \Omega \to \R^d$ into an elastic and a plastic part, i.e.\
$\nabla \varphi = F_{\mathrm{el}} P$
with $P \in \R^{d\times d}$ the plastic tensor, usually assumed with determinant $\mathrm{det}(P)=1$.  While the elastic part $F_{\mathrm{el}} =\nabla \varphi P^{-1}$ contributes to  energy storage and is at elastic equilibrium, energy is dissipated through changes of the plastic tensor, which  thus plays the role of a  (dissipative)  internal variable.
\par
The model for  rate-independent finite-strain plasticity we address was  first analyzed in \cite{MaiMi08} within the framework of energetic solutions.
The PDE system in the unknowns $(\varphi, P)$ can be formally written as 
\begin{subequations}
\label{PDE-finite-plast}
\begin{equation}
\label{PDE-finite-plast-syst}
\begin{aligned}
& 
\varphi(t) \in \mathrm{Argmin} \left( 
\int_{\Omega} W(x,\nabla \hat{\varphi}P^{-1}(t)) \dd x - \langle \ell(t), \hat{\varphi} \rangle\, : \ \hat\varphi \in \mathscr{F}  
\right),  &&\quad t \in (0,T),
\\
&
\partial\mathrm{R}(\dot P P^{-1}) \transpi{P}    + \btransp{\nabla \varphi P^{-1}} \mathrm{D}_{F} W(x,\nabla \varphi P^{-1}) \transpi{P}    && 
\\
& 
 \quad \quad \quad \quad \quad  \quad \quad \quad + \rmD_{P}H(x,P,\nabla P) -\mathrm{div}(\rmD_{\nabla P} H(x,P,\nabla P)) =0,  && \quad  (x,t)\in \Omega \times (0,T)\,.
\end{aligned}
\end{equation}
Here, $W = W(x,F)$ is  the elastic energy density, $\ell$ is  a time-dependent loading, e.g.\ associated with an  applied body force $f$ and a traction $g$ on the Neumann part $\Gneu$ of $\partial\Omega$, $\mathscr{F}  $ is  the set of admissible deformations (cf.\ \eqref{time-dependend-Dir-plast} below), 
  the dissipation potential $\mathrm{R}(x,\cdot)$ is 
$1$-homogeneous, and the energy density $H$ encompasses hardening and regularizing effects through the term 
 $\int_\Omega |\nabla P|^r \dd x $, for some $r>1$ specified later.
 System \eqref{PDE-finite-plast} is  further supplemented with 
  a time-dependent Dirichlet condition for $\varphi$
\begin{equation}
\label{time-dependend-Dir-plast}
\varphi(t,x) = \gdir (t,x) \quad (t,x) \in [0,T] \times \Gdir,
\end{equation}
with
$\gdir :[0,T] \times \Gdir \to \R^d$ given on the  Dirichlet boundary
$\Gdir \subset \partial\Omega$ such that  $\mathscr{H}^{d-1}(\Gdir)>0$. 
\end{subequations}
Following
\cite{FraMie06ERCR, MaiMi08}, to treat  \eqref{time-dependend-Dir-plast} 
compatibly with the multiplicative decomposition of $\nabla \varphi$,
we  will  seek for $\varphi$ in the form of a composition
\begin{equation}
 \label{multiplic-split}
 \varphi(t,x) = \gdir (t, y(t,x)) \quad \text{with } y(t,\cdot) \text{ fulfilling  } y =\mathrm{Id}
 \text{ on } \Gdir,
\end{equation}
where we have denoted by the same symbol the extension of $\gdir$ 
 to $[0,T]\times
\R^d$, cf.\ \eqref{conds-g} below.  
\par
Therefore, we  consider the pair $(y,P)$ as state variables and, accordingly, the  Energetic formulation  of system
\eqref{PDE-finite-plast} is given in the following setup:
\paragraph{\bf Ambient space:} we  take $X=U\times Z$, with 
\begin{subequations}
\label{setup-plast}
\begin{equation}
\label{sp-plast}
\begin{gathered}
\spu := \left\{ y \in
W^{1,q_{Y}}(\Omega;\R^d)\, : \ y = \mathrm{Id}  \ \text{on
$\Gamma_{\Dir}$} \right\} \ \ \text{for  $q_Y>1$ to be
specified later, and }
\\
\spz = \{ P \in  W^{1, r}(\Omega ;\R^{d\times d}) \cap L^{q_P}(\Omega;\R^{d\times d}) \, : \, P(x) \in G \quad \foraa\, x \in \Omega \}, \quad  q_P,\, r>1 \text{ specified below}. 
\end{gathered}
\end{equation}
Here, $G$ is a  Lie subgroup of 
$ \mathrm{GL}^+(d): =\{ P\in \R^{d\times d})\, : \ \mathrm{det}(P)>0\} $.
From now on, we will focus on the case
\[
G = \mathrm{SL}(d): =\{ P\in \R^{d\times d})\, : \ \mathrm{det}(P)=1\}
\]
cf.\ \cite{Mie02} for other examples of $G$.
We take $\sigma_\spu$ as  the weak topology of $W^{1,q_{Y}}(\Omega;\R^d)$ and $\sigma_\spz$ as  the weak topology of $W^{1, r}(\Omega ;\R^{d\times d}) \cap L^{q_P}(\Omega;\R^{d\times d}) $.
\paragraph{\bf Energy functional:} 
    $\calE: [0,T]\times \Xs \to (-\infty, \infty]$ is given by 
\begin{equation}
\label{en-plast}
\enet t{y}P: = \calE_1(P) + \calE_2(t,y,P).
\end{equation}
The functional
$\calE_1: [0,T]\times \spz \to \R$ includes the hardening and gradient regularizing terms, i.e.
\[
\calE_1(P) = \int_\Omega H(x,P(x),\nabla P(x)) \dd x \text{ with } H:\Omega \times  \R^{d\times d} \times  \R^{d\times d \times d}\to \R    \text{ fulfilling \eqref{hyp-e-plast} below}.
\]
The stored elastic energy $\calE_2$ reflects the  multiplicative split
for the deformation gradient
 $\nabla \varphi = \nabla \gdir (t,y) \nabla y $ due to \eqref{multiplic-split}, and it is thus of the form 
\[
\calE_2(t,y,P) := \int_\Omega W(x,\nabla \gdir(t,y) \nabla y P^{-1} ) \dd x  - \pairing{}{W^{1,q_Y}}{\ell(t)}{\gdir(t,y)},
\]
with the elastic energy density $W$ specified ahead and $\nabla \gdir$ the gradient of $\gdir$ w.r.t.\ the variable $y$. 
\paragraph{\bf  Dissipation distance:}
Along the footsteps of   \cite{MaiMi08} (cf.\ also \cite{Mie02,HMM03}), we consider on $X$ dissipation distances 
of the form 
\begin{equation}
\label{d-plast}
\mathsf{d}_\spz(P_0,P_1): = \int_\Omega \mathcal{R}(P_1(x)P_0^{-1}(x)) \dd x,
\end{equation}
 \end{subequations}
where the functional $\mathcal{R} :  \SLD \to [0,\infty)$
(for simplicity, we omit the possible $x$-dependence of $\mathcal{R}_1$)
 is generated by a norm-like function $\mathrm{R}$, cf.\ \eqref{hyp-R-plast} below, 
  on the Lie-algebra $T_{\bf 1} \SLD$
 via the formula
\[
\mathcal{R}(\Sigma): = \inf \left\{ \int_0^1 \mathrm{R}(\dot{\Xi}(s)\Xi(s)^{-1})\dd s \, : \ \Xi \in \mathrm{C}^1([0,1]; G), \ \Xi(0)=\mathbf{1}, \ \Xi(1) = \Sigma \right\}.
\]
\par
Let us now detail our assumptions on the constitutive functions $H$ and $W$,  on
$\mathrm{R}$, and on 
 the problem data. 
\begin{subequations}
\label{hyp-e-plast}
The hardening function $H$ satisfies 
\begin{equation}
\begin{aligned}
& 
H:\Omega \times \R^{d\times d} \times \R^{d\times d \times d} \to \R \text{ is a normal integrand,  } H(x,P,\cdot) \text{ convex  for all } (x,P) \in  \Omega \times   \R^{d\times d}\,,
\\
&
\exists\, c_1>0\  \exists\, h \in L^1(\Omega)  \ \exists\, \, q_P>1, \, r>1\,   \ 
\foraa\, x \in \Omega   \ 
 \forall\, (P,A) \in \R^{d\times d} \times  \R^{d\times d \times d}\, :
 \\
 &  \qquad \qquad \qquad \qquad  \qquad \qquad H(x,P,A)\geq h(x) + c_1(|P|^{q_P} + |A|^{r})\,,
\end{aligned}
\end{equation}
while we require the following conditions on the elastic energy density $W: \Omega \times \R^{d\times } \to [0,\infty]$: Firstly,
\begin{align}
  \label{w0} & \dom(W) = \Omega \times \GLD,
  \text{ i.e. }  W(x,F) = \infty \text{ for } \mathrm{det}F \leq 0
  \text{ for all $x \in \Omega$},
  \\
\label{w1}  &
\begin{aligned}
&  \exists \, c_2>0 \ 
 \exists\, j \in L^1 (\Omega) \ 
\exists\, \qf> d  \ \   \forall (x,F) \in \dom(W):
\quad W(x,F) \geq j(x) +c_2 |F|^{\qf},
\end{aligned}
 \intertext{
 and we impose a further compatibility condition between the integrability powers $q_Y,\qf, q_P$, i.e.}
 &
\label{integrab-powers}
\frac1{q_F} + \frac1{q_P} = \frac1{q_Y}<\frac1d.
 \intertext{
 Secondly, $W(x,\cdot): \R^{d\times d} \to
(-\infty,\infty]$ is \emph{polyconvex}  for all $x \in
\Omega$, i.e.\ it is a convex function of its minors:
} & \label{w2}
\begin{aligned}
&
 \exists\, \mathbb W: \Omega \times \R^{\mu_d} \to
(-\infty,\infty] \ \text{such that}
\\ &
\text{(i)} \ \mathbb W \ \text{is  a normal integrand,}
\\
& \text{(ii)} \  \forall\, (x,F) \in \Omega \times \R^{d \times d}\,
:
 \ \ W(x,F)= \mathbb W(x,\mathbb M (F)),
\\
& \text{(iii)} \ \forall x\in \Omega\,:  \  \ \mathbb W(x,\cdot):
\R^{\mu_d} \to (-\infty,\infty] \text{ is convex,}
\end{aligned}
\intertext{where $\mathbb M : \R^{d\times d}\to R^{\mu_d}$ is  the
  function which maps a matrix to all its minors, with $\mu_d :=
  \sum_{s=1}^d \binom{d}{s}^2 $. Thirdly,  $W$  satisfies the multiplicative stress control conditions} & \label{w3}
\begin{aligned}
& \exists\, \delta>0 \    \ \exists\, c_3,\, c_4>0    \ \  \forall\,
(x,F) \in \dom(W) \ \forall\, N \in \mathcal{N}_\delta\,:
\\
&   \text{(i)} \  \text{$W(x,\cdot) : \mathrm{GL}^+(d) \to \R$ is differentiable,} 
\\
& \text{(ii)} \ |  \rmD_F W(x,F) \transp{F}  | \leq  c_3 (W(x,F) +1),
\\
& \text{(iii)} \ | \rmD_F W(x,F)\transp{F} -  \rmD_F
W(x,N F) \transp{(NF)} | \leq  c_4 |N -\mathbf{1}| (W(x,F) +1),
\end{aligned}
\end{align}
with  $\mathcal{N}_\delta: = \left\{N \in \R^{d \times d}\, : \
|N-\mathbf{1}| <\delta \right\}. $
\end{subequations}
We refer to \cite{MaiMi08} for examples of functionals $H$ and $W$ complying with \eqref{hyp-e-plast}. 
Finally, the functional (whose possible dependence on $x$ is neglected by simplicity) 
 \begin{equation}
 \label{hyp-R-plast}
 \begin{gathered}
\mathrm{R} : T_{\bf 1} \SLD \to [0,\infty)  \text{ is $1$-positively homogeneous and   fulfills} 
\\
\exists\, c_R,\, C_R>0 \ \forall\, \Sigma \in T_{\bf 1} \SLD \, : \quad
c_R|\Sigma | \leq  \mathrm{R}(\Sigma)  \leq C_R |\Sigma |\,,
\end{gathered}
\end{equation}
cf.\ \cite{HMM03} for examples in von-Mises  and single-crystal plasticity. 
For the Dirichlet loading  $\gdir$  we require 
\begin{subequations}
\label{data-plast}
\begin{equation}
 \label{conds-g}
 \begin{aligned}
 &
 \gdir \in \rmC^1 ([0,T]\times \R^d;\R^d), \quad \nabla \gdir \in
 \mathrm{BC}^1  ([0,T]\times \R^d;\R^{d\times d} ), \\  
 &
 \exists\, c_5 >0 \ \  \forall\, (t,x) \in [0,T]\times \R^d\,:  \ \ 
 \ |\nabla \gdir (t,x)^{-1}| \leq c_5,
 \end{aligned}
\end{equation} 
where $\mathrm{BC}$ stands for \emph{bounded continuous}. Finally, on the external load $\ell$ we impose  
\begin{equation}
\label{hyp-ell}
\ell \in \mathrm{C}^1([0,T]; W^{1,q_Y}(\Omega;\R^d)^*)\,.
\end{equation}
 \end{subequations}
\paragraph{\bf The viscous correction:}
We will    take  viscous corrections 
 \begin{compactenum}
 \item either of the form 
  \begin{subequations}
  \label{delta-plast}
  \begin{equation}
  \label{delta-plast-easy}
  \corrs \spz {P_0}{P_1} = h(\mathrm{d}_\spz(P_0,P_1)) \quad \text{with $h$ as in \eqref{a-funct-of-d}},
  \end{equation} 
  \item or
we  define 
   $\corrn_{\spz}: \spz \times \spz \to [0,\infty]$  by 
  \begin{equation}
\label{delta-dir}
  \corrs \spz {P_0}{P_1} : =
  \begin{cases}
 \int_\Omega \mathrm{R}_q ((P_1(x)-P_0(x))P_1(x)^{-1})  \dd x =   \int_\Omega \mathrm{R}_q (P_1(x)P_0(x)^{-1} - \mathbf{1}) \dd  x  & \text{if }   \mathrm{R}_q (P_1P_0^{-1} - \mathbf{1})  \in L^1(\Omega),
 \\
 \infty & \text{otherwise,}
 \end{cases}
\end{equation}
\end{subequations}
for a   given  convex lower semicontinuous 
 functional 
  $\mathrm{R}_q: T_{\bf 1} \SLD \to [0,\infty) $ 
  fulfilling 
\begin{equation}
\label{hyp-Rq}
\mathrm{R}_q(\Sigma) = \mathrm{R}_q({-}\Sigma)
\text{ for all }  \Sigma \in  T_{\bf 1} \SLD \quad 
 \text{ and } 
\lim_{\Sigma \to 0} \frac{ \mathrm{R}_q(\Sigma)}{|\Sigma|^q}  = C_q  \in (0,\infty) \text{ for some } q>1\,.
\end{equation}
\end{compactenum}
\par
For our existence result of $\VE$ solutions to the rate-independent system 
$\RIS$
from
\eqref{setup-plast}, 
 like for  the damage system in Sec.\ \ref{s:3} we  shall  strengthen the condition $r>1$
  to $r>d$ when addressing the non-trivial viscous correction \eqref{delta-dir}.
\begin{theorem}
\label{thm:VEplast}
Assume  \eqref{hyp-e-plast}, \eqref{hyp-R-plast},   and \eqref{data-plast}.  Furthermore, 
if the viscous correction 
$\corrn_{\spz}$ is given by \eqref{delta-dir}, suppose in addition that $r>d$. 
%
Then, for every $P_0 \in  \spz $ there exists a $\VE$ solution $(y,P)$
of the rate-independent finite-plasticity system $\RIS$ \eqref{setup-plast},
with  the viscous correction 
$\corrn_{\spz}$  from  \eqref{delta-plast}, such that 
 $P(0)=P_0$ and 
\begin{equation}
\label{summ-props-yP}
y\in L^\infty(0,T;W^{1,q_Y}(\Omega;\R^d)),  \quad  P \in 
L^\infty(0,T;W^{1,r}(\Omega;\R^{d\times d})) \cap \BV ([0,T];L^1(\Omega;\R^{d\times d}))\,.
\end{equation}
\end{theorem}
The \emph{proof} will be carried out in   Section \ref{ss:4.1} ahead.
\begin{remark}[Extensions]
\label{rmk:ext-plast}
\upshape
The model for finite plasticity considered in \cite{MaiMi08} is actually more general than that addressed
here, as it features a further internal variable $p\in \R^m$, $m\geq 1$, besides the plastic tensor $P$.
The vector $p$ possibly encompasses hardening variables/slip strains and, like $P$, it is subject to a gradient
regularization. Under the very same conditions as in \cite[Thm.\ 3.1]{MaiMi08}, it is possible to
show that the energy functional comprising $p$ 
complies with condition $<A>$ in the metric topological setup where
\[
Z= \left( L^{q_P}(\Omega;\R^{d\times d}){\cap} W^{1, r}(\Omega; ;\R^{d\times d}) \right) \times 
\left( L^{q_P}(\Omega;\R^{m}){\cap} W^{1, r}(\Omega; ;\R^{m}) \right)\,.
\]
A typical example where the 
additional variable $p$ comes into play is \emph{isotropic hardening}, cf.\ \cite[Example 3.3]{MaiMi08}.
There, the scalar $p\in\R$ measures the amount of hardening and the variables  $(P,p)$ are subject to some constraint. 
 The relevant dissipation distance 
 accounts for such constraint and takes $\infty$ as a 
value.
\par
Actually, our analysis could be extended to dissipation distances with values in $[0,\infty]$ under the very same conditions enucleated in 
\cite[formula (3.4)]{MaiMi08}. In particular, if we take the `trivial' viscous correction $\delta_\spz$ from
\eqref{delta-plast-easy}, then 
 the same argument as in \cite[Sec.\ 5.3]{MaiMi08}
allows us to check condition \eqref{MRS},
 whence the validity of assumption  $<C>$ of the general existence Thm.\ \ref{thm:exists-VE}.
With the viscous correction in 
\eqref{delta-plast-easy}
we can generalize our existence Thm.\ \ref{thm:VEplast}  for $\VE$ solutions also in the other directions outlined in 
 \cite[Sec.\ 6]{MaiMi08}.
\end{remark}
\begin{remark}[$\VE$ solutions are in between $\ENE$ and $\BV$ solutions (II)]
\label{rmk:intermediate-plast}
\upshape
The statement of Thm.\ \ref{thm:VEplast}, as well as Remark \ref{rmk:ext-plast}, highlight the fact that, in the 
case of the viscous correction \eqref{delta-plast-easy},
the existence theory for $\VE$ solutions to the finite-strain plasticity system
works under the very same conditions as for $\ENE$ solutions.
Nonetheless, when bringing into play a different viscous correction such as that in 
\eqref{delta-dir},  like for the damage system in Sec.\ \ref{s:3} we need to strengthen our conditions on the gradient regularization
and in fact impose $r>d.$ For $\ENE$ solutions to the finite plasticity system, this requirement was made only in the cases in which the dissipation distance
took values in $[0,\infty]$, cf.\ \cite{MaiMi08}. Instead, 
in the case of the viscous correction from \eqref{delta-dir}
we cannot weaken this condition even when $\mathsf{d}_\spz$ 
is valued in $[0,\infty)$, cf.\ also Remark \ref{rmk:cannot} ahead. 
\par
At any rate, the existence of $\VE$ solutions is proved here under weaker conditions than for $\BV$ solutions. 
Although the latter have not yet been addressed in the  context of finite plasticity, we may observe that a prerequisite
for tackling them is the existence of solutions to the corresponding viscously regularized problem,
which has been recently done in  \cite{MRS2018}.  
 Such  viscous solutions have to fulfill an energy-dissipation balance that,
 in turn, relies on the validity of a suitable chain rule for the driving energy. Actually this chain rule is at the very core  of the existence argument.
In \cite{MRS2018} it has been possible to prove this condition, and to ultimately conclude the existence of solutions to the viscoplastic 
finite-strain system, only for a considerably regularized version of the energy functional $\calE$ from \eqref{en-plast}.
\end{remark}

\subsection{Proof of Theorem \ref{thm:VEplast}}
\label{ss:4.1}
Preliminarily, we collect 
 the properties of $\mathcal{R}_1$ in the following result.
\begin{lemma}
\label{l:calR1}
Assume \eqref{hyp-R-plast}. Then, the functional $\calR_1 : \SLD \to [0,\infty) $ is continuous,  strictly positive for $\Sigma \neq \mathbf{1}$,  satisfies the triangle inequality
$
\mathcal{R}_1(\Sigma_1\Sigma_0) \leq \mathcal{R}_1(\Sigma_0) +  \mathcal{R}_1(\Sigma_1)
$  for all $\Sigma_0,\, \Sigma_1 \in  T_{\bf 1} \SLD $, as well as  the estimate
\begin{equation}
\label{bounded-integrand}
\exists\, C_1>0\  \ \exists\, q_{\gamma} \in [1, q_P)   \ \  \forall\, \Sigma_0,\, \Sigma_1\in \SLD\, : \qquad 
\mathcal{R}_1(\Sigma_1\Sigma_0^{-1}) \leq C_1 (1{+} |\Sigma_0|^{q_\gamma}{+}|\Sigma_1|^{q_\gamma})\,.
\end{equation}
Moreover,
\begin{equation}
\label{needed}
\forall\, M>0 \ \exists\, c_M>0 \ \forall \, \Sigma\in \SLD \,: \qquad 
\mathcal{R}_1(\Sigma) \leq M \ \Rightarrow \ \mathcal{R}_1(\Sigma)\geq c_M | \Sigma - \mathbf{1}|\,.
\end{equation}
\end{lemma}
\begin{proof}
In order to check \eqref{needed} (we refer to \cite[Sec.\ 3]{MaiMi08} for the proof of  all 
the other properties of $\calR_1$), let
$\Sigma$ fulfill
$\mathcal{R}_1(\Sigma) \leq M$: we  choose an infimizing sequence $(\Xi_n)_n \subset \mathrm{C}^1 ([0,1];G)$ such that $\Xi_n(0)=\mathbf{1}$ and $\Xi_n(1) = \Sigma$, fulfilling 
$\lim_{n\to\infty} \int_0^1 \mathrm{R}_1(\dot{\Xi}_n(s)\Xi_n(s)^{-1})\dd s   = \mathcal{R}_1(\Sigma)$. 
We define 
\[
\mathsf{s}_n : [0,1]\to [0,1]  \quad \text{by } \mathsf{s}_n(t): = c_n  \int_0^t \left( 1{+}  \mathrm{R}_1(\dot{\Xi}_n  \Xi_n^{-1}) \right) \dd s,
\]
with the normalization constant 
$
c_n: = \left(1 +  \int_0^1 \mathrm{R}_1(\dot{\Xi}_n(s)\Xi_n(s)^{-1})\dd s\right)^{-1}\,,
$
and set  
\[
 \mathsf{t}_n : =   \mathsf{s}_n^{-1}, \qquad  \widetilde{\Xi}_n : = \Xi_n \circ   \mathsf{t}_n \,.
\]
Therefore,  for $n$ sufficiently big we have 
\[
\begin{aligned}
2+M \geq  1 +  \int_0^1 \mathrm{R}_1(\dot{\Xi}_n(s)\Xi_n(s)^{-1})\dd s = \frac1{c_n}   & \geq  \frac1{c_n}
\frac{ \mathrm{R}_1(\dot{\Xi}_n(\mathsf{t}_n(s))\Xi_n(\mathsf{t}_n(s))^{-1})}{(1+ \mathrm{R}_1(\dot{\Xi}_n(\mathsf{t}_n(s))\Xi_n(\mathsf{t}_n(s))^{-1}))}  \\ & = 
 \mathrm{R}_1(\dot{\widetilde{\Xi}}_n(s)\widetilde{\Xi}_n(s)^{-1}) \geq c_R |\dot{\widetilde{\Xi}}_n(s)\widetilde{\Xi}_n(s)^{-1}|\,,
 \end{aligned}
\]
for all $s\in[0,1]$, 
where the latter estimate ensues from   \eqref{hyp-R-plast}. 
Hence the function $s\mapsto \Lambda_n(s): =  \dot{\widetilde{\Xi}}_n(s)\widetilde{\Xi}_n(s)^{-1}$ is uniformly bounded in $L^\infty (0,1; \R^{d\times d})$. Writing
\[
\widetilde{\Xi}_n(s) : = \mathbf{1} + \int_0^s  \Lambda_n(r) \widetilde{\Xi}_n(r) \dd r 
\]
we conclude, via the Gronwall Lemma, that 
\[
\exists\, \tilde{c}_M>0 \ \forall\, n \in \N\, : \qquad  \|\widetilde{\Xi}_n\|_{L^\infty (0,1;\R^{d\times d})} \leq \tilde{c}_M\,.
\]
Therefore,
\[
\begin{aligned}
\mathcal{R}_1(\Sigma)  = \lim_{n\to\infty}
\int_0^1  \mathrm{R}_1(\dot{\widetilde{\Xi}}_n(s)\widetilde{\Xi}_n(s)^{-1}) \dd s   &  \geq c_R  \liminf_{n\to\infty} \int_0^1 | \dot{\widetilde{\Xi}}_n(s)\widetilde{\Xi}_n(s)^{-1} | \dd s
\\ & \geq \frac{c_R }{\tilde{c}_M}   \liminf_{n\to\infty} \int_0^1 | \dot{\widetilde{\Xi}}_n(s) | \dd s
\geq  \frac{c_R }{\tilde{c}_M}  |\Sigma -\mathbf{1}| 
\end{aligned}
\]
where we have used the estimate $|AB^{-1}|\geq \frac{|A|}{|B|}$. 
\end{proof} 
\begin{corollary}
\label{cor4.6}
Assume \eqref{hyp-R-plast} and \eqref{hyp-Rq}. Then, 
$\mathsf{d}_\spz$ from \eqref{d-plast} is a (possibly asymmetric) quasi-distance 
separating the points of $\spz$, and fulfilling 
 \begin{equation}
\label{d-need}
\begin{aligned}
&
\forall\, M>0 \ \exists\, \tilde{c}_M>0 \ \forall \, P_0,P_1\in  \spz \,: \  
\\
& \|P_0\|_{L^\infty(\Omega)}+\|P_1\|_{L^\infty(\Omega)}
\leq M  \   \Rightarrow \ \mathsf{d}_\spz(P_0,P_1) \geq \tilde{c}_M \int_\Omega| P_1(x)P_0(x)^{-1} - \mathbf{1}|\dd x \,.
\end{aligned}
\end{equation}
Furthermore, the viscous correction $\delta_Z$ from \eqref{delta-dir} is $\sigma_Z$-lower semicontinuous on $Z\times Z$. 
\end{corollary}
\begin{proof}
To check that $\mathsf{d}_\spz$ separates the points of $\spz$,
we observe that
\[
\mathrm{d}_\spz(P_0,P_1)=0 \ \Rightarrow \ \calR_1(P_1(x)P_0^{-1}(x)) =0 \ \foraa\, x \in \Omega \ \Rightarrow \ P_0(x) = P_1(x) \ \foraa\, x \in \Omega
\]
since $\calR_1(\Sigma) >0 $ if $\Sigma \neq \mathbf{1}$. 
\par
Let us now show how \eqref{d-need} derives from \eqref{needed}. 
From $\|P_0\|_{L^\infty}+\|P_1\|_{L^\infty}
\leq M $ it follows that 
 $\|P_0^{-1}\|_{L^\infty}+\|P_1\|_{L^\infty}
\leq \widetilde{M} $. To check this, we use that  
\begin{equation}
\label{cof-matrix}
P_{0}^{-1} =\frac1{\mathrm{det}(P_0)} \transp{ \mathrm{cof}(P_0) } =  \transp{ \mathrm{cof}(P_0)} 
\end{equation}
($\mathrm{cof}(P_0) $ denoting cofactor matrix of $P_0$),
 as $P_0 \in \SLD$.
Since $\calR_1$ is continuous,  we deduce that 
\[
\exists\,  \widetilde{M}'>0 \, :  \quad 
\sup_{x\in \Omega} \calR_1(P_1(x) P_0^{-1}(x)) \leq \widetilde{M}',
\]
so that \eqref{needed} yields $ \tilde{c}_M>0$ such that 
\[
\calR_1(P_1(x) P_0^{-1}(x)) \geq   \tilde{c}_M |P_1(x) P_0^{-1}(x) {-} \mathbf{1}| \quad \text{for almost all } x \in \Omega\,.
\]
Then, \eqref{d-need}
follows. 
\par
Finally, let $(P_i^n)_n \subset \spz$ fulfill $P_i^n \weakto P_i$ as $n\to\infty$ in $L^{q_P}(\Omega;\R^{d\times d})\cap W^{1, r}(\Omega; ;\R^{d\times d})$, for $i=0,1$. 
Therefore,
$P_i^n \to P_i$ in $L^r(\Omega; ;\R^{d\times d})\cap L^{q_P -\epsilon} (\Omega; ;\R^{d\times d})$ for every $\epsilon \in (0,q_P-1]$. This  implies that
\[
P_i^n(x) \to P_i(x),
 \text{ whence }  (P_i^n(x))^{-1} \stackrel{\eqref{cof-matrix}}{=}  \transp{ \mathrm{cof}(P_i^n(x))}   \to  \transp{ \mathrm{cof}(P_i(x))} \stackrel{\eqref{cof-matrix}}{=}
(P_i(x))^{-1}   \ \foraa\, x \in \Omega, \  i =\{0,1\},
\]
as $\mathrm{det}(P_i^n(x)) =1$ for a.a.\ $x \in \Omega$, and hence $\mathrm{det}(P_i)\equiv 1$ a.e.\ in $\Omega$.
All in all, we conclude that 
\begin{equation}
\label{ptw-cvg-matrices}
P_1^n(x)(P_0^n(x))^{-1} \to P_1(x) (P_0(x))^{-1} \qquad \foraa\, x \in \Omega\,.
\end{equation}
Therefore, if $\liminf_{n\to\infty} \delta_{\spz}(P_0^n,P_1^n)<\infty$, we easily conclude that $\delta_{\spz}(P_0,P_1)<\infty$ and 
\[
\lim_{n\to\infty} \delta_{\spz}(P_0^n,P_1^n) = \liminf_{n\to\infty} \int_{\Omega} \mathrm{R}_q(P_1^n(x) (P_0^n(x))^{-1} {-} \mathbf{1}) \dd x \geq   \int_{\Omega} \mathrm{R}_q(P_1(x) (P_0(x))^{-1} {-} \mathbf{1}) \dd x  =  \delta_{\spz}(P_0,P_1),
\]
i.e.\ the claimed lower semicontinuity of $\delta_\spz$. 
\end{proof}
\par
We are now in a position to carry out the \underline{\textbf{proof of Theorem \ref{thm:VEplast}}} by verifying the validity of the conditions of Theorem \ref{thm:exists-VE}.   As we will see, the 
requirement
 $r>d$
  enters in the proof of $<B>$ \& $<C>$,  only in the case the viscous correction is given by \eqref{delta-dir}.
 \par\noindent $\vartriangleright$ \textbf{Assumption $<T>$:}  It follows from Corollary \ref{cor4.6}. 
\par\noindent $\vartriangleright$ \textbf{Assumption $<A>$:} In the proof of \cite[Thm.\ 3.1]{MaiMi08} it was shown that 
\begin{equation}
\label{coerc-plast}
\begin{aligned}
&
\exists\, C_2,\, C_3>0 \ \forall\, (t,y,P) \in [0,T]\times \spu \times \spz\, : 
\\
& \quad
 \enet tyP\geq 
C_2(\|\nabla y\|_{L^{q_Y}(\Omega)}^{q_Y}{+} \|P\|_{L^{q_P}(\Omega)}^{q_P}{+} \|\nabla P\|_{L^r(\Omega)}^{r}) -C_3\,.
\end{aligned}
\end{equation}
In view of Korn's inequality, this yields that the 
 sublevels of $\calE(t,\cdot,\cdot) $ are bounded in  the space
$V: = W^{1,q_Y}(\Omega;\R^d) \times W^{1,r}(\Omega;\R^{d\times d})$, 
uniformly w.r.t.\ $t\in[0,T]$, i.e.
\begin{equation}
\label{bounded-sublevels-finite-plast}
\forall\, S>0 \ \exists\, R_S>0 \ \forall\, (t,y,P) \in [0,T]\times \spu \times \spz\, : \quad |\enet tyP|\leq S \ \Rightarrow  \ (y,P) \in \overline{B}_{R_S}^V
\end{equation}
(cf.\ Notation \ref{gen-not}). 
We will now show that
\begin{equation}
\label{sigma-lower-semicontinuity}
\begin{aligned}
&
\Big( t_n\to t \text{ in } [0,T], \ y_n\weakto y \text{ in }  W^{1,q_Y}(\Omega;\R^d), \ P_n \weakto P \text{ in } W^{1,r}(\Omega;\R^{d\times d}) \cap L^{q_P} (\Omega;\R^{d\times d}) \Big) \\
& \   \Rightarrow  \ \ \liminf_{n\to\infty} \enet {t_n}{y_n}{P_n}\geq \enet {t}yP\,.
\end{aligned}
\end{equation}
The (sequential) lower semicontinuity of the functional $\calE_1 $  w.r.t.\ $\sigma_Z$ follows from 
\cite[Thm.\ 5.2]{MaiMi08}. We adapt the 
arguments from the latter result to show the lower semicontinuity of $\calE_2$.  First of all, since $q_Y>d$ by \eqref{integrab-powers}, from $y_n\weakto y$   in  $  W^{1,q_Y}(\Omega;\R^d)$ we deduce that $y_n\to y$ in $\rmC^0(\overline\Omega;\R^d)$. Therefore, by \eqref{conds-g} we deduce that 
\begin{equation}
\label{nablag-n}
\nabla \gdir(t_n,y_n) \to \nabla \gdir(t,y) \quad \text{ in $\rmC^0(\overline\Omega;\R^d)$.}
\end{equation}
All in all, we conclude that $ \gdir(t_n,y_n) \weakto  \gdir(t,y)
 $ in  $  W^{1,q_Y}(\Omega;\R^d)$ so that, since $\ell(t_n) \to \ell(t) $ in $  W^{1,q_Y}(\Omega;\R^d)^*$ by \eqref{hyp-ell}, we ultimately find 
 \[
  \pairing{}{W^{1,q_Y}(\Omega;\R^d)}{\ell(t_n)}{\gdir(t_n,y_n)} \to   \pairing{}{W^{1,q_Y}(\Omega;\R^d)}{\ell(t)}{\gdir(t,y)}  \quad \text{ as } n \to\infty\,.
  \]
To conclude \eqref{sigma-lower-semicontinuity},
it remains to check that 
\[
\liminf_{n\to\infty}  \int_\Omega W(x,\nabla \gdir(t_n,y_n(x)) \nabla y_n(x) P_n(x)^{-1} ) \dd x   \geq  \int_\Omega W(x,\nabla \gdir(t,y(x)) \nabla y(x) P(x)^{-1} ) \dd x \,.
\]
For this, we follow the very same arguments as in the proof of \cite[Thm.\ 5.2]{MaiMi08}, also exploiting \eqref{nablag-n}.
Clearly, \eqref{coerc-plast} and \eqref{sigma-lower-semicontinuity} ensure the validity of 
 $<A.1>$  and $<A.2>$.  
\par
It was shown in \cite[Lemma 6.1]{MRS2018} that for every $(y,P)\in \spu\times \spz$ the mapping $t\mapsto \enet tyP$ is differentiable on $[0,T]$, with 
\[
\partial_t \enet tyP  = \int_\Omega
   \inpow{x}{\gdir}{y(x)}{P(x)^{-1}} : V(t,y(x)) \dd x
   - \pairing{}{W^{1,q_Yi}}{\dot{\ell}(t)}{\gdir (t,y)} -
   \pairing{}{W^{1,q_Y}}{\ell (t)}{\dot{g}_\Dir (t,y)},
\] 
with the short-hand notation
 $\Kirchx {x}F := \rmD_F W(x,F) \transp{F}$ for the (multiplicative) Kirchhoff stress tensor, 
and  $V(t,y):= \nabla \dot{g}_\Dir (t,y) (\nabla \gdir
(t,y))^{-1} $. The power-control estimate 
\eqref{power-control} holds too, cf.\ again \cite[Lemma 6.1]{MRS2018}. 
\par
Now,  for all $\Xi \in \spz$ the functional  $\mathsf{d}_{\spz}(\cdot,\Xi)$  is left-continuous 
on $(\spz,\sigma_\spz)$ in the sense of  \eqref{dist-continuity}.
Indeed, from $P_n\weakto P$ in $W^{1,r}(\Omega;\R^{d\times d}) \cap  L^{q_P}(\Omega;\R^{d\times d})$ as $n\to\infty$ 
we have that $P_n\to P$ in  $L^{q_P-\epsilon}(\Omega;\R^{d\times d})$ for all $\epsilon \in (0,q_P-1]$.
Combining the growth condition \eqref{bounded-integrand} of $\calR_1$  and 
   the dominated convergence theorem we deduce that
\begin{equation}
\label{left-cont-dist}
\mathsf{d}_{\spz}(P_n,\Xi) = \int_{\Omega}\calR_1(\Xi(x)P_n^{-1}(x)) \dd x \to  \int_{\Omega}\calR_1(\Xi(x)P^{-1}(x)) \dd x = \mathsf{d}_{\spz}(P,\Xi) \,.
\end{equation}
   Therefore, we can check $<A.3'>$, namely the conditional upper semicontinuity \eqref{usc-power}. This has been done  in \cite[Prop.\ 4.4]{MaiMi08} by resorting to  Prop.\  \ref{prop:GF}.
\par\noindent{$\vartriangleright$ \bf  Assumption $<C>$:}  We will in fact check   \eqref{MRS}.  Let $(t_n, y_n,P_n)_n$, converging to $(t,y,P)$,  be a sequence 
as in \eqref{MRS}: with the very same arguments used for $<A.3'>$, from $\sup_{n\in \N} \enet {t_n}{y_n}{P_n} \leq C$ we deduce that
  $P_n\to P$ in  $L^{q_P-\epsilon}(\Omega;\R^{d\times d})$ for all $\epsilon \in (0,q_P-1]$.
Let us now pick \emph{any} $(y',P')\in \spu \times \spz$ with $y' \in \mathrm{Argmin}_{y\in \spu} \enet tyP$ and take the \emph{constant} recovery sequence
$
(y_n',P_n'): = (y',P')
$ for all $n\in \N$. 
Clearly, 
$\lim_{n\to\infty} \enet t{y_n'}{P_n'} = \enet t{y'}{P'}$, which entails
$\limsup_{n\to\infty} \red t{P_n'} \leq \red t{P'}$ for the reduced energy.
 Arguing as in the above lines, we also find $\mathsf{d}_{\spz}(P_n,P_n') = \mathsf{d}_{\spz}(P_n,P')\to     \mathsf{d}_{\spz}(P,P')$ as $n\to\infty$,
  which concludes the proof of \eqref{MRS} in the case the viscous correction $\delta_{\spz}$ is the `trivial' one, as in  \eqref{delta-plast-easy}. 
\par
When $\delta_{\spz}$ is instead given by \eqref{delta-dir},
we rely on 
 the compact embedding $W^{1,r}(\Omega;\R^{d\times d}) \Subset \mathrm{C}^0(\overline\Omega;\R^{d\times d})$ due to  $r>d$. 
 This guarantees that  the sequence 
 $(P_n)_n$, bounded in $W^{1,r}(\Omega;\R^{d\times d})$, in fact fulfills 
  $P_n\to P$ in  $\mathrm{C}^{0}(\overline\Omega;\R^{d\times d})$.
Therefore, $\transp{ \mathrm{cof}(P_n)}\to \transp{ \mathrm{cof}(P)}$ in   $\mathrm{C}^{0}(\overline\Omega;\R^{d\times d})$ and thus we find
\begin{equation}
\label{PnP}
P_n' P_n^{-1} = P' \transp{\mathrm{cof}(P_n)}\to P'\transp{\mathrm{cof}(P)} = P'P^{-1}  \quad \text{in }  \mathrm{C}^{0}(\overline\Omega;\R^{d\times d})
\end{equation}
(here we have again used that $P_n^{-1}=\transp{\mathrm{cof}(P_n)}$, and analogously for $P'$,  in view of  \eqref{cof-matrix} and of the fact that   $\mathrm{det}(P') =  \mathrm{det}(P_n) =1$ for every $n\in \N$).
Thus, by the continuity of $\mathrm{R}_q$ we have that $\sup_{x\in\Omega}\mathrm{R}_q (P'(x)P_n(x)^{-1}-\mathbf{1}) \leq C$. 
With the dominated convergence theorem we then infer that 
$\delta_{\spz}(P_n,P_n')  \to \delta_{\spz}(P,P')$, which establishes the validity of   \eqref{MRS}. 
\par\noindent{$\vartriangleright$ \bf  Assumption $<B>$:}  
Since $\mathrm{R}_q(0)=0$ by \eqref{hyp-Rq}, we easily check 
that 
the viscous correction $\delta_\spz$ from \eqref{delta-dir} complies with $<B.1>$. Condition 
$<B.2>$ follows from the 
very same arguments as in the above lines.
We will prove $<B.3>$ through \eqref{will-be-checked}. 
Let us now consider $(t,P) \in \mathscr{S}_{\cmdn}$ and a sequence $(t_n,P_n)\widetilde{\rightharpoonup}(t,P) $, i.e.\
$(t_n,P_n)_n
\subset \mathscr{S}_{\cmdn}$, 
$t_n \uparrow t$, 
$
P_n \weakto P$ in $W^{1,r}(\Omega)$,  $\mathsf{d}_{\spz}(P_n,P)\to 0$.
Since 
$P_n\to P$ in $\mathrm{C}^0(\overline\Omega;\R^{d\times d})$,
we may use that 
\begin{equation}
\label{d-from-below}
\exists\, \bar{c}>0 \ \forall\, n \in \N\, : \ 
\sfd_\spz(P_n,P) \geq \bar{c} \| PP_n^{-1} - \mathrm{1} \|_{L^1(\Omega;\R^{d\times d})}\,
\end{equation}
thanks to \eqref{needed}. 
Moreover,  observing that, indeed, we even have 
that 
\begin{equation}
\label{6.16}
 PP_n^{-1} \to \mathbf{1} \quad \text{ in  $ \mathrm{C}^{0}(\overline\Omega;\R^{d\times d}) $}
 \end{equation}
  (cf.\ \eqref{PnP}),
in view of \eqref{hyp-Rq} we find, for $n$ sufficiently big, 
\[
\mathrm{R}_q(P(x)P_n^{-1}(x)) \leq \left( C_q{+}\frac12 \right) | P(x)P_n^{-1}(x) {-} \mathbf{1}|^q  \quad \text{for all } x \in \Omega\,. 
\]
Therefore, 
\begin{equation}
\label{delta-p-coerc}
\delta_{\spz}(P_n,P) \leq \left( C_q{+}\frac12 \right)\| P P_n^{-1} -  \mathbf{1} \|_{L^q(\Omega;\R^{d\times d})}^q\,.
\end{equation}
Ultimately, we conclude
\[
\begin{aligned}
\lim_{(t_n,P_n)\widetilde{\rightharpoonup}(t,P) } \frac{\delta_{\spz}(P_n,P)}{\sfd_\spz(P_n,P)} &  \leq  C
\lim_{(t_n,P_n)\widetilde{\rightharpoonup}(t,P)}  
 \frac{ \| P P_n^{-1} -  \mathbf{1} \|_{L^q(\Omega)}^q}{\|PP_n^{-1} - \mathrm{1}\|_{L^1(\Omega)}} 
\\
& 
\leq C  \lim_{(t_n,P_n)\widetilde{\rightharpoonup}(t,P)}
 \frac{ \| P P_n^{-1} -  \mathbf{1} \|_{W^{1,r}(\Omega)}^{\theta q}  
 {\| P P_n^{-1} -  \mathbf{1}\|_{L^{1}(\Omega)}^{(1{-}\theta) q }}}{ \| P P_n^{-1} -  \mathbf{1}\|_{L^{1}(\Omega)}}
\\
& \leq C    \lim_{(t_n,P_n)\widetilde{\rightharpoonup}(t,P)}
 \| P P_n^{-1} -  \mathbf{1}\|_{L^{1}(\Omega)}^{(1{-}\theta) q -1} =0\,.
 \end{aligned}
\]
Here we have used 
the Gagliardo-Nirenberg inequality
in the very same way as in the proof  of Thm.\ \ref{thm:VEdam},
and the previously established convergence \eqref{6.16}.
Hence, we conclude  condition  \eqref{will-be-checked}, yielding  $<B.3>$.
\par Thus, we are in  a position to apply  Thm.\ \ref{thm:exists-VE} and conclude the existence of $\VE$ solutions.
The summability properties \eqref{summ-props-yP} follows from the energy bound $\sup_{t\in (0,T)} |\enet t{y(t)}{P(t)}| \leq C$, cf.\ \eqref{sup-perto},
combined with the coercivity estimate \eqref{coerc-plast}. 
We
have thus finished the proof of Thm.\ \ref{thm:VEplast}.
\QED

\begin{remark}
\upshape
\label{rmk:cannot}
A close perusal of the proof of the validity of conditions $<B>$ 
and $<C>$, in the case  of the non-trivial viscous regularization $\delta_\spz$ from \eqref{delta-dir}, 
reveals the key role played by the condition $r>d$
(which has been for instance used in the proof of \eqref{d-from-below}).
 Unlike for the damage system
tackled in Sec.\ \ref{s:3},  it is not clear how to weaken this requirement.
\end{remark}

\section{Passing from adhesive contact to brittle delamination with Visco-Energetic solutions}
\label{s:adh2bri} 
In this section we construct $\VE$ solutions to a rate-independent system modeling brittle delamination between two elastic bodies,
by passing to the limit in the Visco-Energetic formulation of an approximating system for adhesive contact.
Besides providing the existence of $\VE$ solutions
for brittle delamination,
 Theorem \ref{th:VE-brittle} below  is, in fact, a first result on the \emph{Evolutionary Gamma-Convergence} of Visco-Energetic solutions. 
 \par
 First of all, let us briefly sketch the model. 
 We  consider delamination between two bodies $\Omega_+, \, \Omega_- \subset \R^d$, $d\in \{2,3\}$   along their common boundary.
 More precisely, throughout this section we shall suppose that 
 \begin{equation}
   \label{geometry-brittle}
    \begin{aligned}
    &
    \Omega_{\pm}, \ \  \Omega: =  \Omega_+\cup \GC \cup \Omega_- \   \text{ are Lipschitz domains}, 
    \\
    & \partial\Omega = \Gdir \cup \Gneu \text{ with }
    \begin{cases}
    \mathscr{H}^{d-1} (\partial \Omega_\pm \cap \Gdir) >0,
    \\
    \overline{\GC} \cap \overline{\Gdir} =\emptyset\,.
    \end{cases}
    \end{aligned}
    \end{equation}
 The process is modeled
with the aid of an internal
delamination variable $z : [0; T]  \to \GC$, $0\leq z\leq 1 $ on $\GC$, 
 which describes the state of the adhesive material  located on $ \GC$ during
a time interval $[0, T]$. 
  In particular, in our notation $z(x,t) = 1$,  resp.\ $z(x,t) = 0$, shall indicate that
the glue is fully intact, resp.\ broken,
  at the point $x\in \GC$ and at the process time $t\in [0,T]$.
  Within the assumption of small strains, we also consider the displacement variable
  $u:\Omega \to \R^d$.
  \emph{Brittle} delamination  is characterized by the
  \begin{equation}
  \label{brittle-constr}
  \text{\emph{brittle constraint}} \qquad z(x,t) \JUMP{u(x,t)} =0  \quad \text{ on } \GC \times (0,T),
  \end{equation}
  where $\JUMP{u}: = u^+|_{\GC} - u^-|_{\GC}   $ is the difference of the traces on $\GC$ of $u^\pm= u|_{\Omega_\pm}$. 
   This condition allows for displacement jumps only at points $x\in \GC$ where the bonding is completely broken,
i.e. $z(x,t) = 0$; at points where $ z(x,t) > 0$  it ensures $\JUMP{u(x,t)}=0$,  i.e. the continuity of the displacements.
   Therefore, \eqref{brittle-constr} distinguishes between the crack set, where the displacements may jump, and
the complementary set with active bonding, where it imposes a transmission condition on the displacements.
   \par
   The (formally written) rate-independent system for brittle delamination reads
   \begin{subequations}
   \label{RIS-brittle}
   \begin{align}
   &
      \label{RIS-brittle-1}
   -\mathrm{div}(\bbC \eps(\tilde{u})) = f  && \text{ in } \Omega \times (0,T), && && 
   \\
   & 
      \label{RIS-brittle-2}
      \tilde{u} = w_\Dir && \text{ on } \Gdir \times (0,T), && \bbC \eps(\tilde{u})|_{\Gneu} \nu = g  && \text{ on } \Gneu \times (0,T), 
      \\
      & 
      \label{RIS-brittle-3}
      \bbC \eps(\tilde{u})|_{\GC} n  +\partial_u I_{\mathsf{C}}(\tilde{u}, z)  + \partial {I}_{\mathsf{U}(x)}(\JUMP{\tilde{u}}) \ni 0   && \text{ on } \GC \times (0,T), &&  && 
         \\
         & 
      \label{RIS-brittle-4}
              \partial \mathrm{R}(x,\dot z ) + \partial_z  I_{\mathsf{C}}(\JUMP{\tilde{u}}, z)  + \partial I_{[0,1]}(z) \ni a_0 && \text{ on } \GC \times (0,T). &&  && 
   \end{align}
   \end{subequations}
   The static momentum balance \eqref{RIS-brittle-1}, where $\bbC$ is the (positive definite, symmetric) elasticity tensor and $f$ a body force, is coupled with a time-dependent Dirichlet condition on the Dirichlet portion
   $\Gdir$
    of the boundary $\partial\Omega$, with outward unit normal $\nu$ (cf.\ \eqref{geometry-brittle} below).
    On the Neumann part $\Gneu$ a surface force $g$ is assigned. 
    The evolutions of $u$ and $z$    
    are coupled by the Robin-type boundary condition  \eqref{RIS-brittle-3} on the contact surface $\GC$, where  
    $\partial_u I_{\mathsf{C}} : \R^d \rightrightarrows \R^d $ is the   (convex analysis) subdifferential
    w.r.t.\ $u$
     of the 
     indicator function of  the set
     \[
     \mathsf{C} : = \{ (v,z)\in \R^d \times \R\, : \ \JUMP{v} z =0 \},
     \]
    while
    $\partial I_{\mathsf{U}(x)}:  \R^d \rightrightarrows \R^d   $ is the subdifferential of the indicator of 
    \[
    \mathsf{U}(x) = \{ v \in \R^d\, : v \cdot n(x) \geq 0\}, \quad x \in \GC,
    \]
   with $n $ the  unit normal to $\GC$, oriented from $\Omega_+$ to $\Omega_-$. Hence, besides \eqref{brittle-constr}, we are also imposing the \emph{non-penetration} constraint $\JUMP{u} \cdot n \geq 0$ in $\Omega$ between $\Omega_+$ and $\Omega_-$. Finally, the flow rule \eqref{RIS-brittle-4}
 for the delamination parameter $z$
 involves the very same  dissipation density $\mathrm{R}$ from \eqref{diss-dam}, the subdifferential w.r.t.\ $z$ of $I_{\mathsf{C}}$, and  the coefficient $a_0$, i.e.\
 the phenomenological specific energy per area which is stored 
  by disintegrating the adhesive.
\par
From now on, we will again use the splitting 
$\tilde u= u+w_\Dir$, with $w_\Dir$
an extension of the Dirichlet datum to  the whole of $\Omega$.   In view of 
\eqref{geometry-brittle}, without loss of generality
we may assume that  this extension fulfills
\begin{equation}
\label{WLOG-brittle}
w_\Dir|_{\GC}   \equiv 0 \text{ on } \GC, \text{ so that } \JUMP{\tilde u} = \JUMP{u+w_\Dir} = \JUMP{u}.
\end{equation}
The Energetic formulation of the brittle system \eqref{RIS-brittle}
 thus involves the following:
 \begin{subequations}
 \label{brittle-setup}
  \paragraph{\bf Ambient space:} $X=\spu \times \spz$ with 
\begin{equation}
\label{sp-bri}
\spu = H_{\Gamma_\Dir}^{1}(\Omega{\setminus}\GC;\R^d): = \{ u\in H^1(\Omega{\setminus}\GC;\R^d)\, : \, u=0 \text{ on } \Gamma_{\Dir} \}, \ \ 
\spz := \{ z\in L^\infty(\GC)\, : \ 0 \leq z \leq 1 \text{ on } \GC\},
\end{equation}
endowed with the weak topology  $\sigma_\spu$ of $H^1(\Omega{\setminus}\GC;\R^d)$ and with the weak$^*$-topology $ \sigma_\spz$  of $L^\infty(\GC)$, respectively. 
\paragraph{\bf  Energy functional:}
 $\calE: [0,T]\times \Xs \to (-\infty, \infty]$ is given by 
\begin{align}
\label{en-bri}
\begin{aligned}
\enet tuz: = &  \frac12  \int_{\Omega{\setminus}\GC}  \bbC \eps(u{+}w_\Dir) :   \eps(u{+}w_\Dir)  \dd x  \\ & + \int_{\GC}\left(  I_{\mathsf{U}(x)}(\JUMP{u}) 
{+} I_{\mathsf{C}}(\JUMP{u},z) {+}I_{[0,1]}(z) {-}a_0 z  \right) \dd \mathscr{H}^{d{-}1}(x)  -   \pairing{}{H^1(\Omega{\setminus}\GC;\R^d)}{\ell(t)}{u+w_\Dir(t)},
\end{aligned}
\end{align}
where the function $\ell : [0,T]\to H^1(\Omega;\R^d)^*$ subsumes the body and surface forces $f$ and $g$. 
Observe that the domain of $\calE$ does not depend on the time variable, i.e.\
\[
\mathrm{D}(\mathcal{E}(t,\cdot))  = \{ (u,z)\in \spu \times \spz \, : 
\JUMP{u(x)} \in \mathsf{U}(x), \ z(x) \JUMP{u(x)}=0, \ 
\ z(x) \in [0,1] \ \foraa\,x \in \GC \}\, \quad \text{for all } t \in [0,T]\,.
\] 
\paragraph{\bf Dissipation distance:}
We consider the extended asymmetric quasi-distance
 $\mdsn{\spz}: \spz \times \spz \to [0,\infty]$ defined by 
\begin{equation}
\label{d-bri}
\mds{\spz}{z}{z'} : = \mathcal{R}(z'-z)  \quad  \text{ with }  \mathcal{R}:L^1(\GC) \to [0,\infty], \  \calR(\zeta) : = \int_{\GC} \mathrm{R}(x,\zeta(x)) \dd \mathscr{H}^{d-1}(x) \,
\end{equation}
and  the dissipation density $\mathrm{R}$ from \eqref{diss-dam}. 
  \end{subequations}
  Due to the highly nonconvex character of the brittle constraint  
   \eqref{brittle-constr},  
the existence of Energetic solutions to the rate-independent system $\RIS$  from \eqref{brittle-setup}
cannot be proved by  directly passing to the time-continuous limit in the associated time-incremental minimization scheme.
 Indeed, an existence result
 was obtained in 
\cite{RoScZa09QDP} 
 by passing to the limit
 in the Energetic formulation for a penalized version of system \eqref{RIS-brittle}. The resulting system
  is in fact a model for \emph{adhesive contact}.
The  relevant energy functional, in the  very same displacement and delamination variables,
is given by 
\begin{align}
\label{en-adh}
\begin{aligned}
\enetk ktuz: = &  \frac12  \int_{\Omega{\setminus}\GC}  \bbC \eps(u{+}w_\Dir) :   \eps(u{+}w_\Dir)  \dd x  \\ & + \int_{\GC}\left(  I_{\mathsf{U}(x)}(\JUMP{u}) 
{+} \tfrac{k} 2 z|\JUMP{u}|^2  {+}I_{[0,1]}(z) {-}a_0 z  \right) \dd \mathscr{H}^{d{-}1}(x)  -   \pairing{}{H^1}{\ell(t)}{u+w_\Dir(t)}, \quad k>0.
\end{aligned}
\end{align}
Note that the brittle constraint \eqref{brittle-constr} is penalized by the term $ \tfrac{k} 2 z|\JUMP{u}|^2$. 
Via the  Evolutionary Gamma-convergence theory  from \cite{MRS06},   in \cite{RoScZa09QDP}  it was shown that $\ENE$ solutions  to the 
adhesive contact system $\RISK{k}$  converge as $k\to\infty$ to $\ENE$ solutions to the brittle delamination system 
$\RIS$. 
\par
We aim to extend this approach to the existence of $\VE$ solutions of the brittle system. In fact, 
$\VE$ solutions of the  adhesive contact system were tackled in \cite[Example 4.5]{SavMin16} with the 
\paragraph{\bf Viscous correction:} $\delta_\spz : \spz \times \spz \to [0,\infty]$ of the form
\begin{equation}
\label{delta-bri}
\delta_\spz(z,z') : =  h(\mds{\spz}{z}{z'}) \quad \text{with } h \text{ as in \eqref{a-funct-of-d}},
\end{equation}
cf.\ also Remark \ref{rmk:more-gen-delta} below. 
Under the condition that 
\begin{equation}
\label{ass-ell-bri}
w_\Dir \in \rmC^1([0,T]; H^1(\Omega;\R^d))\,,
\qquad 
\ell \in \mathrm{C}^1 ([0,T]; H^1(\Omega;\R^d)^*)\,,
\end{equation}
 the existence of $\VE$ solutions $(u_k,z_k)
\in L^\infty (0,T; H^1(\Omega;\R^d)) \times (L^\infty (\GC{\times}(0,T)) {\cap} \BV([0,T];L^1(\GC)))$ 
to the adhesive contact system $\RISK{k}$ with the  viscous correction   from \eqref{delta-bri} 
was derived 
in \cite{SavMin16}
(again, observe that the summability property  $u\in  L^\infty (0,T; H^1(\Omega;\R^d))  $ derives from the 
energy bound $\sup_{t\in (0,T)} |\enetk kt{u(t)}{z(t)}| \leq C$, cf.\ \eqref{sup-perto}).
\par 
  We now address the limit passage in the $\VE$ formulation of  $\RISK{k}$ as $k\to\infty$. From now on, we will assume for simplicity that $k\in \N$.
\begin{theorem}
\label{th:VE-brittle}
Assume \eqref{geometry-brittle}, \eqref{WLOG-brittle},  and \eqref{ass-ell-bri}. \\
 Let $(u_k,z_k)_{k\in \N} \subset L^\infty (0,T; H^1(\Omega;\R^d)) \times (L^\infty (\GC{\times}(0,T)) {\cap} \BV([0,T];L^1(\GC)))$ be a sequence of $\VE$ solutions to the rate-independent systems $\RISK{k}$,  with $\delta_\spz$ from \eqref{delta-bri} and initial datum 
$z_0\in \domene z$. 
\par
Then, for any sequence $(k_j)_{j\in\N}$ with $k_j\to\infty$ as $j\to\infty$ there exist a (not relabeled) subsequence $(u_{k_j},z_{k_j})_{j\in \N}$
 and $(u,z) \in L^\infty (0,T; H^1(\Omega{\setminus}\GC;\R^d)) \times (L^\infty (\GC{\times}(0,T)) {\cap} \BV([0,T];L^1(\GC)))$ such that 
 \begin{enumerate}
 \item $z(0)=z_0$; 
 \item
 the following convergences hold as $j\to\infty$
 \begin{subequations}
 \label{cv:adh2bri}
 \begin{align}
  \label{cv:adh2bri-1}
 &   u_{k_j}(t) \weakto u(t)  && \text{ in } H^1(\Omega{\setminus}\GC;\R^d) &&    \text{for all } t \in [0,T],  
\\
&
 \label{cv:adh2bri-2}
z_{k_j}(t) \weaksto z(t)  && \text{ in } L^\infty(\GC) && \text{for all } t \in [0,T],
\end{align}
 \end{subequations}
 \item 
 $(u,z) $ is a $\VE$ solution of the 
brittle delamination system $\RIS$ \eqref{brittle-setup}, with the  viscous  correction from \eqref{delta-bri}, such that the minimality property
\eqref{MINIMALITY} holds at all $t\in [0,T]\setminus \bar{\rmJ}$, with  $ \bar{\mathrm{J}}$ a negligible subset 
of $(0,T]$.
\end{enumerate}
 Furthermore, we have the additional convergences
 as $j\to\infty$
\begin{equation}
\label{energy-cv}
\enetk {k_j}{t}{u_{k_j}(t)}{z_{k_j}(t)} \to \enet t{u(t)}{z(t)} \text{ and } \Vari {\mdsn{\spz},\vecostname}{z_{k_j}}0{t} \to \Vari {\mdsn{\spz},\vecostname}{z}0{t}
 \qquad \text{for all } t \in [0,T].
\end{equation}
\end{theorem}
Observe that we are able to recover the minimality property \eqref{MINIMALITY}
only almost everywhere on $(0,T)$. 
The \emph{proof}  of Thm.\ \ref{th:VE-brittle} will be carried out throughout Sec.\ \ref{ss:6.1}, also relying on a technical result, Lemma \ref{prop:key-lsc-adh2bri}
ahead,
proved in Sec.\ \ref{ss:6.2}.
  \begin{remark}
  \upshape
 \label{rmk:more-gen-delta} 
The existence of $\VE$ solutions to the adhesive contact system $\RISK{k}$ could be 
extended to the case of the `non-trivial' viscous correction
\begin{equation}
\label{nontriv-delta-bri}
\delta_{\spz}(z,z'): = \frac1q \|z'{-}z\|_{L^q(\GC)}^\gamma, \quad q, \, \gamma>1,
\end{equation}
as soon as a gradient regularizing term of the type $|\nabla z|^r$ is added to  the energy functional
$\calE_k$ (under the additional, technical condition that $\GC$ is a `flat' $(d{-}1)$-dimensional surface, so that Laplace-Beltrami operators can be avoided).
 The exponents $r,\, q,\, \gamma$ should satisfy the compatibility condition \eqref{compatibility-below}. For instance,  in the case
$\Omega \subset \R^3$, so that $\GC \subset\R^2$, with $r=2$ and $q=2$ one would have to  take $\gamma>2$. 
\par
We could perform the adhesive-to-brittle limit passage with $\delta_\spz$ from \eqref{nontriv-delta-bri} by straightforwardly adapting the arguments in the proof of 
Thm.\ \ref{th:VE-brittle}. Anyhow, we have preferred not to
do so in order to focus on the analytical difficulties related to the limit passage in the notion of $\VE$ solution.
\end{remark}

\subsection{Proof of Theorem \ref{th:VE-brittle}}
\label{ss:6.1}
Preliminarily, let us recall the $\Gamma$-convergence properties of the adhesive contact energies $(\calE_k)_k$. These properties are  at the core of  the proof of Thm.\ 
 \ref{th:VE-brittle}.
 \begin{lemma} \cite[Corollary 3.2]{RoScZa09QDP}
 \label{l:Gamma-convg-energies}
 Assume  \eqref{geometry-brittle}, 
 \eqref{WLOG-brittle},
  and \eqref{ass-ell-bri}. Then the functionals $\calE_k$ from
 \eqref{en-adh} $\Gamma$-converge as $k\to\infty$ to $\calE$ w.r.t.\ to the 
 $\sigma_\R$-topology of  $[0,T]\times H^1(\Omega{\setminus}\GC;\R^d) \times L^\infty(\GC)$
 (i.e., the weak$^*$-topology), namely there hold the
 \begin{equation}
 \label{Gamma-adh-bri}
 \begin{aligned}
 &\hspace{-0.5cm}  \hbox{$\Gamma$-$\liminf$ estimate: }
 (t_k u_k,z_k) \wsigmaR (t,u,z)  \   \Rightarrow &&  \liminf_{k\to\infty} \enetk {k}{t_k}{u_k}{z_k}\geq \enet tuz,
 \\
 &  \hspace{-0.5cm}  \text{$\Gamma$-$\limsup$ estimate: } 
 \forall\, (t,u,z) \ \exists\, (t_k,u_k,z_k)_k\, :  (t_k u_k,z_k)   \wsigmaR (t,u,z),  && \limsup_{k\to\infty} \enetk {k}{t_k}{u_k}{z_k}\leq \enet tuz\,.
 \end{aligned}
 \end{equation}
 \end{lemma}
 In order to pass to the limit in the $\VE$-formulation, we also need to investigate the closure,  as $k\to\infty$,  of  the stable 
 (in the Visco-Energetic sense)
 sets 
 \[
 \begin{aligned}
 \mathscr{S}_{\cmdsn{\spz}}^k: = \{ (t,u,z) \in [0,T] \times  \spu\times \spz\, : \ &  \enetk ktuz \leq \enetk kt{u'}{z'} + \mds{\spz}z{z'} + h( \mds{\spz}z{z'})
 \\
  & \quad \text{for all } (u',z') \in \spu \times \spz
  \},  \ k \in \N,
  \end{aligned}
 \]
 with $h$ as \eqref{delta-bri} (while we will denote by $\mathscr{S}_{\cmdn}$ the stable set for the brittle delamination system). 
 More precisely, we will study the 
 \emph{Kuratowski limit inferior}
 \[
 \mathrm{Li}_{k\to\infty}  \mathscr{S}_{\cmdsn{\spz}}^k: = \{ (t,u,z) \in [0,T] \times  \spu\times \spz\, : \  \exists\, (t_k,z_k,u_k) \in  
 \mathscr{S}_{\cmdsn{\spz}}^k \text{ such that } (t_k,z_k,u_k)    \wsigmaR (t,u,z) \}\,.
 \]
 Recall that, by  \eqref{propR} $ \mathscr{S}_{\cmdsn{\spz}}^k$ is the zero set of the residual stability function 
 \[
 \begin{gathered}
  \rstabtk ktz:  = \sup_{z'\in \spz} \left\{ \redk ktz - \redk kt{z'} -  \mds{\spz}z{z'} - h( \mds{\spz}z{z'})  \right\} \quad  \text{with the reduced energy  } 
  \\
  \redk ktz: = \inf_{u\in H_{\Gdir}^1(\Omega{\setminus}\GC;\R^d)} \enetk ktuz = \min_{u\in H_{\Gdir}^1(\Omega{\setminus}\GC;\R^d)} \enetk ktuz,
  \end{gathered}
  \]
  (the $\inf$ in the definition of $\rednk k $ is attained
since
  for every $(t,z)\in [0,T]\times \spz$  the functional $u\mapsto \enetk ktuz $ has sublevels bounded in 
  $H_{\Gdir}^1(\Omega{\setminus}\GC;\R^d)$ by Korn's inequality, 
  cf.\ \eqref{Korn-brittle} ahead, 
  and is lower semicontinuous w.r.t.\ $H^1$-weak convergence).
In fact,  the study of $ \mathrm{Li}_{k\to\infty}  \mathscr{S}_{\cmdsn{\spz}}^k $ is related to the $\Gamma$-$\liminf$ (w.r.t.\ $\sigma_\R$-topology) of the 
  functionals $(\rstabnamek k )_{k}$.  That is why,  we will  further obtain the $\liminf$-inequality \eqref{liminf-Rk} below. 
  Such estimate  will also play a crucial role
  for the limit passage in the Visco-Energetic energy-dissipation balance as $k\to\infty$.
  \begin{lemma}
  \label{l:closure-brittle}
   Assume  \eqref{geometry-brittle}, \eqref{WLOG-brittle},  \eqref{ass-ell-bri}.
   Then,
   \begin{equation}
\label{trivial-consequence}
 \mathrm{Li}_{k\to\infty}  \mathscr{S}_{\cmdsn{\spz}}^k \subset \mathscr{S}_{\cmdsn{\spz}}\,
\end{equation}
 and, in fact,  for every
     $(t_k,z_k)_k\subset [0,T]\times \spz$ there holds
     \begin{equation}
     \label{liminf-Rk}
      (t_k ,z_k) \wsigmaR (t,u,z)   \Rightarrow \liminf_{k\to\infty}   \rstabtk k{t_k}{z_k}  \geq \rstab tz\,.
     \end{equation}
     \end{lemma}
  \begin{proof}
We start by showing \eqref{liminf-Rk}.
We use that
  \begin{subequations}
  \begin{align}
  &
  \label{def-rstabk}
   \rstabtk k{t_k}{z_k} = \sup_{z'\in \spz} \left( \redk k {t_k}{z_k} {-}   \redk k {t_k}{z'}{ -} \mds{\spz}{z_k}{z'} {-}  h(\mds{\spz}{z_k}{z'})\right),
   \\
   & 
     \label{def-rstab-bri}
      \rstabt {t}{z} = \sup_{z'\in \spz} \left( \red{t}{z} {-}   \red {t}{z'}{ -} \mds{\spz}{z}{z'} {-}  h(\mds{\spz}{z}{z'})\right)
      \end{align}
      \end{subequations}
      (with $\calI$ the reduced energy associated with $\calE$). 
In order to prove \eqref{liminf-Rk} it is therefore sufficient to exhibit, for any fixed $z'\in Z$ such that $\mds{\spz}{z}{z'}<\infty$
(i.e., $z'\leq z $ a.e.\ in $\GC$) and $\redk k{t}{z'}<\infty$, a \emph{recovery} sequence $(z_k')_k\subset \spz$
 such that 
 \begin{equation}
 \label{recovery-sequence}
 \limsup_{k\to\infty} \left(  \redk k {t_k}{z_k'}{ +} \mds{\spz}{z_k}{z_k'} {+}  h(\mds{\spz}{z_k}{z_k'}) \right)
 \leq \left(   \red{t}{z'}{ +} \mds{\spz}{z}{z'} {+}  h(\mds{\spz}{z}{z'}) \right)\,.
 \end{equation}
 Then, we will  have 
 \[
 \begin{aligned}
  \liminf_{k\to\infty}  \rstabtk k{t_k}{z_k}  &  \stackrel{\eqref{def-rstabk}}{\geq}  \liminf_{k\to\infty} \left(  \redk k{t_k}{z_k} {-}   \redk k{t_k}{z_k'}{ -} \mds{\spz}{z_k}{z_k'} {-}  h(\mds{\spz}{z_k}{z_k'}) \right) 
  \\
  &
  \stackrel{\eqref{recovery-sequence}}{\geq} \left( \red{t}{z} {-}   \red{t}{z'}{ -} \mds{\spz}{z}{z'} {-}  h(\mds{\spz}{z}{z'}) \right),
  \end{aligned}
 \]
 where we have also exploited the $\Gamma$-$\liminf$-estimate in \eqref{Gamma-adh-bri}. 
Then, \eqref{liminf-Rk} shall follow from the arbitrariness of $z'$.
 We borrow the definition of the sequence $(z_k')_k$ from the proof of \cite[Thm.\ 3.3]{RoScZa09QDP}, letting 
 \begin{equation}
 \label{def-rec-seq}
 z_k': = \begin{cases}
 z_k \frac{\displaystyle z'}{\displaystyle z} & \text{ if } z'>0,
 \\
 0 & \text{ otherwise}.
 \end{cases}
 \end{equation}
 Since $z'\leq z $ a.e.\ in $\GC$, it is immediate to verify that $0\leq z_k'\leq z_k \leq 1 $ a.e.\ in $\GC$.
 Furthermore, $z_k\weaksto z$ in $L^\infty(\GC) $ gives $z_k' \weaksto z'$ in $L^\infty(\GC) $.
 Therefore,
 \begin{equation}
 \label{recovery-dissips}
 \begin{gathered}
 \lim_{k\to\infty}  \mds{\spz}{z_k}{z_k'}  =   \lim_{k\to\infty}\int_{\GC} \kappa (z_k(x){-}z_k'(x) ) \dd \mathscr{H}^{d-1}(x) 
 =  \int_{\GC} \kappa (z(x){-}z'(x) ) \dd \mathscr{H}^{d-1}(x)  =   \mds{\spz}{z}{z'} 
 \\
\text{whence }  \lim_{k\to\infty} h( \mds{\spz}{z_k}{z_k'}) = h( \mds{\spz}{z}{z'}), 
 \end{gathered}
 \end{equation}
too.
 Let us now consider the (unique) minimizer $u' \in \spu$ for $\enet t{\cdot}{z'}$. We have 
  \begin{equation}
 \label{recovery-energies}
 \limsup_{k\to\infty}   \redk k {t_k}{z_k'} \leq  \limsup_{k\to\infty}   \enetk k {t_k}{u'}{z_k'} =\lim_{k\to\infty}   \enetk k {t_k}{u'}{z_k'} 
 \stackrel{(1)}{=}    \enet {t}{u'}{z'} = \red t{z'}.
 \end{equation}
Indeed, for
{\footnotesize (1)} we have used the fact that 
$z_k' \JUMP{u'} =0$ a.e.\ in $\GC$, which follows from 
$z'\JUMP{u'}=0$ and from the definition \eqref{def-rec-seq} of $z_k'$. 
 From \eqref{recovery-dissips} and \eqref{recovery-energies} we clearly conclude
  \eqref{recovery-sequence}, whence \eqref{liminf-Rk}.
  \par
  In order to show that every element $(t,u,z)$ in  $\mathrm{Li}_{k\to\infty}  \mathscr{S}_{\cmdsn{\spz}}^k $ fulfills the
  $\cmdsn{\spz}$-stability condition with the brittle energy functional,   for every $(u',z')$ we need to exhibit a recovery sequence $(u_k',z_k')_k $ such that 
  \[
 \limsup_{k\to\infty} \left(  \enetk k {t_k}{u_k'}{z_k'}{ +} \mds{\spz}{z_k}{z_k'} {+}  h(\mds{\spz}{z_k}{z_k'}) \right)
 \leq \left(   \enet {t}{u'}{z'}{ +} \mds{\spz}{z}{z'} {+}  h(\mds{\spz}{z}{z'}) \right)\,.
 \]
  The  sequence $(u_k',z_k')_k: = (u',z_k')_k$ with $(z_k')_k$ from  \eqref{def-rec-seq},
  does the job. 
    This finishes the proof.
  \end{proof}
\par
The \underline{\bf proof of Thm.\     \ref{th:VE-brittle}} will be carried out in the following steps:
\begin{compactenum}
\item First of all,
 we will show that the sequence  $(z_{k_j})_{j\in \N}$ of $\VE$ solutions in the  statement of the theorem  does admit a subsequence converging in the sense  of \eqref{cv:adh2bri-2} to $z$;
 \item Secondly, 
 we will prove that $z$ complies with the stability condition \eqref{stab-VE} for the brittle system 
 $\RIS$  from  \eqref{brittle-setup} and, as a byproduct, obtain  convergence
\eqref{cv:adh2bri-1}  for $(u_{k_j})_{j\in \N}$ ;
 \item Thirdly, we  will show that $(u,z)$  fulfills the upper energy-dissipation estimate  \eqref{enue-VE} for the brittle system also relying on Proposition \ref{prop:6.2} ahead;
 \item We shall  thus conclude that $(u,z) $ is a $\VE$ solution to  the brittle system $\RIS$    \eqref{brittle-setup}. 
\end{compactenum}
\par\noindent
 $\vartriangleright$ \textbf{Step $1$:}  Since the constant $C_0$ in \eqref{sup-perto} only depends on the initial data $(u_0,z_0)$,
  which in turn do not depend on ${k_j}$, for the  $\VE$ solutions $(u_{k_j},z_{k_j})_j$ to the adhesive contact system the following bounds are valid 
\[
\exists\, C>0 \ \forall\, j \in \N \ \forall\, t \in [0,T]\, : \quad \sup_{t\in [0,T]} |\enetk {k_j} t{u_{k_j}(t)}{z_{k_j}(t)}| + \Vari {\mdsn{\spz}}{z_{k_j}}0{T}\leq   C\,.
\]
In turn, it follows from the positive definiteness of $\bbC$, Korn's inequality and from \eqref{ass-ell-bri} that 
\begin{equation}
\label{Korn-brittle}
\exists\, c_1,\, c_2>0  \ \  \forall\, (t,u,z)\in [0,T]\times \spu \times \spz \, : \qquad \enetk ktuz \geq c_1 \| u\|_{H^1(\Omega{\setminus}\GC)}^2 -c_2\,.
\end{equation}
We ultimately conclude that the sequences $(u_{k_j})_j$ and $(z_{k_j})_j$ are bounded in 
$L^\infty (0,T; H_\Dir^1(\Omega{\setminus}\GC;\R^d)) $ and in $L^\infty(\GC{\times}(0,T)) \cap \BV([0,T];L^1(\GC))$, respectively.
An infinite-dimensional version of Helly's compactness theorem (cf., e.g., \cite[Thm.\ 3.2]{MaiMie05EREM}) yields that, up
 to a not relabeled subsequence, convergence \eqref{cv:adh2bri-2} 
 for $(z_{k_j})_j$  holds. As for $(u_{k_j})_j$, 
for every $t\in (0,T]$ there exist a subsequence $(k_j^t)$, possibly depending on $t$, 
 and $\tilde u(t) \in H_\Dir^1(\Omega{\setminus}\GC;\R^d)$ such that 
\begin{equation}
\label{conv-kj-t}
u_{k_j^t}(t)\weakto \tilde{u}(t) \quad \text{ in $ H_\Dir^1(\Omega{\setminus}\GC;\R^d).$}
\end{equation}
\par
Furthermore, mimicking the arguments in the proof of 
\cite[Thm.\ 7.2]{SavMin16}, we also find a finer approximation property at every $t$ in the jump set $  \mathrm{J}_z$ of $z$, namely
\begin{equation}
\label{finer-approx}
\forall\, t\in  \mathrm{J}_z \cap (0,T)  \ \exists\, (\alpha_{k_j})_j,\, (\beta_{k_j})_j \subset [0,T] \text{ such that }
\begin{cases}
 \alpha_{k_j}\uparrow t \text{ and } z_{k_j}(\alpha_{k_j}) \weaksto \lli zt \text{ in } L^\infty(\GC),
 \\
 \beta_{k_j}\downarrow t \text{ and } z_{k_j}(\beta_{k_j}) \weaksto \rli zt \text{ in } L^\infty(\GC),
 \end{cases}
\end{equation}
with  obvious modifications at $t\in \mathrm{J}_z \cap \{ 0,T\}$. 
\par\noindent
 $\vartriangleright$ \textbf{Step $2$:}
Let us introduce
 the $\limsup$ of the jump sets  $(\mathrm{J}_{z_{k_j}})_j$, i.e.\
  $\tilde{\mathrm{J}}: = \cap_{m\in N} \cup_{j\geq m} \mathrm{J}_{z_{k_j}} $. Observe that
for every $t\in [0,T]\setminus \tilde{\mathrm{J}}$ there exists $m_t\in \N$ such that for every $j\geq m_t$ we have $t\in [0,T]\setminus \mathrm{J}_{z_{k_j}}$. Therefore,
up to taking a bigger $m_t$ if necessary,
we have 
$
(t, u_{k_j^t}(t), z_{k_j^t}(t)) \in    \mathscr{S}_{\cmdn}^{k_j^t} $
for all $ j \geq m_t.$
By virtue of \eqref{trivial-consequence}, we conclude that 
\begin{equation}
\label{partial-stability}
(t,\tilde{u}(t), z(t)) \in \mathscr{S}_{\cmdn} \quad \text{for all }  t \in [0,T]\setminus \tilde{\mathrm{J}}.
\end{equation}
From \eqref{partial-stability} we gather, in particular, that $\tilde{u}(t) \in \mathrm{Argmin}_{u' \in \spu} \enet t{u'}{z(t)}$. Since the latter set is a singleton by Korn's inequality, we ultimately find that  $\tilde{u}(t)$ is uniquely determined. Therefore,
convergence \eqref{conv-kj-t} 
 holds at every  $t\in [0,T]\setminus \tilde{\mathrm{J}}$
 along the \emph{whole} sequence $(k_j)_j$. This shows   \eqref{cv:adh2bri-1} at all   
 $t\in [0,T]\setminus \tilde{\mathrm{J}}$.
 \par
 Finally, we conclude the validity of \eqref{cv:adh2bri-1}  at \emph{every} $t \in [0,T]$ by observing that,
 at every $t$ in   the \emph{countable} set 
 $\tilde{\mathrm{J}}$ we can extract a subsequence  of 
 $(k_j)_j$
 such that 
 \eqref{conv-kj-t}. With a diagonal procedure we thus construct a subsequence fitting all $t\in \tilde{\mathrm{J}}$ and \eqref{cv:adh2bri-1}  follows. 
 \par
 We now show that 
 \begin{equation}
 \label{stab-at-jumps}
 \lli zt, \, \rli zt \in  \mathscr{S}_{\cmdn}(t) \quad \text{for all } t \in (0,T),  \ \ \rli z0 \in  \mathscr{S}_{\cmdn}(0), \ \
 \lli zT \in  \mathscr{S}_{\cmdn}(T)\,.
 \end{equation}
In order to prove the assert at $t\in (0,T)$ and, e.g.,  for $\lli zt$,  we pick a sequence $(t_n)_n\subset [0,T]\setminus  \tilde{\mathrm{J}}$ with $t_n\uparrow t$ as $n\to\infty$, so that $z(t_n) \weaksto \lli zt$ in $L^\infty(\GC)$ (cf.\ Definition \ref{def:sigma-d-reg}).
From \eqref{partial-stability} we have that $\rstab {t_n}{z(t_n)} =0$ for all $n\in \N$. 
With the very same arguments as in the proof of Lemma \ref{l:closure-brittle}, it can be shown that 
$\rstabname$ is lower semicontinuous w.r.t.\ the weak$^*$-topology of   $[0,T] \times \spz$. 
 Thus, we conclude that $\rstab {t}{z(t)} =0$.
\par
From \eqref{stab-at-jumps} we clearly conclude that $(u,z)$ fulfills the stability condition \eqref{stab-VE} at all $t \in [0,T] \setminus \jump z$, which in particular yields
 the minimality property \eqref{MINIMALITY}  at all $t\in [0,T]\setminus \jump z$. All in all, 
\eqref{MINIMALITY} holds at every  $t\in [0,T] \setminus \bar{\rmJ}$ with $\bar{\rmJ} = \tilde\rmJ \cap \jump z$. 
\par\noindent
 $\vartriangleright$ \textbf{Step $3$:} Let us now take the $\liminf$ as  $k\to\infty$ in the (upper) energy-dissipation estimate \eqref{enue-VE}
 for the adhesive contact system.
 We handle the terms on the left-hand side 
 by observing that 
 \[
 \liminf_{j\to\infty} \enetk {k_j}{t}{u_{k_j}(t)}{z_{k_j}(t)} \geq \enet t{u(t)}{z(t)}
 \text{ and } \liminf_{j\to\infty} \Vari {\mdsn{\spz},\vecostname}{z_{k_j}}0{t} \geq \Vari {\mdsn{\spz},\vecostname}{z}0{t}
  \quad \text{for all } t \in [0,T],
 \]
 where the first inequality is
due to  the $\Gamma$-$\liminf$ 
 estimate \eqref{Gamma-adh-bri}, 
 and the second one follows from Proposition \ref{prop:6.2} below. 
As for the right-hand side, we observe that 
\[
\begin{aligned}
\partial_t \enetk {k_j}{t}{u_{k_j}(t)}{z_{k_j}(t)} =   - \pairing{}{H^1}{\dot\ell(t)}{u_{k_j}(t)+w_\Dir(t)} - \pairing{}{H^1}{\ell(t)}{\dot{w}_\Dir(t)}
\to  &  - \pairing{}{H^1}{\dot\ell(t)}{u(t)+w_\Dir(t)} - \pairing{}{H^1}{\ell(t)}{\dot{w}_\Dir(t)} \\ & = \partial_t \enet t{u(t)}{z(t)} 
\end{aligned}
\]
for every $t\in [0,T]$, with $|\partial_t \enetk {k_j}{t}{u_{k_j}(t)}{z_{k_j}(t)}| \leq C $ by \eqref{ass-ell-bri} and the 
previously obtained bound for $(u_{k_j})_j$  in $L^\infty(0,T;H_{\Gdir}^1(\Omega{\setminus}\GC;\R^d))$. 
Then,
\begin{equation}
\label{cvg-powers}
\lim_{j\to\infty}\int_0^t 
\partial_t \enetk {k_j}{s}{u_{k_j}(s)}{z_{k_j}(s)} \dd s  = \int_0^t  \partial_t \enet s{u(s)}{z(s)} \dd s   \quad \text{for all } t \in [0,T],
\end{equation}
and we thus conclude  the upper energy-dissipation estimate \eqref{enue-VE} for the brittle system.
\par\noindent
 $\vartriangleright$ \textbf{Step $4$, conclusion of the proof:} 
 Since we have proved the stability condition \eqref{stab-VE} and the upper
 energy-dissipation estimate  \eqref{enue-VE}, thanks to Proposition \ref{prop:charact} we conclude that $(u,z)$ is a $\VE$ solution of the brittle system. The energy convergence \eqref{energy-cv} ensues from the following standard argument:
 \[
 \begin{aligned}
 \limsup_{j\to\infty} \left( \enetk {k_j}{t}{u_{k_j}(t)}{z_{k_j}(t)}  {+} \Vari {\mdsn{\spz},\vecostname}{z_{k_j}}0{t} \right)
  &  \stackrel{(1)}{\leq} 
 \enet 0{u_0}{z_0} +\lim_{j\to\infty}  \int_0^t 
\partial_t \enetk {k_j}{s}{u_{k_j}(s)}{z_{k_j}(s)} \dd s
\\ &  
 \stackrel{(2)}{=} 
\enet 0{u_0}{z_0} +\int_0^t  \partial_t \enet s{u(s)}{z(s)} \dd s
\\ & 
  \stackrel{(3)}{=}  \enet t{u(t)}{z(t)} + \Vari {\mdsn{\spz},\vecostname}{z}0{t}\,,
  \end{aligned}
 \]
 with {\footnotesize (1)} due to \eqref{enbal-VE} for the adhesive system,  {\footnotesize (2)} due to \eqref{cvg-powers}, and {\footnotesize (3)} following from the  energy balance 
 \eqref{enbal-VE} for the brittle system.
 This finishes the proof of Thm.\     \ref{th:VE-brittle}. 
 \QED 
 \par
With the following result we obtain the key lower semicontinuity estimate for the total variation functionals exploited in Step $3$ of the proof of  Thm.\     \ref{th:VE-brittle}. 
\begin{proposition}
\label{prop:6.2}
   Assume  \eqref{geometry-brittle}, \eqref{WLOG-brittle},  and \eqref{ass-ell-bri}. Let 
   $(z_k)_k, \, z \subset L^\infty (\GC{\times}(0,T)) \cap \BV ([0,T];L^1(\Omega))$
   fulfill
   \begin{subequations}
   \label{conds-4-lsc}
   \begin{align}
      \label{conds-4-lsc-2}
   & z_k(t) \weaksto z(t) \text{ in } L^\infty(\GC)  \text{ for all } t \in [0,T],
   \\
   & 
      \label{conds-4-lsc-3}
   \forall\, t \in \mathrm{J}_z \  \exists\, (\alpha_{k})_k,\, (\beta_{k})_k \subset [0,T] \text{ such that }
\begin{cases}
 \alpha_{k}\uparrow t \text{ and } z_{k}(\alpha_{k}) \weaksto \lli zt \text{ in } L^\infty(\GC),
 \\
 \beta_{k}\downarrow t \text{ and } z_{k}(\beta_{k}) \weaksto \rli zt \text{ in } L^\infty(\GC).
 \end{cases}
\end{align}
   \end{subequations}
   Then,
   \begin{equation}
   \label{ineq:key-lsc}
   \liminf_{k\to\infty}   \Vari {\mdsn{\spz},\vecostname}{z_{k}}a{b} \geq  \Vari {\mdsn{\spz},\vecostname}{z}a{b} \quad \text{for all } [a,b]
   \subset [0,T].
   \end{equation}
\end{proposition}
The \emph{proof} follows the very same lines as the argument for \cite[Thm.\ 4]{RS17}, to which  we shall refer for all details. 
Let us just outline it: up to an extraction we may suppose that $\sup_{k\in \N} \Vari {\mdsn{\spz},\vecostname}{z_{k}}0{T} \leq C$.
Therefore,
the non-negative and bounded Borel measures $\eta_k$ on $[0,T]$ defined by
$\eta_k([a,b]): = \Vari {\mdsn{\spz},\vecostname}{z_{k}}a{b} $ for all $[a,b]\subset[0,T]$ 
weakly$^*$ converge (in the duality with $\rmC^0([0,T])$) to a measure $\eta$. We observe that 
\begin{equation}
\label{2diffuse}
\eta([a,b]) \stackrel{(1)}{\geq} \limsup_{k\to\infty} \eta_k([a,b]) \stackrel{(2)}{\geq}  \limsup_{k\to\infty}
\Vari {\mdsn{\spz}}{z_{k}}a{b}  \stackrel{(3)}{\geq}\Vari {\mdsn{\spz}}{z}a{b},
\end{equation}
with {\footnotesize (1)} due to the upper semicontinuity of the weak$^*$ convergence of measures on closed sets, 
 {\footnotesize (2)} due to the fact that $\Varname{\mdsn{\spz},\vecostname} \geq \Varname{\mdsn{\spz}}  $, and 
 {\footnotesize (3)} due to  \eqref{conds-4-lsc-2}. 
 It follows from Lemma \ref{prop:key-lsc-adh2bri} ahead that, at any $t\in \mathrm{J}_z$ and for all  sequences
 $(\alpha_k)_k,\, (\beta_k)_k \subset [0,T]$
 fulfilling   \eqref{conds-4-lsc-3} there holds
 \begin{equation}
\label{2jump}
 \eta(\{ t\}) \geq  \limsup_{k\to\infty} \eta_k([\alpha_k,\beta_k]) \geq \vecost t{\lli zt}{\rli zt}\,.
 \end{equation}
 Combining \eqref{2diffuse}
 and \eqref{2jump} and arguing in the very same way as in the proof of \cite[Thm.\ 4]{RS17}
 (cf.\ also \cite[Prop.\ 7.3]{MRS13}), we  establish \eqref{ineq:key-lsc}.
 \QED
 We conclude this section by stating a crucial  lower estimate for the 
 Visco-Energetic total variation of a sequence $(z_k)_k$ 
 of solutions to the adhesive contact system (for notational simplicity, we 
 drop the subsequence $(k_j)_j$ and 
 revert to the  original sequence of indexes $(k)$).
  The total variation of the curves $z_k$ is considered on  a sequence of intervals shrinking as
 $k\to\infty$ to a jump point of the limit curve $z$. 
\begin{lemma}
\label{prop:key-lsc-adh2bri}
   Assume  \eqref{geometry-brittle}, \eqref{WLOG-brittle},  and \eqref{ass-ell-bri}. Let 
   $(z_k)_k, \, z \subset L^\infty (\GC{\times}(0,T)) \cap \BV ([0,T];L^1(\Omega))$
   fulfill \eqref{conds-4-lsc}. For any $t\in \mathrm{J}_z$ pick two sequences $(\alpha_k)_k$ and
   $(\beta_k)_k$ converging to $t$ and fulfilling  \eqref{conds-4-lsc}. Then,
   \begin{equation}
   \label{basta-lsc}
  \liminf_{k\to\infty} \Vari {\mdsn{\spz},\vecostname}{z_{k}}{\alpha_k}{\beta_k} \geq \vecost t{\lli zt}{\rli zt}\,.
   \end{equation}
\end{lemma}
The  \emph{proof} will be given in Sec.\ \ref{ss:6.2}. 

\subsection{Proof of Lemma \ref{prop:key-lsc-adh2bri}}
\label{ss:6.2}
 Let us  briefly outline the proof, partially borrowed from that of \cite[Prop.\ 3]{RS17}:
\begin{compactenum}
\item
for every $k\in \N$,  the curve $z_k$ has countably many jump points $(t_m^k)_{m\in M_k}$ between  $\alpha_k$ and $\beta_k$.
 Along the footsteps of \cite{RS17}, we will suitably reparameterize both the continuous pieces of the trajectory $z_k$ and 
  the optimal transitions 
$\ojt zmk$ connecting the left and right limits $\lli {z_k}{t_m^k}$ and  $\rli {z_k}{t_m^k}$ at a jump point $t_m^k$. We will then glue
the (reparameterized) continuous pieces  and  the (reparameterized) jump transitions
together. 
\item
In this way, we shall  obtain a sequence of curves $(\invcur k)_k$, defined on  compact sets $(\mathfrak{C}_k)_k$, 
to which we will apply a refined compactness argument from \cite{SavMin16}, yielding the existence of 
a limiting Lipschitz  curve $\invcu$, defined on a compact set $\mathfrak{C} \Subset \R$, connecting the left and the right limits $\lli zt$ and $\rli zt$. 
\item We will then show that 
\begin{equation}
\label{crucial2show}
  \liminf_{k\to\infty} \Vari {\mdsn{\spz},\vecostname}{z_{k}}{\alpha_k}{\beta_k}  \geq \tcost{\VE}{t}{\invcu}{\mathfrak{C}}\,.
\end{equation}
\item From \eqref{crucial2show} we shall
conclude \eqref{basta-lsc}. 
\end{compactenum}
\par\noindent
$\vartriangleright $  \textbf{Step $1$ (reparameterization):}
We set 
\[
\newmresc k: = \beta_k -\alpha_k +  \Vari{\mdsn{\spz},\vecostname}{z_{k}}{\alpha_k}{\beta_k}+ \sum_{m \in M_k} 2^{-m}\,
\]
and define the  rescaling function $\nresc k : [\alpha_k, \beta_k]\to [0,\newmresc k]$   by 
 \[
 \nresc k(t): 
 = t-\alpha_k +    \Vari{\mdsn{\spz},\vecostname}{z_{k}}{\alpha_k}{t}+ \sum_{\{m \in M_k :\, t_m^k \leq t\}} 2^{-m}\,.
 \]
 Observe that $\nresc k $ is strictly increasing, with jump set $\jump{\nresc k} = (t_m^k)_{m\in M_k}$. We set
 \[
 I_m^k: =  (\nresc k(t_m^k-), \nresc k(t_m^k+)),
  \quad I_k: = \cup_{m\in M_k}  I_m^k, \quad  \Lambda_k: =[ \nresc k (\alpha_k), \nresc k (\beta_k)]. 
 \]
 On   $\Lambda_k\setminus I_k$ 
    the inverse $\ninvresc k: \Lambda_k \setminus I_k \to [\alpha_k, \beta_k]$ 
    of $\nresc k $ 
    is well defined and Lipschitz continuous. We introduce
    \begin{equation}
    \label{invcurk-1}
    \invcur k (s): = (u_k \circ \ninvresc k)(s) \qquad \text{for all } s \in \Lambda_k \setminus I_k
    \end{equation}
    and observe that $\invcur k$ is Lipschitz as well. 
    \par
    We now reparameterize the `jump pieces' of the trajectory. Recall that at every 
    jump point $t_m^k$ there exists an optimal jump transition $\ojt zmk \in \mathrm{C}_{\sigma_\spz,\mdsn{\spz}}(E_m^k;\spz)$,
    fulfilling 
\begin{equation}
\label{tight-OJT}
\begin{aligned}
&
\lli z{t_m^k} = \ojt zmk ((E_m^k)^-), \qquad \rli z{t_m^k} = \ojt zmk ((E_m^k)^+), \qquad z(t_m^k)\in \ojt zmk(E_m^k), \\
&
\enetk k{t_m^k}{\lli u{t_m^k}}{\lli z{t_m^k}} - \enetk k{t_m^k}{\rli u{t_m^k}} {\rli z{t_m^k}}    = \vecost{t_m^k}{\lli z{t_m^k}}{\rli z{t_m^k}}  = \tcost{\VE}{t_m^k}{\ojt zmk}{E_m^k}\,.
 \end{aligned}
\end{equation}
We define the rescaling function $\rresc mk $ on $ E_m^k $ 
 by
\[
\begin{aligned}
 \rresc mk (t): = &  \frac1{2^m} \frac{t-(E_m^k)^-}{(E_m^k)^+ - (E_m^k)^-} +  \Vars {\mdsn{\spz}}{\ojt zmk} {E_m^k \cap [(E_m^k)^-,t]} 
 \\ & 
 + \Gap{\delta_\spz}{\ojt zmk} {E_m^k \cap [(E_m^k)^-,t]}  +\sum_{r\in  [(E_m^k)^-,t]\setminus (E_m^k)^+} 
  \rstabtk k{t_m^k}{\ojt zmk(r)}+\nresc k(t_m^k-) \quad \text{for all $ t \in E_m^k$.}
 \end{aligned}
\]
It can be checked that $\rresc mk$ is continuous and strictly increasing,
with  image a compact set  $S_m^k\subset I_m^k $ such that 
$
(S_m^k)^\pm =  \rresc mk ((E_m^k)^\pm) =\nresc k(t_m^k\pm) $. 
Its inverse function   $\siresc mk :  S_m^k  \to E_m^k$ is Lipschitz continuous.
\par
Finally, we introduce the \emph{compact} set
\[
\mathfrak{C}_k: = (\Lambda_k {\setminus} I_k) \cup (\cup_{m\in M_k} S_m^k) \subset \Lambda_k \subset [0,\newmresc k]
\]
and
 extend  the functions $\ninvresc k$ and 
 $\invcur k $,  so far defined on $\Lambda_k \setminus I_k$, only, to the set $\mathfrak{C}_k$ 
by setting
\[
\ninvresc k(s) \equiv t_m^k \quad\text{and} \quad
\invcur k (s) : = \teta_m^k( \siresc mk(s)) \qquad \text{whenever } s \in S_m^k \text{ for some  } m \in M_k. 
\] 
It has been checked 
in \cite{RS17}
that 
the extended curve 
 $\invcur k $ is in $ \mathrm{C}_{\sigma_\spz,\mdsn{\spz}}(\mathfrak{C}_k;\Xs)  \cup \BV_{\mdsn{\spz}}(\mathfrak{C}_k;\Xs)$, with
 \begin{equation}
 \label{variations-k}
 \begin{aligned}
 &
\Vari{\mdsn{\spz}}{\invcur k}{s_0}{s_1} \leq 
\Vari{\mdsn{\spz}}{z_k}{\ninvresc k(s_0)}{\ninvresc k(s_1)} +(\ninvresc k(s_1){-}\ninvresc k(s_0))  \quad \text{for all } s_0,\, s_1 \in \Lambda_k \setminus I_k \text{ with } s_0<s_1,
  \\
  & 
\Vars {\mdsn{\spz}}{\invcur k} {S_m^k}=
 \Vars {\mdsn{\spz}}{\ojt zmk} {E_m^k} , \qquad  \Gap{\delta_\spz}{\invcur k} {S_m^k} =\Gap{\delta_\spz}{\ojt zmk} {E_m^k}, 
 \\
 &
 \sum_{s\in S_m^k{\setminus}\{(S_m^k)^+\}}  \rstabtk k {t_m^k}{\invcur k(s)} =  \sum_{r\in E_m^k{\setminus}\{(E_m^k)^+\}} 
 \rstabtk k{t_m^k}{\ojt zmk(r)}.
\end{aligned}
 \end{equation}
\par\noindent
$\vartriangleright $ {\bf Step $2$  (a priori estimates and compactness):}
 We refer to the proof of \cite[Prop.\ 3]{RS17} for the calculations leading to these a priori estimates:
 \begin{equation}
 \label{a-prio-ests}
 \exists\, \overline{C}>0 \ \forall\, k \in \N \, : \qquad 
 \begin{cases}
 \mathfrak{C}_k^+\leq \overline{C},
\\
 \Vars{\mdsn \spz}{\invcur k}{\mathfrak{C}_k} \leq \overline{C},
 \\
 \Vars{\mdsn\spz}{\invcur k}{\mathfrak{C}_k \cap [s_0,s_1]} \leq (s_1{-}s_0) \quad \text{for all } s_0,s_1 \in \mathfrak{C}_k \text{ with } s_0<s_1,
 \\
 \sup_{s\in \mathfrak{C}_k} \pertok k  {\mathfrak{u}_k (s)}{\invcur k(s)}\leq \overline{C},
 \end{cases}
 \end{equation}
 where $\mathfrak{u}_k(s) $  is the unique element in  $ \mathrm{Argmin}_{u\in \spu} \enetk k{s}{u}{\invcur k(s)}$ and $\pertokname k$ the perturbed functional associated with 
 $\calE_k$ via 
 \eqref{pert-func}.
 \par
 Therefore, we are in a position to apply the compactness result from 
 \cite[Thm.\ 5.4]{SavMin16} and conclude that 
  there exist a (not relabeled) subsequence,  a compact set   $\mathfrak{C}\subset [0, \overline{C}]$ with $\overline C$ as in 
  \eqref{a-prio-ests},
  and a function $\invcu \in \mathrm{C}_{\sigma_\spz, \mdsn{\spz}}(\mathfrak{C};\Xs)$ such that, as $k\to\infty$, there hold
\begin{compactenum}
\item $\mathfrak{C}_k \to \mathfrak{C}$ \`a la Kuratowski, namely
\[
\begin{aligned}
&
\mathrm{Li}_{k\to\infty} \mathfrak{C}_k = 
\mathrm{Ls}_{k\to\infty} \mathfrak{C}_k  =\mathfrak{C}
\quad \text{
with }
\\
& \quad
\mathrm{Li}_{k\to\infty} \mathfrak{C}_k:= \{ t \in [0,\infty)\, : \exists\, t_k \in \mathfrak{C}_k  \text{ s.t. } t_k\to t\},
\\
& \quad
\mathrm{Ls}_{k\to\infty} \mathfrak{C}_k := \{ t \in [0,\infty)\, : \exists\,
j\mapsto k_j \text{ increasing and } t_{k_j} \in \mathfrak{C}_{k_j} 
  \text{ s.t. } t_{k_j}\to t\};
  \end{aligned}
\]
\item 
for every $s\in \mathfrak{C}$ there exists a sequence $(s_k)_k$, with $s_k\in \mathfrak{C}_k$ for all $k\in \N$, such that 
$s_k\to s$ and $\invcur k(s_k) \wsigmaz \invcu (s)$ in $\spz$ as $k\to\infty$; 
\item whenever $s_{k} \in  \mathfrak{C}_k$ converge to $s\in \mathfrak{C}$, then $\invcur {k}(s_{k}) \wsigmaz  \invcu(s)$ in $\spz$;
\item 
$\invcur {k}((\mathfrak{C}_k)^\pm )\wsigmaz \invcu(\mathfrak{C}^\pm)$;
\item for every $I \in \hole {\mathfrak{C}}$ (recall \eqref{holes}) there exists a sequence $(J_k)_k$ with 
\begin{equation}
\label{approximation-holes}
J_k \in \hole{\mathfrak{C}_k} \text{ for all } k \in \N \text{ and } J_k^+\to I^+, \ J_k^-\to I^-\,.
\end{equation}
\end{compactenum}
Therefore, $\invcu(C^-) =  \lli z t$, and $\invcu(C^+) = \rli z t$. 
Finally, for later use    we observe that
\begin{equation}
\label{singleton}
\lim_{k\to\infty}  \sup_{s\in \mathfrak{C}_k} | \ninvresc k(s)- t | =0, 
\end{equation} 
since 
  the functions $\ninvresc k$ take values in the intervals $[\alpha_k,\beta_k]$ shrinking to the singleton $\{t\}$.
\par\noindent
$\vartriangleright $
{\bf Step $3$ (proof of  \eqref{crucial2show}):}
Repeating the very same arguments as in the proof of  \cite[Thm.\ 5.3]{SavMin16}, 
from the above convergence properties 
we conclude 
\begin{equation}
\label{liminf-total-variations}
\Vars{\mdsn{\spz}}{\invcu}{\mathfrak{C}} \leq \liminf_{k\to\infty}\Vars{\mdsn{\spz}}{\invcur k}{\mathfrak{C}_k}\,.
\end{equation}
\par
Let us now address the term in the transition cost  involving the residual stability function.
To this end, we fix a finite set $\{ \mathfrak{C}^-= :\sigma^0<\sigma^1<\ldots <\sigma^N : = \mathfrak{C}^+\} \subset \mathfrak{C}$ such that 
$\rstabt {\sigma^n}{\invcu(\sigma^n)}>0$ for all $n=1,\ldots, N-1$.
We use that for every $n\in \{1,\ldots, N\}$ there exists a sequence $(\sigma_k^n)_k$ with 
$\sigma_k^n \in \mathfrak{C}_k$ for all $k\in \N$, $\sigma_k^n\to \sigma^n$ and $\invcur k(\sigma_k^n) \wsigmaz \invcu(\sigma^n)$
as $k\to\infty$. Furthermore,  in view of \eqref{singleton} we have that $ \ninvresc k(\sigma_k^n) \to t$ as $k\to\infty$ for all $n\in \{0,\ldots,N\}$. 
By the 
$\Gamma$-$\liminf$ estimate 
\eqref{liminf-Rk}, we infer that 
\[
\liminf_{k\to\infty} \rstabtk k{ \ninvresc k(\sigma_k^n)}{\invcur k(\sigma_k^n)} \geq \rstabt  {t}{\invcu(\sigma^n)} \quad \text{for all } n \in \{0,\ldots,N\},
\]
therefore there exist $c>0$ and an index $\bar k \in \N$ such that 
\[
 \rstabtk k{ \ninvresc k(\sigma_k^n)}{\invcur k(\sigma_k^n)}\geq c>0 \quad  \text{for all } n \in \{1,\ldots,N-1\}\,.
\]
This entails that 
for every
$n \in \{1,\ldots,N-1\}$ and 
 $k \geq \bar k$ there exists $m_k^n \in M_k$ (the countable set of jump points of $z_k$ between $\alpha_k$ and $\beta_k$) such that 
 $ \ninvresc k(\sigma_k^n) = t_{m_k^n}$. 
 All in all, we conclude that 
 \[
 \begin{aligned}
 \sum_{n=1}^{N-1}  \rstabt  {t}{\invcu(\sigma^n)} \leq   \sum_{n=1}^{N-1}  \liminf_{k\to\infty} 
  \rstabtk k{ t_{m_k^n}}{\invcur k(\sigma_k^n)}  & \leq \liminf_{k\to\infty}   \sum_{n=1}^{N-1}
  \rstabtk k{ t_{m_k^n}}{\invcur k(\sigma_k^n)} \\ &  \leq  \liminf_{k\to\infty}  \sum_{m\in M_k}    \sum_{s\in S_m^k{\setminus}\{(S_m^k)^+\}}  \rstabtk k {t_m^k}{\invcur k(s)}  \\ & =  \liminf_{k\to\infty}  \sum_{m\in M_k}  \sum_{r\in E_m^k{\setminus}\{(E_m^k)^+\}} 
 \rstabtk k{t_m^k}{\ojt zmk(r)},
 \end{aligned}
 \]
the latter identity due to \eqref{variations-k}. Taking the supremum of the left-hand side
over all finite subsets of $\mathfrak{C} \setminus \{\mathfrak{C}^+\}$, we then conclude that 
\begin{equation}
\label{liminf-Rk-variations}
\sum_{\sigma \in \mathfrak{C} {\setminus}  \{ \mathfrak{C}^+ \}}   \rstabt  {t}{\invcu(\sigma)} \leq  \liminf_{k\to\infty}  \sum_{m\in M_k}  \sum_{r\in E_m^k{\setminus}\{(E_m^k)^+\}} 
 \rstabtk k{t_m^k}{\ojt zmk(r)}\,.
\end{equation}
\par
Finally, 
\eqref{approximation-holes} and,
 again, the very same arguments as in the proof of  \cite[Thm.\ 5.3]{SavMin16} yield that 
\begin{equation}
\label{liminf-GapVar}
\begin{aligned}
\Gap{\mdsn\spz}\invcu C= \sum_{I\in \hole{\mathfrak{C}}} \delta_\spz(\invcu(I^-),\invcu(I^+)) 
 & \leq \liminf_{k\to\infty} \sum_{J\in \hole{\mathfrak{C}_k}} \delta_\spz(\invcur k(J^-),\invcur k(J^+)) 
 \\ & =  \liminf_{k\to\infty} \sum_{m\in M_k}  \Gap{\delta_\spz}{\invcur k} {S_m^k} \\ & \stackrel{(1)}{=}
  \liminf_{k\to\infty} \sum_{m\in M_k}  \Gap{\delta_\spz}{\ojt zmk} {E_m^k}
  \end{aligned}
\end{equation}
with {\footnotesize (1)} due to \eqref{variations-k}. 
Combining \eqref{liminf-total-variations}, 
\eqref{liminf-Rk-variations} and \eqref{liminf-GapVar}, 
we  deduce  \eqref{crucial2show}.
\par\noindent
$\vartriangleright $ {\bf Step $4$ (conclusion):} 
Observe that 
\[
\vecost {t}{\lli zt}{\rli zt} \leq  \tcost{\VE}{t}{\invcu}{\mathfrak{C}}\,.
\]
Therefore, \eqref{basta-lsc} follows from \eqref{crucial2show}. This finishes the proof of Lemma \ref{prop:key-lsc-adh2bri}.  \QED

\bibliographystyle{alpha}
\bibliography{ricky_lit}

\begin{thebibliography}{DMDM06}

\bibitem[DMDM06]{DMDSMo06QEPL}
G.~Dal~Maso, A.~DeSimone, and M.G. Mora.
\newblock Quasistatic evolution problems for linearly elastic-perfectly plastic
  materials.
\newblock {\em Arch.\ Rational Mech.\ Anal.}, 180:237--291, 2006.

\bibitem[DMT02]{DM-Toa2002}
G.~Dal~Maso and R.~Toader.
\newblock A model for the quasi-static growth of brittle fractures: existence
  and approximation results.
\newblock {\em Arch. Ration. Mech. Anal.}, 162(2):101--135, 2002.

\bibitem[EM06]{EfeMie06RILS}
M.~Efendiev and A.~Mielke.
\newblock On the rate--independent limit of systems with dry friction and small
  viscosity.
\newblock {\em J. Convex Analysis}, 13(1):151--167, 2006.

\bibitem[FM06]{FraMie06ERCR}
G.~Francfort and A.~Mielke.
\newblock Existence results for a class of rate-independent material models
  with nonconvex elastic energies.
\newblock {\em J.\ reine angew.\ Math.}, 595:55--91, 2006.

\bibitem[Fr{\'e}02]{Fre02}
M.~Fr{\'e}mond.
\newblock {\em Non-Smooth Thermomechanics}.
\newblock Springer-Verlag Berlin Heidelberg, 2002.

\bibitem[HMM03]{HMM03}
K.~Hackl, A.~Mielke, and D.~Mittenhuber.
\newblock Dissipation distances in multiplicative elastoplasticity.
\newblock 2003.
\newblock In: Wendland, W., Efendiev, M. (eds.) Analysis and Simulation of
  Multifield Problems, pp.\ 87--100. Springer, New York.

\bibitem[HN75]{HalNgu75MSG}
B.~Halphen and Q.S. Nguyen.
\newblock Sur les mat\'eriaux standards g\'en\'eralis\'es.
\newblock {\em J. M\'ecanique}, 14:39--63, 1975.

\bibitem[KRZ13]{KnRoZa13VVAR}
D.~Knees, R.~Rossi, and C.~Zanini.
\newblock A vanishing viscosity approach to a rate-independent damage model.
\newblock {\em Math. Models Methods Appl. Sci.}, 23:565--616, 2013.

\bibitem[KRZ18]{KnRoZa17}
D.~Knees, R.~Rossi, and C.~Zanini.
\newblock Balanced viscosity solutions to a rate-independent system for damage.
\newblock {\em European J. Appl. Math.}, 2018.
\newblock doi:10.1017/S0956792517000407.

\bibitem[Mie02]{Mie02}
A.~Mielke.
\newblock Finite elastoplasticity {L}ie groups and geodesics on {${\rm
  SL}(d)$}.
\newblock In {\em Geometry, mechanics, and dynamics}, pages 61--90. Springer,
  New York, 2002.

\bibitem[Mie11]{Mielke-Cetraro}
A.~Mielke.
\newblock Differential, energetic, and metric formulations for rate-independent
  processes.
\newblock In {\em Nonlinear {PDE}'s and applications}, volume 2028 of {\em
  Lecture Notes in Math.}, pages 87--170. Springer, Heidelberg, 2011.

\bibitem[Mie16]{Mielke-evol}
A.~Mielke.
\newblock On evolutionary {$\varGamma$}-convergence for gradient systems.
\newblock In {\em Macroscopic and large scale phenomena: coarse graining, mean
  field limits and ergodicity}, volume~3 of {\em Lect. Notes Appl. Math.
  Mech.}, pages 187--249. Springer, [Cham], 2016.

\bibitem[Min17]{Minotti17}
L.~Minotti.
\newblock Visco-energetic solutions to one-dimensional rate-independent
  problems.
\newblock {\em Discrete Contin. Dyn. Syst.}, 37(11):5883--5912, 2017.

\bibitem[MM05]{MaiMie05EREM}
A.~Mainik and A.~Mielke.
\newblock Existence results for energetic models for rate-independent systems.
\newblock {\em Calc. Var. Partial Differential Equations}, 22:73--99, 2005.

\bibitem[MM09]{MaiMi08}
A.~Mainik and A.~Mielke.
\newblock Global existence for rate-independent gradient plasticity at finite
  strain.
\newblock {\em J. Nonlinear Sci.}, 19(3):221--248, 2009.

\bibitem[MR06]{MiRou06}
A.~Mielke and T.~Roub{\'i}{\v c}ek.
\newblock Rate-independent damage processes in nonlinear elasticity.
\newblock {\em M$^3$AS Math.\ Models Methods Appl. Sci.}, 16:177--209, 2006.

\bibitem[MR15]{MieRouBOOK}
A.~Mielke and T.~Roub{\'{\i}}{\v{c}}ek.
\newblock {\em Rate-independent systems. Theory and application}, volume 193 of
  {\em Applied Mathematical Sciences}.
\newblock Springer, New York, 2015.

\bibitem[MRS08]{MRS06}
A.~Mielke, T.~Roub{\'i}{\v c}ek, and U.~Stefanelli.
\newblock {$\Gamma$}-limits and relaxations for rate-independent evolutionary
  problems.
\newblock {\em Calc. Var. Partial Differential Equations}, 31:387--416, 2008.

\bibitem[MRS12]{MRS12}
A.~Mielke, R.~Rossi, and G.~Savar{\'e}.
\newblock B{V} solutions and viscosity approximations of rate-independent
  systems.
\newblock {\em ESAIM Control Optim. Calc. Var.}, 18(1):36--80, 2012.

\bibitem[MRS16]{MRS13}
A.~Mielke, R.~Rossi, and G.~Savar{\'e}.
\newblock Balanced viscosity ({BV}) solutions to infinite-dimensional
  rate-independent systems.
\newblock {\em J. Eur. Math. Soc. (JEMS)}, 18(9):2107--2165, 2016.

\bibitem[MRS18]{MRS2018}
A.~Mielke, R.~Rossi, and G.~Savar\'e.
\newblock Global {E}xistence {R}esults for {V}iscoplasticity at {F}inite
  {S}train.
\newblock {\em Arch. Ration. Mech. Anal.}, 227(1):423--475, 2018.

\bibitem[MS18]{SavMin16}
L.~Minotti and G.~Savar\'e.
\newblock Viscous {C}orrections of the {T}ime {I}ncremental {M}inimization
  {S}cheme and {V}isco-{E}nergetic {S}olutions to {R}ate-{I}ndependent
  {E}volution {P}roblems.
\newblock {\em Arch. Ration. Mech. Anal.}, 227(2):477--543, 2018.

\bibitem[MT99]{MieThe99MMRI}
A.~Mielke and F.~Theil.
\newblock A mathematical model for rate-independent phase transformations with
  hysteresis.
\newblock In H.-D. Alber, R.M. Balean, and R.~Farwig, editors, {\em Proceedings
  of the Workshop on ``Models of Continuum Mechanics in Analysis and
  Engineering''}, pages 117--129, Aachen, 1999. Shaker-Verlag.

\bibitem[MT04]{MieThe04RIHM}
A.~Mielke and F.~Theil.
\newblock On rate-independent hysteresis models.
\newblock {\em NoDEA Nonlinear Differential Equations Appl.}, 11(2):151--189,
  2004.

\bibitem[Neg17]{Negri16}
M.~Negri.
\newblock An {$L^2$} gradient flow and its quasi-static limit in phase-field
  fracture by alternate minimization.
\newblock {\em Adv. Calc. Var.}, 2017.
\newblock doi: 10.1515/acv-2016-0028, in press.

\bibitem[RS13]{RosSav12}
R.~Rossi and G.~Savar{\'e}.
\newblock A characterization of energetic and {BV} solutions to one-dimensional
  rate-independent systems.
\newblock {\em Discrete Contin. Dyn. Syst. Ser. S}, 6(1):167--191, 2013.

\bibitem[RS17]{RS17}
R.~Rossi and G.~Savar\'e.
\newblock From {V}isco-{E}nergetic to {E}nergetic and {B}alanced {V}iscosity
  solutions of rate-independent systems.
\newblock In {\em Solvability, Regularity, and Optimal Control of Boundary
  Value Problems for PDEs.}, pages 489--531. Springer INdAM Series, vol 22.
  Springer, Cham, 2017.
\newblock Colli P., Favini A., Rocca E., Schimperna G., Sprekels J. (eds).

\bibitem[RSZ09]{RoScZa09QDP}
T.~Roub{\'i}{\v c}ek, L.~Scardia, and C.~Zanini.
\newblock Quasistatic delamination problem.
\newblock {\em Continuum Mech. Thermodynam.}, 21(3):223--235, 2009.

\bibitem[Tem83]{Temam83}
R.~Temam.
\newblock {\em Probl\`emes math\'ematiques en plasticit\'e}, volume~12 of {\em
  M\'ethodes Math\'ematiques de l'Informatique [Mathematical Methods of
  Information Science]}.
\newblock Gauthier-Villars, Montrouge, 1983.

\bibitem[Tho13]{Thom11QEBV}
M.~Thomas.
\newblock Quasistatic damage evolution with spatial {BV}-regularization.
\newblock {\em Discrete Contin.\ Dyn.\ Syst.\ Ser.\ S}, 6(1):235--255, 2013.

\bibitem[TM10]{ThoMie09DNEM}
M.~Thomas and A.~Mielke.
\newblock Damage of nonlinearly elastic materials at small strain: existence
  and regularity results.
\newblock {\em Zeit. Angew. Math. Mech.}, 90(2):88--112, 2010.

\bibitem[TS80]{Temam-Strang80}
R.~Temam and G.~Strang.
\newblock Duality and relaxation in the variational problems of plasticity.
\newblock {\em J. M\'ecanique}, 19:493--527, 1980.

\end{thebibliography}
\end{document}